\documentclass{amsart}

\usepackage{amsfonts,amsmath,amsxtra,amsthm,tikz}
\usepackage{mathrsfs}
\usepackage{times}
\usepackage{anysize}
\usepackage{graphicx}
\usepackage{youngtab}
\usepackage[latin1]{inputenc}
\usepackage{color}
\usepackage{subfigure}
\usepackage{stmaryrd, MnSymbol}


\marginsize{1in}{1in}{1in}{1in}

\newtheorem{lemma}{Lemma}[section]
\newtheorem{theorem}[lemma]{Theorem}
\newtheorem{conjecture}[lemma]{Conjecture}
\newtheorem{corollary}[lemma]{Corollary}
\newtheorem{proposition}[lemma]{Proposition}
\theoremstyle{definition}   
\newtheorem{example}[lemma]{Example}
\newtheorem{remark}[lemma]{Remark}
\newtheorem{definition}[lemma]{Definition}
 
\DeclareMathOperator{\Hom}{H}
\newcommand{\BZ}{\mathbb{Z}}
\newcommand{\Hh}{\mathbf{H}}
\newcommand{\CO}{\mathcal{O}}
\newcommand{\CG}{\mathcal{G}}
\newcommand{\CF}{\mathcal{F}}

\DeclareMathOperator{\Fl}{Fl}

\DeclareMathOperator{\Ord}{Ord}
\DeclareMathOperator{\ish}{ish}
\DeclareMathOperator{\inv}{inv}

\DeclareMathOperator{\bounce}{bounce}
\DeclareMathOperator{\dinv}{dinv}
\DeclareMathOperator{\area}{area}
\DeclareMathOperator{\spann}{span}
\DeclareMathOperator{\modd}{mod}
\DeclareMathOperator{\An}{\mathcal{A}} 
\DeclareMathOperator{\PS}{\mathcal{PS}} 
\DeclareMathOperator{\SP}{\mathcal{SP}} 
\DeclareMathOperator{\PF}{\mathcal{PF}} 

\newcommand{\ZZ}{\mathbb{Z}_{\ge 0}}

\newcommand{\Smn}{S_n^m} 
\newcommand{\Sn}{S_n}
\newcommand{\Snn}[1]{S_{#1}}
\newcommand{\aSmn}{\widetilde{S}_n^m} 
\newcommand{\maSn}{{}^m\widetilde{S}_n} 
\newcommand{\aSmnn}[2]{\widetilde{S}_{#1}^{#2}} 
\newcommand{\maSnn}[2]{{}^{#2}\widetilde{S}_{#1}} 
\newcommand{\Smnn}[2]{{S}_{#1}^{#2}} 
\newcommand{\aSn}{\widetilde{S}_n}
\newcommand{\Pmn}{\PF_{m/n}} 
\newcommand{\minl}{\Omega} 
\newcommand{\minlnm}{\minl_n^m} 
\newcommand{\Ymn}{Y_{m,n}} 
\newcommand{\lYmn}{\widehat{Y}_{m,n}} 
\newcommand{\aBrn}{\widetilde{B}_n} 
\DeclareMathOperator{\Sh}{Sh} 
\newcommand{\Shikn}{\Sh^k_n} 
\DeclareMathOperator{\Reg}{Reg} 
\newcommand{\Regkn}{\Reg^k_n} 

\newcommand{\aSnn}[1]{\widetilde{S}_{#1}}
  
\DeclareMathOperator{\Inv}{{Inv}}
\DeclareMathOperator{\Invall}{{\overline{Inv}}}
\DeclareMathOperator{\invsub}{{Mod_{m,n}}}
\DeclareMathOperator{\invsubn}{{Mod_{n+1,n}}}

\def\l{{\ell}}
\newcommand{\w}{{\omega}}
\newcommand{\wi}{{\w^{-1}}}
\newcommand{\Afund}{{\rm A}_0}  
\newcommand{\A}{{\rm A}} 
\DeclareMathOperator{\alc}{alc} 
\newcommand{\balc}{\widehat \alc} 
\newcommand{\C}{{\mathbb C}}
\newcommand{\R}{{\mathbb R}}
\newcommand{\N}{{\mathbb N}}
\newcommand{\Z}{{\mathbb Z}}

\newcommand{\x}{{\bf\overline{x}}}
\def\shuffle{\,\raise 1pt\hbox{$\scriptscriptstyle\cup{\mskip -4mu}\cup$}\,}
\newcommand{\pl}{\llparenthesis} 
\newcommand{\pr}{\rrparenthesis} 
\newcommand{\pw}[1]{\pl #1 \pr} 
\newcommand{\Dw}{\Delta_{\w}}  
\newcommand{\id}{\mathrm{id}}
\newcommand{\wm}{\w_m} 

\newcommand{\omitt}[1]{}  

\definecolor{Red}{rgb}{0.6,0,0.1}

\author{Eugene Gorsky} 
\address{Department of Mathematics, Columbia University.
2990 Broadway, New York, NY 10027 }
\email{egorsky@math.columbia.edu}

\author{Mikhail Mazin} 
\address{Mathematics Department, Kansas State University.
Cardwell Hall, Manhattan, KS 66506}
\email{mmazin@math.ksu.edu}

\author{Monica Vazirani}
\address{Department of Mathematics , UC Davis .
One Shields Ave,
Davis, CA 95616-8633}
\email{vazirani@math.ucdavis.edu}

\title{Affine permutations and rational slope parking functions}
\date{\today}

\begin{document}
\begin{abstract}
We introduce a new approach to the enumeration of rational slope parking functions with respect to the $\area$ and a generalized $\dinv$ statistics, and relate
the combinatorics of parking functions to that of affine permutations.
We relate our construction to two previously known combinatorial constructions: Haglund's bijection $\zeta$ exchanging the
pairs of statistics $(\area,\dinv)$ and $(\bounce, \area)$ on Dyck paths, and the
Pak-Stanley labeling of the regions of $k$-Shi hyperplane arrangements by $k$-parking functions. Essentially, our approach can be viewed as a generalization and a unification of these two constructions. We also relate our combinatorial constructions to representation theory. We derive new formulas for the  Poincar\' e polynomials of certain affine Springer fibers and describe a connection to the theory of finite dimensional representations of DAHA and nonsymmetric Macdonald polynomials. 
\end{abstract}

\maketitle

\section{Introduction}

Parking functions
 are ubiquitous in the modern combinatorics.
 There is a natural action of the symmetric group on parking functions,
and the orbits are labeled by the non-decreasing parking functions which correspond
naturally to the Dyck paths.
This provides a link between parking functions and various combinatorial objects counted
by Catalan numbers.
In a series of papers  Garsia,  Haglund,  Haiman, et al. \cite{Hd08,HHLRU},
 related the combinatorics of Catalan numbers and parking functions to
the space of diagonal harmonics.
There are also deep connections to the geometry of the Hilbert scheme.

Since the works of Pak and Stanley \cite{St96}, Athanasiadis and Linusson \cite{AL99} , it became clear that parking functions are tightly related to the combinatorics of the affine symmetric group. In particular, they provided two different bijections between the parking functions and the regions of Shi hyperplane arrangement. It has been remarked in \cite{Ar11,FV1,So} that the inverses of the affine permutations labeling the minimal alcoves in Shi regions belong to a certain simplex $D_{n}^{n+1}$, which is isometric to the $(n+1)$-dilated fundamental alcove. As a result, the alcoves in $D_{n}^{n+1}$ can be labeled by parking functions in two different ways.

In this paper we develop a ``rational slope'' generalization of this correspondence.
A function $f:\{1,\ldots, n\}\to \BZ_{\ge 0}$ is called an $m/n$-parking function if the Young diagram with row lengths equal to $f(1),\dots, f(n)$ put in the decreasing order, fits under the diagonal in an $n\times m$ rectangle.  
 
Recall that a bijection $\w:\BZ\to\BZ$ is called an affine permutation if $\w(x+n)=\w(x)+n$ for all $x$ and $\sum\limits_{i=1}^n \w(i)=\frac{n(n+1)}{2}$. Given a positive integer $m$, we call an affine permutation {\em $m$-stable}, if the inequality $\w(x+m)> \w(x)$
holds
 for all $x$. All constructions in the present paper are based on the following basic observation (see Section \ref{Subsection: Sommers region} for details).

\begin{proposition}
\label{intro sommers}
If $m$ and $n$ are coprime then $m$-stable affine permutations label the alcoves in a certain simplex $D_n^m$ which is isometric to the $m$-dilated fundamental alcove. In particular, the number of $m$-stable affine permutations equals $m^{n-1}$.
\end{proposition}

The simplex $D_n^{m}$  (first defined in \cite{fan,So}) plays the central role in our study. We show that the alcoves in it naturally label various algebraic and geometric objects such as cells in certain affine Springer fibres and nonsymmetric Macdonald polynomials at $q^m=t^n$.
We provide a clear combinatorial dictionary that allows one to pass from one description to another.

We define two maps $\An, \PS$ between the $m$-stable affine permutations and $m/n$-parking functions and prove the following results about them.

\begin{theorem} Maps $\An$ and $\PS$ satisfy the following properties:
\begin{enumerate}
\item The map $\An $ is a bijection for all $m$ and $n$. 
\item The map $\PS$ is a bijection for $m=kn\pm 1$. For $m=kn+1$, it is equivalent to the Pak-Stanley labeling of Shi regions.
\item The map $\PS\circ \An^{-1}$ generalizes the bijection $\zeta$ constructed by Haglund in \cite{Hd08}. More concretely, if one takes $m=n+1$ and restricts the maps $\An$ and $\PS$ to minimal length right coset
representatives of $\Sn\backslash\aSn,$ then $\PS\circ \An^{-1}$ specializes to Haglund's $\zeta.$ 
\end{enumerate}
\end{theorem}

\begin{remark}
For $m=n+1$ the bijection $\An$ is similar to the Athanasiadis-Linusson \cite{AL99}
labeling of Shi regions, but actually differs from it.
\end{remark}

\begin{conjecture}
The map $\PS$ is bijective for all relatively prime $m$ and $n$.
\end{conjecture}
 
 The map $\PS$ has an important geometric meaning. In \cite{LS91} Lusztig and Smelt considered a certain  Springer fibre  $\CF_{m/n}$ in the affine flag variety and proved that it can be paved by $m^{n-1}$ affine cells. In \cite{GM11,GM12} a related subvariety of the affine Grassmannian 
 has been studied under the name of Jacobi factor, and a bijection between its cells and the Dyck paths in $m\times n$ rectangle has been constructed. In \cite{H12} Hikita generalized this combinatorial analysis and constructed a bijection between the cells in the affine Springer fiber and $m/n$-parking functions (in slightly different terminology). He gave a quite involved combinatorial formula for the dimension of a cell. We reformulate his result in terms of the map $\PS$.
 
\begin{theorem}
The affine Springer fiber $\CF_{m/n}$ admits a paving by affine cells $\Sigma_{\omega}$ naturally labeled by the $m$-stable affine permutations $\omega$. The dimension of such a cell equals 
$$\dim \Sigma_{\omega}=\sum_{i=1}^{n} \PS_{\omega}(i).$$
\end{theorem}

\begin{corollary}
If the map $\PS$ is a bijection (in particular, if $m=kn\pm 1$), then the Poincar\'e polynomial of $\CF_{m/n}$ is given by the following simple formula:
$$
\sum_{k=0}^{\infty}t^{k}\dim \Hom^{k}\left(\CF_{m/n}\right)=\sum_{f\in \PF_{m/n}}t^{2\sum_{i}f(i)}.
$$
\end{corollary}

It had been proven by Varagnolo, Vasserot and Yun  \cite{vv,yun} that the cohomology of affine Springer fibers $\CF_{m/n}$ carry the action of double affine Hecke algebra (DAHA). In fact, all finite-dimensional DAHA representations can be constructed this way. 
On the other hand, Cherednik, the third named author and Suzuki \cite{chered-fourier,SuVa}
gave a combinatorial description  of DAHA representations in terms of periodic standard Young tableaux and nonsymmetric Macdonald polynomials.

\begin{theorem}
There is a basis (of nonsymmetric Macdonald polynomials) in the finite-dimensional DAHA representation naturally labeled by the alcoves of the $m$-dilated fundamental simplex. By Proposition \ref{intro sommers}, these alcoves can be identified with the $m$-stable permutations $\omega$. The weight of such a nonsymmetric Macdonald polynomial can be explicitly computed in terms of the parking function $\An_\omega$. 
\end{theorem}

The maps $\An$ and $\PS$ can be used to define two statistics on $m$-stable permutations (or, equivalently, on $m/n$-parking functions):
$$
\area(\omega):=\frac{(m-1)(n-1)}{2}-\sum \An_{\omega}(i),\ \dinv(\omega):=\frac{(m-1)(n-1)}{2}-\sum \PS_{\omega}(i).
$$
For the case $m=n+1$ Armstrong showed in \cite{Ar11} (in slightly different terms) that $\area$ and $\dinv$ statistic agrees with the statistics defined in \cite{HHLRU} as a part of ``Shuffle Conjecture''. 

\begin{conjecture}
The combinatorial Hilbert series 
$$H_{m/n}(q,t):=\sum_{\omega}q^{\area(\omega)}t^{\dinv(\omega)}$$ 
is symmetric in $q$ and $t$ for all $m$ and $n$: 
$$
H_{m/n}(q,t)=H_{m/n}(t,q).
$$
\end{conjecture}

To support this conjecture, let us remark that the ``weak symmetry''  $H_{m/n}(q,1)=H_{m/n}(1,q)$ would follow from the bijectivity of the map $\PS.$ Indeed,
$$
H_{m/n}(q,1)=\sum_{\omega}q^{\frac{(m-1)(n-1)}{2}-\sum \An_{\omega}(i)}=\sum_{f\in \PF_{m/n}}q^{\frac{(m-1)(n-1)}{2}-\sum f(i)}=\sum_{\omega}q^{\frac{(m-1)(n-1)}{2}-\sum \PS_{\omega}(i)}=H_{m/n}(1,q).
$$ 
The second equation follows from the bijectivity of the map $\An$, and the third one follows from the bijectivity of the map $\PS$. In particular, the ``weak symmetry'' holds for $m=kn\pm 1.$

Surprisingly enough, we found a version of the map $\PS$ for the finite symmetric group $S_n$. A permutation $\omega\in S_n$ is called $m$-stable, if $\omega(i+m)>\omega(i)$ for all $i\le n-m.$ It is easy to see that the number of $m$-stable permutations is given by a certain multinomial coefficient. We define $\PS_\w(\w(i))$ as the number of inversions of height at most $m$ in $\omega$, containing $i$ as the right end. 

\begin{theorem}
The restriction of the map $\PS$ to the finite symmetric group $\Sn$
is injective for all $m$ and $n$.
\end{theorem}

 For example, for $m=2$ the map $\PS$ provides a bijection between the set of $2$-stable permutations in $S_n$ and the set of length $n$ Dyck paths with free right end. We also discuss a relation of this finite version of our construction to the theory of Springer fibers.
 
The rest of the paper is organized as follows. In Section \ref{Section: Tools and Definitions} we introduce and review the main ingredients of our construction: rational slope parking functions, affine permutations, and Sommers regions. In Section \ref{Section: Main Constructions} we construct the maps $\An$ and $\PS$ from the set of $m$-stable affine permutations to the set of rational slope parking functions and prove that $\An$ is a bijection. We also discuss the statistics arising from our construction and introduce the combinatorial Hilbert polynomial. In Section \ref{Section: Pak-Stanley} we study the case $m=kn\pm 1$ and its relation to the theory of extended Shi arrangements and Pak-Stanley labeling. In Section \ref{Section: zeta} we discuss the specializations of the maps $\An$ and $\PS$ to minimal length
coset representatives and their relation to Haglund's bijection $\zeta.$ In Section \ref{Section: DAHA} we relate our construction to the theory of finite dimensional representations of Cherednik's DAHA and nonsymmetric Macdonald polynomials. In Section \ref{Section: finite} we discuss a version of the map $\PS$ for the finite symmetric group and prove its injectivity. In Section \ref{Section: Springer} we discuss how our constructions are related to the theory of Springer fibers. Finally, we consider some examples for $m\neq kn\pm 1$ in Section \ref{Section: Examples}. 

It is worth to mention that the combinatorial structure of the dilated fundamental alcove has been recently investigated in \cite{ThWi}, where the alcoves in it 
were labeled by certain sequences of numbers (but not parking functions).  We plan to investigate the connections of our work to \cite{ThWi} in the future.


\omitt
{Many of the concepts above have ``rational" counterparts corresponding to a coprime pair
$(m,n)$, for which the classical case is $(n+1,n)$.
In \cite{GM11} and \cite{GM12} the first and the second authors proved that the rational Dyck paths label the affine cells in a certain algebraic variety,
the compactified Jacobian of a plane curve singularity with one Puiseaux pair $(m,n)$.
This work was generalized by Hikita in \cite{H12} who proved that $m/n$ parking functions naturally label the cells in a certain Springer fibre in the affine flag variety for the affine symmetric group $\aSn$. 
As a consequence of this paper, we can  use our construction to prove new formulas for the
Poincar\'e
polynomials of certain affine Springer fibers relying on the
work of \cite{GM11} and \cite{GM12}.

The present paper is dedicated to the systematic study of the $m/n$ parking functions in terms of the affine symmetric group $\aSn$.
We consider the set $\aSmn$ of all affine permutations in $\aSn$ with no inversions
of height $m$ and two maps $\An$ and $\PS$ from $\aSmn$ to the set of $m/n$ parking
functions $\Pmn$.
When one composes $\PS$ with $\An^{-1}$,
one recovers the zeta map $\zeta$ of Haglund, which was used to give a combinatorial proof of the weak symmetry property of the $q,t$-Catalan numbers.}

\section*{Acknowledgements}

The authors would like to thank American Institute of Mathematics (AIM) for hospitality, and  D. Armstrong, F. Bergeron, S. Fishel, I. Pak, R. Stanley, V. Reiner, B. Rhoades, A. Varchenko, G. Warrington  and N. Williams for useful discussions and suggestions.
The work of E. G. was partially supported by the grants RFBR-13-01-00755, NSh-4850.2012.1.
The work of M.~V. was partially supported by the grant NSA MSP H98230-12-1-0232.

\section{Tools and definitions}\label{Section: Tools and Definitions}

We start with a brief review of the definitions and basic results involving parking functions, affine permutations, and hyperplane arrangements, which will play the key role in our constructions.

\subsection{Parking Functions}

\begin{definition}
A function $f:\{1,\dots, n\}\to \BZ_{\ge 0}$ is called an
{\em  $m/n$-parking function}
 if the Young diagram with row lengths equal to $f(1),\dots, f(n)$ put in the decreasing order,
bottom to top,
 fits under the diagonal in an $n\times m$ rectangle. The set of such functions is denoted by $\Pmn.$
\end{definition}

We  will often use the notation $f=\pw{f(1)f(2)\ldots f(n)}$ for parking functions.

\begin{example}
\label{example:pf diagram}
Consider the function $f:\{1,2,3,4\}\to \BZ_{\ge 0}$ given by $f(1)=2,$ $f(2)=0,$ $f(3)=4,$ and $f(4)=0$ (i.e. $f=\pw{2040}$). The corresponding Young diagram fits under the diagonal in a $4\times 7$ rectangle, but it does not fit under the diagonal in a $4\times 5$ rectangle. Therefore, $f\in PF_{7/4}$ but $f\notin PF_{5/4}$ (see Figure \ref{Figure: parking function}).

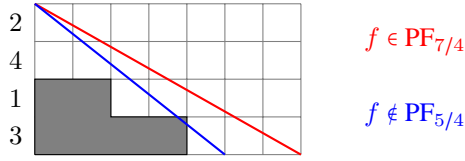
\begin{figure}
\begin{center}
\begin{tikzpicture}
\draw [step=0.5, gray] (0,0) grid (3.5,2);

\draw [red, thick] (0,2)--(3.5,0);

\filldraw [fill=gray] (0,0)--(0,1)--(1,1)--(1,0.5)--(2,0.5)--(2,0)--(0,0);

\draw [blue, thick] (0,2)--(2.5,0);

\draw (-0.25,1.75) node {$2$};
\draw (-0.25,1.25) node {$4$};
\draw (-0.25,0.75) node {$1$};
\draw (-0.25,0.25) node {$3$};

\draw [red] (5,1.5) node {$f\in\mbox{PF}_{7/4}$};

\draw [blue] (5,0.5) node {$f\notin\mbox{PF}_{5/4}$};

\end{tikzpicture}
\end{center}
\caption{The labeled diagram for the parking function $f=\pw{2040}.$}
\label{Figure: parking function}
\end{figure}
\end{example}

Equivalently, a function $f:\{1,\ldots,n\}\to \mathbb Z_{\ge 0}$ belongs to $\PF_{m/n}$ if and only if it satisfies one of the following two equivalent conditions:

\begin{equation*}\label{Definition: m/n parking}
\forall \ell\in \{0,\ldots,m-1\},\ \sharp\{k\in\{1,\dots, n\} \mid 
f(k)< \ell\}\ge \frac{\ell n}{m},
\end{equation*}
or
\begin{equation*}
\forall i\in \{0,\ldots,n-1\},\ \sharp\{k\in\{1,\dots, n\} \mid f(k)\le \frac{im}{n}\}\ge i+1.
\end{equation*}

Let $P:\Pmn\to\Ymn$ denote the natural map from the set of parking functions to the set $\Ymn$ of Young diagrams that fit under diagonal in an $n\times m$ rectangle. To recover a parking function $f\in\Pmn$ from the corresponding Young diagram $P(f)$ one needs some extra information. Lengths of the rows of $P(f)$ correspond to the values of $f,$ but one needs also to assign the preimages to them. That is, one should label the rows of $P(f)$ by integers $1,2,\ldots,n.$ Note that if $P(f)$ has two rows of the same length, then the order of the corresponding labels does not matter. One should choose one of the possible orders. We choose the decreasing order (read from bottom to top).  

\begin{definition}
Let $\lYmn$ denote the set of couples $(D,\tau)$ of a Young diagram $D\in\Ymn$ and a (finite) permutation $\tau\in\Sn,$ such that if $k$th and $(k+1)$th rows of $D$ have the same length, than $\tau(k+1)<\tau(k).$ We will refer to $\tau$ as the row-labeling of $D.$ 
\end{definition}
Note that $\tau \in \Sn$ is the permutation of maximal length such that $f \circ \tau$
is non-increasing. 

\begin{example}
In Example \ref{example:pf diagram}, one has $\tau = [3,1,4,2]$, so $f\circ\tau=\pw{2040} \circ [3,1,4,2] = \pw{4200}$.
\end{example}


We get the following lemma:

\begin{lemma}\label{Lemma:PF <-> labeled diagrams}
The set of $m/n$-parking functions $\Pmn$
is in bijection with
the set of labeled Young diagrams $\lYmn.$  
\end{lemma}

\begin{remark}
Note that for $m=n+1$ the set $\Pmn$ is exactly the set of classical parking functions $\PF$, and for $m=kn+1$ it is the set of $k$-parking functions $\PF_k$ (e.g. \cite{Hd08}).
\end{remark}

From now on we will assume that $m$ and $n$ are coprime, so there are no lattice points in the diagonal of $n\times m$ rectangle.
By abuse of notation, we will call a non-decreasing parking function increasing. The number of increasing parking functions equals
to the generalized Catalan number $\sharp \Ymn = \frac{1}{n+m}\binom{n+m}{n}$. The number of all parking functions equals $m^{n-1}$.


\subsection{Affine Permutations}\label{Subsection: Affine Permutations}

\begin{definition}
The {\it affine symmetric group} $\aSn$ is generated by elements $s_1,\ldots,s_{n-1},s_0$ subject to the relations
\begin{enumerate}
\item[(a)] $s_i^2=1,$
\item[(b)] $s_is_j=s_js_i$ for $i-j\not\equiv \pm 1$ mod $n,$
\item[(c)] $s_is_js_i=s_js_is_j$ for $i-j\equiv \pm 1$ mod $n$ \, (if $n > 2$).
\end{enumerate} 
\end{definition}
Let 
$$
\x=\left( \begin{array}{l} x_1 \\ \vdots \\ x_n \end{array} \right),\ \ V:=\{\x\in \R^n \mid x_1+\ldots+x_n=0\}\subset\mathbb R^n$$
 and let $H_{ij}^k$ be the hyperplane
$\{ 
\x\in V \mid x_i-x_j=k\}\subset V.$
The hyperplane arrangement
\omitt{
$\aBrn=\{H_{ij}^m:0<j<i\le n,m\in\BZ\}$
} 
$\aBrn=\{H_{ij}^k:0<i<j\le n,k\in\BZ\}$
is called the {\it affine braid arrangement.} The connected components of the complement to the affine braid arrangement are called {\it alcoves.} The group $\aSn$ acts on $V$ with the generators $s_i$ acting by reflections in hyperplanes $H_{i,i+1}^0$ for $i>0,$ and $s_0$ acting by reflection in the hyperplane
$H_{1,n}^1.$
The action is free and transitive on the set of alcoves, so that the map $\omega\mapsto\omega(\Afund),$ where
\omitt{
$\Afund:=\{x_1<x_2<\ldots<x_n<x_1+1\}$
} 
$\Afund:=\{ 
\x\in V \mid
x_1>x_2>\ldots>x_n>x_1-1\}$
is the {\it fundamental alcove}, gives a bijection between the group $\aSn$ and the set of alcoves.

 Observe $H_{i, j}^k = H_{j,i}^{-k}$, so we may always take $i<j$.
It is convenient to extend our notation to allow subscripts in $\Z$ via
 $H_{i+tn, j+tn}^k = H_{i,j}^k$ and 
 $H_{i, j}^k = H_{i,j-n}^{k-1}$.
In this way, we can uniquely write each hyperplane in $\aBrn$
as 
 $H_{i, \l}^0$ with $1 \le i \le n$, $i < \l$, $\l \in \Z$. 
Then we can define the {\em height} of the hyperplane
$H_{i, \l}^0$ to be $\l -i$. Observe, in this manner, the reflecting hyperplane of $s_0$ is $H_{1,n}^1 = H_{1,0}^0 = H_{0,1}^0$ of height $1$. Note that with this notation,
the action of the group $\aSn$ on the hyperplanes $H_{i,j}^k$ is given by
$$
\w(H_{i,j}^k)=H_{\w(i),\w(j)}^k.
$$

There is another way to think about the affine symmetric group:

\begin{definition}
A bijection $\omega:\BZ\to \BZ$ is called an affine $S_n$-permutation, if $\omega(x+n)=\omega(x)+n$ for all $x,$
and $\sum_{i=1}^{n}\omega(i)=\frac{n(n+1)}{2}.$
\end{definition}

In this presentation the operation is composition and the generators $s_1,\ldots, s_{n-1},s_0$ are given by  
\begin{enumerate}
\item[(a)] $s_i(x)=x+1$ for $x\equiv i$ mod $n,$
\item[(b)] $s_i(x)=x-1$ for $x\equiv i+1$ mod $n,$
\item[(c)] $s_i(x)=x$ otherwise.
\end{enumerate}
It is convenient to use 
list or {\em window notation}
for $\omega\in \aSn$
as the list
$[\w(1), \w(2), \cdots, \w(n)]$.  
Since $\w(x+n)=\w(x)+n,$ this determines $\w$. The bijection between $\aSn$ and the set of alcoves can be made more explicit in the following way. 

\begin{lemma}\label{Lemma: centroids}
Every alcove $\A$ contains exactly one point $(x_1,\ldots,x_n)^T
\in \A$ in its interior such that the numbers $\frac{n+1}{2}-nx_1,\ldots,\frac{n+1}{2}-nx_n$ are all integers. Moreover, if $\x\in\omega(\Afund)$ is such a point, then in the window notation one has 
\begin{equation}\label{Equation: centroid}
\omega^{-1}=[\frac{n+1}{2}-nx_1,\ldots,\frac{n+1}{2}-nx_n].
\end{equation}
These points are called {centroids} of alcoves.
\end{lemma}

\begin{proof}
Window notation for the identity permutation is $\id =[1,2,\ldots,n].$
By \eqref{Equation: centroid}, the corresponding point is $\frac{1}{2n}(n-1,n-3,\ldots,1-n).$ Note that it belongs to the fundamental alcove $\Afund=\{\x\in V \mid
x_1>x_2>\ldots>x_n>x_1-1\}.$ Moreover, it is the unique point $\x\in\Afund$ such that the numbers $\frac{n+1}{2}-nx_i$ are all integers. Indeed, let $a_i=\frac{n+1}{2}-nx_i$ for all $1\le i\le n.$ Since $x_1>x_2>\ldots>x_n>x_1-1$ we get $a_1<a_2<\ldots<a_n<a_1+n.$ Moreover, since $x_1+\ldots+x_n=0,$ we have $a_1+\ldots+a_n=\frac{n(n+1)}{2}.$ There is a unique collection of integers satisfying these conditions: $a_1=1,a_2=2,\ldots,a_n=n.$

Since $\aSn$ acts freely and transitively on the set of alcoves, all we need to prove is that \eqref{Equation: centroid} is preserved under the action of the generators $s_0,\ldots,s_n.$ Indeed, for $1\le i\le n$ we have 
$$
(s_i\omega)^{-1}=\omega^{-1}s_i=[\omega^{-1}(s_i(1)),\ldots,\omega^{-1}(s_i(n))]
$$
$$
=[\omega^{-1}(1),\ldots,\omega^{-1}(i+1),\omega^{-1}(i),\ldots,\omega^{-1}(n)],
$$
and for $i=0$ we have
$$
(s_0\omega)^{-1}=[\omega^{-1}(n)-n,\omega^{-1}(2),\ldots,\omega^{-1}(n-1),\omega^{-1}(1)+n].
$$
On the other side, generators $s_1,\ldots,s_n\in\Sn$ simply permute the coordinates of points in $V,$ while $s_0$ acts by sending $(x_1,\ldots,x_n)$ to $(x_n+1,x_2,\ldots,x_{n-1},x_1-1).$ Therefore, Equation \ref{Equation: centroid} is preserved by the action of the group $\aSn.$
\end{proof}

The minimal length left coset representatives $\w \in \aSn/\Sn$,
also known as affine Grassmannian permutations, satisfy
$\w(1) < \w(2) < \cdots < \w(n)$, so that their window notation
is an {\em increasing} list of integers (summing to $\frac{ n(n+1)}{2}$
and with distinct remainders $\mod n$).
Their inverses $\w^{-1}$ are the minimal length right coset representatives
and satisfy that the centroid of the alcove $\w^{-1}(\Afund)$ are precisely
those whose coordinates are {\em decreasing}.
That is to say, $\w^{-1}(\Afund)$ is in the dominant chamber
$\{ 
\x\in V\mid x_1 > x_2 > \cdots > x_n \}$.

By a slight abuse of notation, we will refer to the set of minimal
length left (right) coset representatives as $\aSn/\Sn$ (respectively $\Sn\backslash\aSn$).

\subsection{Sommers region}\label{Subsection: Sommers region}

The notions of an inversion and the length of a permutation generalizes from the symmetric group $\Sn$ to the affine symmetric group $\aSn.$ However, the set $\{(i,j)\in\Z^2\mid i<j, \w(i)>\w(j)\}$ is infinite for all $\w\in \aSn$ except identity. That is why it makes more sense to consider inversions up to shifts by multiples of $n:$

\begin{definition}
Let $\w$ be an affine permutation. The set of its inversions is defined as
$$
\Inv(\w) : =
\{ (i,j) \in \Z \times \Z \mid 1\le i\le n,\ i < j,  \w(i) > \w(j) \}.
$$
The length of a permutation $\w$ is then defined as $\ell(\w) = \sharp \Inv(\w).$ We shall say the {\em height} of an inversion $(i,j)$ is $j-i$. We will also use the notation
$$
\Invall(\w) : =
\{ (i,j) \in \Z \times \Z \mid \ i < j,  \w(i) > \w(j) \}
$$
for unnormalized inversions.
\end{definition}

\begin{remark}
\label{remark-inversions-equivalent}
If $(i,j)\in\Inv(\w),$ then obviously $i+kn<j+kn$ and $\w(i+kn)>\w(j+kn)$ for any integer $k.$ Essentially, these couples of integers represent the same inversion of $\w.$ The condition $1\le i\le n$ allows us to count each inversion exactly once. Alternatively, one could also require $1\le j\le n,$\ $1\le \w(i)\le n,$ or $1\le \w(j)\le n.$
\end{remark}

\begin{example}
Consider $\w = [-3,2,3,8] \in \aSnn{4}/S_4$, whose inverse is $\w^{-1} =
[5,2,3,0]$.  The centroid of $\w^{-1}(\Afund)$ is $\frac 18 (11,1,-1,-11)^T$.
Note $\w$ is translation by the vector $\mu = (-1,0,0,1)^T$, as
$\w = [1 -1(4), 2+0(4), 3+0(4), 4+1(4)]$ and likewise $\w^{-1}$
is translation by $-\mu$.  One can see the centroid above is
the centroid of the fundamental alcove translated by $-\mu$.
  In terms of Coxeter generators,
$\w = s_1 s_2 s_3 s_2 s_1 s_0$ and
$\w^{-1} = s_0 s_1 s_2 s_3 s_2 s_1$.
Note
$\Inv(\w) = \{ (4,5), (4,6), (4,7), (4,9), (3,5), (2,5) \}$ and $\ell(\w) = 6$
which is also its Coxeter length.
The inversions are of height $1, 2, 3, 5, 2, 3$ respectively.
Additionally, $\Inv(\w^{-1}) =
\{ (1,2), (1,3), (1,4), (1,8), (2,4), (3,4) \}$.
\end{example}

Geometrically, $\w$ has an inversion of height $m$ if and only if
the alcove $\w^{-1}(\Afund)$ is separated from $\Afund$ by a (corresponding)
hyperplane of height $m$.
More precisely, that hyperplane is $H_{i, i+m}^0$ if the inversion is $(i,i+m)$.
The following definition will play the key role in our constructions:

\begin{definition}
An affine permutation $\omega\in \aSn$ is called 
{\em $m$-stable}
if for all $x$ the inequality $\omega(x+m)>\omega(x)$ holds,
i.e. $\w$ has no inversions of height $m$.
The set of all $m$-stable affine permutations is denoted by $\aSmn$.
\end{definition}

\begin{definition}
An affine permutation $\omega\in \aSn$ is called 
{\em $m$-restricted}
if $\omega^{-1}\in \aSmn$.
We will denote the set of $m$-restricted permutations by $\maSn$.
Note $\w \in \maSn$ if and only if for all $i < j$, $\w(i) - \w(j) \neq m$.
\end{definition}

Lemma \ref{Lemma: centroids} implies an important corollary for the set $\aSmn$:

\begin{lemma}\label{Lemma: Sommers region}
Let $m=kn+r,$ where $0<r<n.$ The set of alcoves $\{\omega(\Afund):\omega\in\maSn\}$ coincides with the set of alcoves that fit inside the region $D_n^m\subset V$ defined by the inequalities:
\begin{enumerate}
\item $x_i - x_{i+r} \ge -k$ for $1\le i\le n-r,$   
\item $x_{i+r-n}-x_i \le k+1$ for $n-r+1\le i\le n.$
\end{enumerate}
\end{lemma}

\begin{proof}
Remark that $D_n^m$ is precisely the region cut out by the hyperplanes of height $m$,
as $H_{i, i+r}^{-k} = H_{i, i+r + kn}^0 = H_{i, i+m}^0$ and likewise
$H_{ i+r-n, i}^{k+1} = H_{ i+r -n + (k+1)n, i}^0 = H_{i+m, i}^0= H_{i, i+m}^0$.
This means that alcove $\w^{-1}(\Afund)$ is inside $D_{n}^{m}$ if and only if the permutation $\w$ has no inversions of height $m$.
\end{proof}

The region $D_n^m$ was considered by Sommers in \cite{So},
therefore we call it the {\it Sommers region}. It is known that $D_{n}^{m}$ is isometric to the $m$th dilation of the fundamental alcove. This  was proven for all types by Fan \cite[Lemma 2.2]{fan} and Sommers \cite[Theorem 5.7]{So}, based on an earlier unpublished observation of Lusztig. It is worth emphasizing that in type $A$ the construction of the isometry is very clear.

\begin{lemma}
\label{isometry}
Let $c=\frac{(m-1)(n+1)}{2}$
The affine permutation $\wm:=[m-c,2m-c,\ldots,nm-c]$ induces an isometry
between $D_{n}^{m}$ and the simplex $mD_{n}^{1} = m \Afund$ in the dominant region cut out by the hyperplane $x_1-x_n=m$:
$$
\wm\left(mD_{n}^{1}\right)=D_{n}^{m}.
$$
\end{lemma}

\begin{proof}
By the proof of Lemma \ref{Lemma: Sommers region}, the region $D_n^m$ is cut out by the hyperplanes $H_{i,i+m}^{0}$ for all integer $i$. Since $m$ and $n$ are coprime, one can equivalently say that it is cut out by the hyperplanes
$$
H_{m-c,2m-c}^{0}, H_{2m-c,3m-c}^{0}, \ldots, H_{nm-c,(n+1)m-c}^{0}.
$$
On the other hand, the hyperplane $x_1-x_n=m$ can be written as $H_{n,mn+1}^0$, so the simplex $mD_{n}^{1}$ is cut out by the hyperplanes $H_{1,2}^{0},H_{2,3}^{0},\ldots, H_{n-1,n}^{0},H_{n,mn+1}^0$.
Remark that $\wm(i)=mi-c$ for $1\le i\le n$, and 
$$
\wm(mn+1)=\wm(1)+mn=m-c+mn=m(n+1)-c.
$$
Therefore $\wm(H_{i,i+1}^{0})=H_{mi-c,m(i+1)-c}^0$
for $1\le i<n$, and
$$\wm(H_{n,mn+1}^{0})=H_{mn-c,m(n+1)-c}^{0} \, ,$$
hence $\wm(mD_{n}^{1})=D_{n}^{m}.$
\end{proof}

Observe that, since $m$ and $n$ are coprime,
 the image of the origin under this isometry will be
the unique vertex of $D_n^m$ with all integer entries (in other words, in the
root lattice). 

\begin{example}
For $n=3$ and $m=2$ we have $\w_2=[024]$. The dilated fundamental alcove is bounded by the hyperplanes $H_{2,3}^{0},H_{1,2}^{0}$ and $H_{1,3}^2$,
the Sommers region $D_{3}^{2}$ is bounded by the hyperplanes
 $H_{1,3}^0, H_{2,4}^{0}$ and $H_{3,5}^{0}$.
Note that
\omitt{
\w_2(H_{2,3}^{0})=H_{2,4}^{0}=H_{1,2}^{1},\
\w_2(H_{1,2}^{0})=H_{0,2}^{0}=H_{2,3}^1,\
\w_2(H_{1,3}^2)=H_{0,4}^{2}=H_{1,3}^{0}. 
}
$$ \w_2(H_{2,3}^{0})=H_{2,4}^{0},\ \w_2(H_{1,2}^{0})=H_{0,2}^{0}=H_{3,5}^{0},\
\w_2(H_{1,3}^2)=H_{0,4}^{2}=H_{1,3}^{0}.  $$
Similarly, for $m=4$ we have $\w_4=[-226]$. The dilated fundamental alcove is bounded by the hyperplanes $H_{2,3}^{0},H_{1,2}^{0}$ and $H_{1,3}^4$,
the Sommers region $D_{3}^{4}$ is bounded by the hyperplanes
$H_{1,5}^{0}, H_{2,6}^{0}$ and $H_{3,7}^{0}$.  
Note that $$
\w_4(H_{2,3}^{0})=H_{2,6}^{0},\ \w_4(H_{1,2}^{0})=H_{-2,2}^{0}=H_{1,5}^{0},\
\w_4(H_{1,3}^4)=H_{-2,6}^{4}=H_{3,7}^{0}.  $$
\omitt{
\w_4(H_{2,3}^{0})=H_{2,6}^{0}=H_{2,3}^{-1},\ \w_4(H_{1,2}^{0})=H_{-2,2}^{0}=H_{1,2}^{-1},\ \w_4(H_{1,3}^4)=H_{-2,6}^{4}=H_{1,3}^{2}.
}
All these hyperplanes are shown in Figure \ref{fig:sommers region 3 2} and  Figure \ref{fig:sommers region 3 4}.
\end{example}

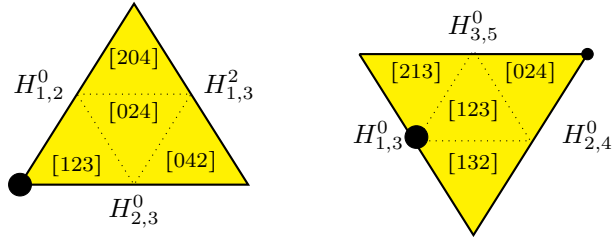
\begin{figure}   
\begin{center}
\begin{tikzpicture}[scale=0.5]

\filldraw [yellow, fill=yellow] (-0.5,-2.6)--(5.5,-2.6)--(2.5,2.2)--(-0.5,-2.6);
\filldraw [fill=black] (-0.5,-2.6) circle (0.3);

\draw [thick] (-0.5,-2.6)--(5.5,-2.6)--(2.5,2.2)--(-0.5,-2.6);
\omitt{
\draw [dotted] (0.5,-0.2)--(4.5,-0.2);
\draw [dotted] (-0.5,2.2)--(3,-3.4);
\draw [dotted] (2,-3.4)--(5.5,2.2);
} 

\draw[dotted] 
(1,-0.2)--(4,-0.2)--(2.5,-2.6)--(1,-0.2);
 
\draw (2.5,0.75) node {\footnotesize   $[204]$};
\draw (4,-2) node {\footnotesize  $[042]$};
\draw (2.5,-0.7) node {\footnotesize  $[024]$};
\draw (1,-2.1) node {\footnotesize  $[123]$};


\draw (2.5,-3.3) node {$H_{2,3}^{0}$};
\draw (0.0,0.0) node {$H_{1,2}^{0}$};
\draw (5,0.0) node {$H_{1,3}^{2}$};

\end{tikzpicture}
\ \ \ 
\qquad
\begin{tikzpicture}[scale=0.5]

\filldraw [yellow, fill=yellow] (2.5,2.2)--(8.5,2.2)--(5.5,-2.6)--(2.5,2.2);
\filldraw [fill=black] (4,0) circle (0.3);
\filldraw [fill=black] (8.5,2.2) circle (0.15);

\draw [thick] (2.5,2.2)--(8.5,2.2)--(5.5,-2.6)--(2.5,2.2);
\omitt{
\draw [dotted] (3.5,-0.2)--(7.5,-0.2);
\draw [dotted] (2.5,-2.6)--(6,3);
\draw [dotted] (8.5,-2.6)--(5,3);
} 

\draw[dotted] 
(4,-0.1)--(7,-0.1)--(5.5,2.3)--(4,-0.1);

\draw (5.5,0.7) node {\footnotesize $[123]$};

\draw (7,1.7) node {\footnotesize $[024]$};

\draw (4,1.7) node {\footnotesize $[213]$};

\draw (5.5,-0.7) node {\footnotesize $[132]$};

\draw (5.5,2.9) node {$H_{3,5}^{0}$};
\draw (8.5,0.0) node {$H_{2,4}^{0}$};
\draw (3.0,0.0) node {$H_{1,3}^{0}$};
 
\end{tikzpicture}
\caption{Dilated fundamental alcove (left) and Sommers regions (right) for $m=2$ ,
$\w_2=[024]$ }
\label{fig:sommers region 3 2}
\end{center}
\end{figure}
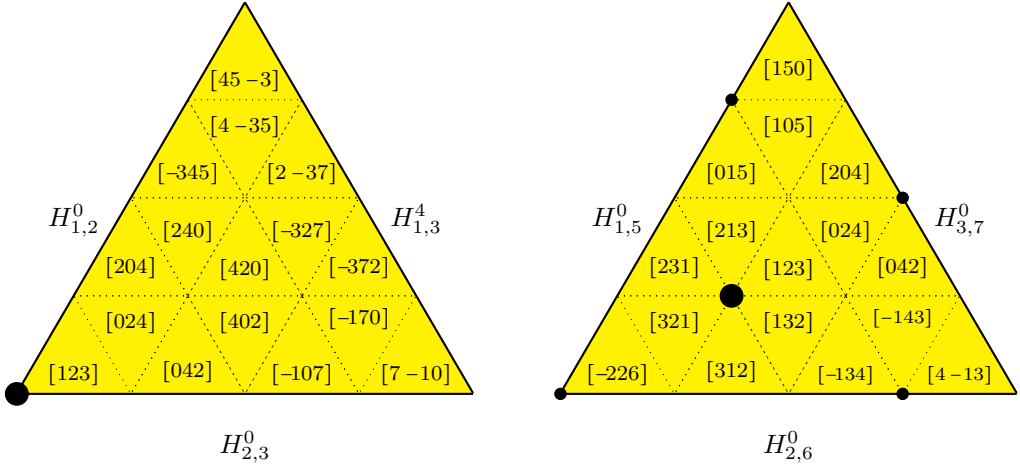
\begin{figure}
\begin{center}
\begin{tikzpicture}[scale=0.5]

\filldraw [yellow, fill=yellow] (5.5,7.8)--(11.5,-2.6)--(-0.5,-2.6)--(5.5,7.8);
\filldraw [fill=black] (-0.5,-2.6) circle (0.3);

\draw[thick]  (-0.5,-2.6)--(11.5,-2.6);
\omitt{
\draw[dotted] (0,0)--(11,0);
\draw[dotted] (0,2.6)--(11,2.6);
\draw[dotted] (0,5.2)--(11,5.2);
}

\draw[thick] (-.5,-2.58)--(5,6.92)--(5.5, 7.783);
\omitt{
\draw[dotted] (2,-3.46)--(7.5,6.06);
\draw[dotted] (4.5,-4.33)--(10,5.2);
\draw[dotted] (7,-5.2)--(12.5,4.33);

\draw[dotted] (-1.5,4.33)--(4,-5.2);
\draw[dotted] (1, 5.2)--(6.5,-4.33);
\draw[dotted]  (3.5,6.06)--(9,-3.46);
}
\draw[thick] (5.5,7.783 )--(6,6.92)--(11.5,-2.58);


\draw[dotted]
 (4,0)--(7,0)--(5.5,2.6)--(4,0);
\draw[dotted]
 (4,0)--(7,0)--(5.5,-2.6)--(4,0);
\draw[dotted]
 (4,0)--(5.5,2.6)--(2.5,2.6)--(4,0);
\draw[dotted]
 (7,0)--(5.5,2.6)--(8.5,2.6)--(7,0);
\draw[dotted]
 (7,0)--(10,0)--(8.5,2.6)--(7,0);
\draw[dotted]
 (1,0)--(4,0)--(2.5,2.6)--(1,0);
\draw[dotted]
 (4,5.2)--(5.5,2.6)--(2.5,2.6)--(4,5.2);
\draw[dotted]
 (7,5.2)--(5.5,2.6)--(8.5,2.6)--(7,5.2);
\draw[dotted]
 (4,0)--(5.5,-2.6)--(2.5,-2.6)--(4,0);
\draw[dotted]
 (7,0)--(5.5,-2.6)--(8.5,-2.6)--(7,0);
\draw[dotted]
 (4,5.2)--(7,5.2)--(5.5,2.6)--(4,5.2);
\draw[dotted]
 (1,0)--(4,0)--(2.5,-2.6)--(1,0);
\draw[dotted]
 (7,0)--(10,0)--(8.5,-2.6)--(7,0);

\draw (5.5,0.7) node {\footnotesize   $[420]$};
\draw (7,1.7) node {\footnotesize   $[-\!327]$};
\draw (4,1.7) node {\footnotesize   $[240]$};
\draw (5.5,-0.7) node {\footnotesize   $[402]$};
\draw (8.5,0.75) node {\footnotesize  $[-\!372]$};
\draw (2.5,0.75) node {\footnotesize   $[204]$};
\draw (4,3.25) node {\footnotesize   $[-\!345]$};
\draw (7,3.25) node {\footnotesize   $[2-\!37]$};
\draw (7,-2.1) node {\footnotesize   $[-\!107]$};
\draw (4,-2) node {\footnotesize  $[042]$};
\draw (5.5,5.7) node {\footnotesize  $[45-\!3]$};
\draw (5.5,4.5) node {\footnotesize $[4-\!35]$};
\draw (10,-2.1) node {\footnotesize $[7-\!10]$}; 
\draw (8.5,-0.6) node {\footnotesize   $[-\!170]$};
\draw (2.5,-0.7) node {\footnotesize  $[024]$};
\draw (1,-2.1) node {\footnotesize  $[123]$};

\omitt{
\draw(10.5,-3.1) node {$H_{2,3}^0$};
\draw(9.7,1.8) node {$H_{1,3}^4$};
\draw(4.3,7.1) node {$H_{1,2}^0$};
} 

\draw (10,2) node {$H_{{1},3}^4$};
\draw (1,2) node {$H_{{1},2}^0$};
\draw (5.5,-4.0) node {$H_{{2},3}^0$};

\end{tikzpicture}
\ \ \ 
\qquad
\begin{tikzpicture}[scale=0.5]

\filldraw [yellow, fill=yellow] (5.5,7.8)--(11.5,-2.6)--(-0.5,-2.6)--(5.5,7.8);

\omitt{
\filldraw [yellow, fill=yellow] (4,0)--(7,0)--(5.5,2.6)--(4,0);
\filldraw [yellow, fill=yellow] (4,0)--(7,0)--(5.5,-2.6)--(4,0);
\filldraw [yellow, fill=yellow] (4,0)--(5.5,2.6)--(2.5,2.6)--(4,0);
\filldraw [yellow, fill=yellow] (7,0)--(5.5,2.6)--(8.5,2.6)--(7,0);
\filldraw [yellow, fill=yellow] (7,0)--(10,0)--(8.5,2.6)--(7,0);
\filldraw [yellow, fill=yellow] (1,0)--(4,0)--(2.5,2.6)--(1,0);
\filldraw [yellow, fill=yellow] (4,5.2)--(5.5,2.6)--(2.5,2.6)--(4,5.2);
\filldraw [yellow, fill=yellow] (7,5.2)--(5.5,2.6)--(8.5,2.6)--(7,5.2);
\filldraw [yellow, fill=yellow] (4,0)--(5.5,-2.6)--(2.5,-2.6)--(4,0);
\filldraw [yellow, fill=yellow] (7,0)--(5.5,-2.6)--(8.5,-2.6)--(7,0);
\filldraw [yellow, fill=yellow] (4,5.2)--(7,5.2)--(5.5,2.6)--(4,5.2);
\filldraw [yellow, fill=yellow] (1,0)--(4,0)--(2.5,-2.6)--(1,0);
\filldraw [yellow, fill=yellow] (7,0)--(10,0)--(8.5,-2.6)--(7,0);
\filldraw [yellow, fill=yellow] (4,5.2)--(5.5,2.6)--(2.5,2.6)--(4,5.2);
\filldraw [yellow, fill=yellow] (4,5.2)--(7,5.2)--(5.5,7.8)--(4,5.2);
\filldraw [yellow, fill=yellow] (10,0)--(11.5,-2.6)--(8.5,-2.6)--(10,0);
\filldraw [yellow, fill=yellow] (1,0)--(2.5,-2.6)--(-0.5,-2.6)--(1,0);
} 

\filldraw [fill=black] (4,0) circle (0.3);
\filldraw [fill=black] (8.5,2.6) circle (0.15);
\filldraw [fill=black] (-.5,-2.6) circle (0.15);
\filldraw [fill=black] (8.5,-2.6) circle (0.15);
\filldraw [fill=black] (4,5.2) circle (0.15);

\draw[thick]  (-0.5,-2.6)--(11.5,-2.6);
\omitt{
\draw[dotted] (0,0)--(11,0);
\draw[dotted] (0,2.6)--(11,2.6);
\draw[dotted] (0,5.2)--(11,5.2);
}

\draw[thick] (-.5,-2.58)--(5,6.92)--(5.5, 7.783);
\omitt{
\draw[dotted] (2,-3.46)--(7.5,6.06);
\draw[dotted] (4.5,-4.33)--(10,5.2);
\draw[dotted] (7,-5.2)--(12.5,4.33);

\draw[dotted] (-1.5,4.33)--(4,-5.2);
\draw[dotted] (1, 5.2)--(6.5,-4.33);
\draw[dotted]  (3.5,6.06)--(9,-3.46);
} 
\draw[thick] (5.5,7.783 )--(6,6.92)--(11.5,-2.58);


\draw[dotted]
 (4,0)--(7,0)--(5.5,2.6)--(4,0);
\draw[dotted]
 (4,0)--(7,0)--(5.5,-2.6)--(4,0);
\draw[dotted]
 (4,0)--(5.5,2.6)--(2.5,2.6)--(4,0);
\draw[dotted]
 (7,0)--(5.5,2.6)--(8.5,2.6)--(7,0);
\draw[dotted]
 (7,0)--(10,0)--(8.5,2.6)--(7,0);
\draw[dotted]
 (1,0)--(4,0)--(2.5,2.6)--(1,0);
\draw[dotted]
 (4,5.2)--(5.5,2.6)--(2.5,2.6)--(4,5.2);
\draw[dotted]
 (7,5.2)--(5.5,2.6)--(8.5,2.6)--(7,5.2);
\draw[dotted]
 (4,0)--(5.5,-2.6)--(2.5,-2.6)--(4,0);
\draw[dotted]
 (7,0)--(5.5,-2.6)--(8.5,-2.6)--(7,0);
\draw[dotted]
 (4,5.2)--(7,5.2)--(5.5,2.6)--(4,5.2);
\draw[dotted]
 (1,0)--(4,0)--(2.5,-2.6)--(1,0);
\draw[dotted]
 (7,0)--(10,0)--(8.5,-2.6)--(7,0);

\draw (5.5,0.7) node {\footnotesize $[123]$};

\draw (7,1.7) node {\footnotesize $[024]$};

\draw (4,1.7) node {\footnotesize $[213]$};

\draw (5.5,-0.7) node {\footnotesize $[132]$};

\draw (8.5,0.75) node {\footnotesize $[042]$};
\draw (2.5,0.75) node {\footnotesize $[231]$};
\draw (4,3.25) node {\footnotesize $[015]$};
\draw (7,3.25) node {\footnotesize $[204]$};
\draw (7,-2.1) node {\scriptsize $[-\!134]$};
\draw (4,-2) node {\footnotesize $[312]$};

\draw (5.5,6) node {\footnotesize $[150]$};
\draw (5.5,4.5) node {\footnotesize $[105]$};
\draw (10,-2.1) node {\scriptsize $[4-\!13]$}; 
\draw (8.5,-0.6) node {\scriptsize $[-143]$};
\draw (2.5,-0.7) node {\footnotesize $[321]$};
\draw (1,-2.1) node {\footnotesize $[-\!226]$};

\omitt{ for d15 2f238c6 
\draw(10.5,-3.1) node {$H_{2,3}^{-1}$};
\draw(6.8,-3.5) node {$H_{1,3}^0$};
\draw(4.0,-3.5) node {$H_{1,2}^1$};
\draw(9.8,1.5) node {$H_{1,3}^2$};
\draw(4.3,7.1) node {$H_{1,2}^{-1}$};
\draw(1.0,1.9) node {$H_{2,3}^1$};
 }

\draw (10,2) node {$H_{{3},7}^0$};
\draw (1,2) node {$H_{{1},5}^0$};
\draw (5.5,-4.0) node {$H_{{2},6}^0$};
\end{tikzpicture}
\caption{Dilated fundamental alcoves (left) and Sommers regions (right) for  $m=4$,
$ \w_4=[-226]$}
\label{fig:sommers region 3 4}
\end{center}
\end{figure}

\section{Main Constructions.}\label{Section: Main Constructions}

\subsection{Bijection $\An:\aSmn\to \Pmn$}

We define the map $\An:\aSmn\to \Pmn$ by the following procedure. Given $\w\in \aSmn$, consider the set $\Delta_{\omega}:=\{i\in\BZ: \omega(i)>0\}\subset\BZ$ and let $M_{\w}$ be its minimal element. Note that the set $\Delta_{\omega}$ is invariant under addition of $m$ and $n.$ Indeed, if $i\in \Delta_{\w}$ then $\w(i+m)>\w(i)>0$ and $\w(i+n)=\w(i)+n>n>0.$ Therefore $i+m\in\Delta_\w$ and $i+n\in\Delta_\w.$

Consider the integer lattice $(\mathbb Z)^2.$ We prefer to think about it as of the set of square boxes, rather than the set of integer points. Consider the rectangle $R_{m,n}:=\{(x,y)\in (\mathbb Z)^2 \mid 0\le x<m, 0\le y<n\}.$ Let us label the boxes of the lattice according to the linear function 
$$
l(x,y):=(mn-m-n)+M_{\w}-nx-my.
$$ 
The function $l(x,y)$ is chosen in such a way that a box is labeled by $M_\w$ if and only if its NE corner touches the line containing the NW-SE diagonal of the rectangle $R_{m,n},$ so $l(x,y)\ge M_{\w}$ if and only if the box $(x,y)$ is below this line. The  Young diagram $D_{\w}$ defined by 
$$
D_{\w}:=\{(x,y)\in R_{m,n}\mid  l(x,y) \in \Delta_{\w}\}.
$$
If $(x,y)\in D_{\w},$ then $\omega(l(x,y))>0$, hence
$$
\omega(l(x-1,y))=\omega(l(x,y)+n)>0,
$$
and
$$
\omega(l(x,y-1))=\omega(l(x,y)+m)>\omega(l(x,y))>0.
$$
Therefore, if $x-1\ge 0,$ then $(x-1,y)\in D_{\w},$ and if $y-1\ge 0,$ then $(x,y-1)\in D_{\w}$. We conclude that $D_{\w}\subset R_{m,n}$ is indeed a Young diagram with the SW corner box $(0,0).$ Note also that $D_\w$ fits under the NW-SE diagonal of $R_{m,n}.$ Therefore, $D_w\in\Ymn.$

Observe that the boxes in the $i^{\mathrm{th}}$ row of the diagram correspond to
coordinates with $y = i-1$.

The row-labeling $\tau_{\w}$ is given by $\tau_{\w}(i)=\w(a_i),$ where $a_i$ is the label on the rightmost box of the $i$th row of $D_{\w}$ (if a row has length $0$ we take the label on the box $(-1,i-1),$ just outside the rectangle in the same row). Note that if $i$th and $(i+1)$th rows have the same length, then $a_{i+1}=a_i-m$ and 
$$
\tau_\w(i+1)=\w(a_{i+1})=\w(a_i-m)<\w(a_i)=\tau_\w(i).
$$  
Therefore, $(D_{\w},\tau_{\w})\in\lYmn$. We define $\An(\w)\in\PF_{m/n}$ to be the parking function corresponding to $(D_\w,\tau_\w).$

\begin{example}\label{Example: An}
Let $n=4$, $m=7$. Consider the affine permutation
$\w = [0,6,3,1] = s_1 s_0 s_2 s_3 s_2:$ 
$$
\begin{array}{ccccccccccccccc}
x &\ldots&-3&-2&-1&0&1&2&3&4&5&6&7&8&\ldots\\
\omega(x) &\ldots&-4&2&-1&-3&0&6&3&1&4&10&7&5&\ldots,\\
\end{array}
$$
The inversion set is 
$\Inv(\w) = \{ (2,3), (2,4), (2,5), (2,8), (3,4) \}$.
Note that there are no inversions of height $7,$ so $\omega$ is $7$-stable.
Equivalently, $\w^{-1} = [4,-2,3,5]$ is $7$-restricted. The set $\Dw=\{-2,2,3,4,\ldots\}$  is invariant under the addition of $4$ and $7,$
and $M_{\w}=-2.$ The diagram $D_{\w}$ is shown in Figure \ref{Figure: An}.
Note that the labels $3,4,5,$ and $-2$ on the rightmost boxes of the rows of $D_{\w}$
 are the $4$-generators of the set $\Delta_{\w},$ i.e. they are the smallest numbers in $\Delta_{\w}$ in the corresponding congruence classes $\mod 4.$ It follows then that the corresponding values $\w(3),\w(4),\w(5),$ and $\w(-2)$ are a permutation of $1,2,3,4.$
Indeed, read bottom  to top, $(\w(3),\w(4),\w(5),\w(-2))=(3,1,4,2).$ This defines the row-labeling $\tau_{\omega}:=[3,1,4,2].$ Note that the last (top) two rows of the diagram have the same length $0.$ Therefore, the difference between the corresponding labels is $5-(-2)=7.$ The $7$-stability condition then implies that $\tau(3)=\w(5)>\w(-2)=\tau(4),$ which is exactly the required monotonicity condition on the labeling. Using the bijection from Lemma \ref{Lemma:PF <-> labeled diagrams}, one obtains the parking function $\An_{\w}=\pw{2040}.$
\begin{figure}
\begin{center}
\begin{tikzpicture}
\draw [step=0.5, gray,thick] (0,0) grid (3.5,2);
\draw [step=0.5, gray,dashed] (-0.5,-0.5) grid (4,2.5);

\draw (-0.25,0.25) node {\footnotesize $19$};
\draw (-0.25,0.75) node {\footnotesize $12$};
\draw (-0.25,1.25) node {\footnotesize $5$};
\draw (-0.25,1.75) node {\footnotesize $-2$};
\draw (0.25,0.25) node {\footnotesize $15$};
\draw (0.25,0.75) node {\footnotesize $8$};
\draw (0.25,1.25) node {\footnotesize $1$};
\draw (0.25,1.75) node {\footnotesize $-6$};
\draw (0.75,0.25) node {\footnotesize $11$};
\draw (0.75,0.75) node {\footnotesize $4$};
\draw (0.75,1.25) node {\footnotesize $-3$};
\draw (1.25,0.25) node {\footnotesize $7$};
\draw (1.25,0.75) node {\footnotesize $0$}; 
\draw (1.25,1.25) node {\footnotesize $-7$};
\draw (1.75,0.25) node {\footnotesize $3$};
\draw (1.75,0.75) node {\footnotesize $-4$};
\draw (2.25,0.25) node {\footnotesize $-1$};
\draw (2.25,0.75) node {\footnotesize $-8$};
\draw (2.75,0.25) node {\footnotesize $-5$};
\draw (3.25,-0.25) node {\footnotesize $-2$};
\draw (3.25,0.25) node {\footnotesize $-9$};

\filldraw [fill=gray] (0,0)--(0,1)--(1,1)--(1,0.5)--(2,0.5)--(2,0)--(0,0);
  
\draw (-0.25,0.25) node {\footnotesize $19$};
\draw (-0.25,0.75) node {\footnotesize $12$};
\draw (-0.25,1.25) node {\footnotesize $5$};
\draw (-0.25,1.75) node {\footnotesize $-2$};
\draw (0.25,0.25) node {\footnotesize $15$};
\draw (0.25,0.75) node {\footnotesize $8$};
\draw (0.25,1.25) node {\footnotesize $1$};
\draw (0.25,1.75) node {\footnotesize $-6$};
\draw (0.75,0.25) node {\footnotesize $11$};
\draw (0.75,0.75) node {\footnotesize $4$};
\draw (0.75,1.25) node {\footnotesize $-3$};
\draw (1.25,0.25) node {\footnotesize $7$};
\draw (1.25,0.75) node {\footnotesize $0$}; 
\draw (1.25,1.25) node {\footnotesize $-7$};
\draw (1.75,0.25) node {\footnotesize $3$};
\draw (1.75,0.75) node {\footnotesize $-4$};
\draw (2.25,0.25) node {\footnotesize $-1$};
\draw (2.25,0.75) node {\footnotesize $-8$};
\draw (2.75,0.25) node {\footnotesize $-5$};
\draw (3.25,-0.25) node {\footnotesize $-2$};
\draw (3.25,0.25) node {\footnotesize $-9$};

\draw [blue] (-0.75,1.75) node {$2$};
\draw [blue] (-0.75,1.25) node {$4$};
\draw [blue] (-0.75,0.75) node {$1$};
\draw [blue] (-0.75,0.25) node {$3$};
\end{tikzpicture}
\caption{The labeled diagram corresponding to the permutation $\w = [0,6,3,1].$}
\label{Figure: An}
\end{center}
\end{figure}
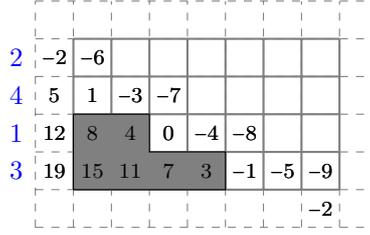

Equivalently, one can start directly from $\w^{-1} = [4,-2,3,5] \in \maSn$ 
and form the same labeled rectangle, noting $M_\w = \min\{ \wi(i) \mid 1 \le i \le n \}$.
Then $\An_\w(i) = 1 +x$ where $x$ is the $x$-coordinate of the box labeled
$\wi(i)$.
\end{example}

Alternatively, one can define the map $\An$ in a more compact, but less pictorial way:

\begin{definition}
Let $\omega\in\aSmn.$ We define the corresponding parking function $\An_{\omega}$ as follows. Let 
$$
M_{\omega}:=\min\{i\in\mathbb Z: \omega(i)>0\}.
$$
Given $\alpha\in\{1,\ldots,n\},$ there is a unique way to express $\omega^{-1}(\alpha)-M_{\w}$ as a linear combination $rm-kn$ with the condition $r\in\{0,\ldots,n-1\}.$ Note that one automatically gets $k\ge 0.$ Indeed, otherwise 
$$
\alpha =\omega(M_{\w}+rm-kn)\ge \omega(M_{\w})-kn>-kn\ge n,
$$
which contradicts the assumption $\alpha\in\{1,\ldots,n\}.$ We set $\An_{\omega}(\alpha):=k.$
\end{definition} 

\begin{lemma}
\label{an an}
The two above definitions of the map $\An$ are equivalent.
\end{lemma}

\begin{proof}
Let $(x,i-1)$ be the rightmost box in the $i$th row of the diagram $D_{\omega}.$ Let 
$$
\alpha:=\omega(l(x,i-1))=\omega(M_{\w}+(mn-m-n)-xn-(i-1)m)=\tau_{\omega}(i).
$$ 
We need to check that $\An_{\omega}(\alpha)$ is equal to the length of the $i$th row, which is $x+1.$ Indeed,
$$
\omega^{-1}(\alpha)=M_{\w}+(mn-m-n)-xn-(i-1)m,
$$ 
or
$$
\omega^{-1}(\alpha)-M_{\w}=(n-i)m-(x+1)n,
$$
with $n-i\in\{0,\ldots n-1\}.$ Therefore, $\An_{\omega}(\alpha)=x+1.$ 
\end{proof}

\begin{theorem}
The map $\An:\aSmn\to \Pmn$ is a bijection.
\end{theorem}

\begin{proof} 
Injectivity of the map $\An:\aSmn\to \Pmn$ is immediate from the construction. Indeed, the diagram $D_{\omega}$ completely determines the set $\Delta_{\omega},$ while the row-labeling $\tau_{\omega}$ determines the values of $\omega$ on the $n$-generators of $\Delta_{\omega},$ which suffices to determine $\omega.$ This gives an injective map $\phi:\Pmn\to \aSn,$ such that $\phi\circ\An=id_{\aSmn}.$ To prove that $\An$ is also surjective, it suffices to show that $\phi(f)$ is $m$-stable for any $f\in\Pmn.$

Indeed, let $f\in\Pmn$ be a parking function, $(D_f,\tau_f)$ be the corresponding Young diagram with row labeling, and $\omega:=\phi(f).$ Suppose that $(l,l+m)$ is
 an inversion of height 
$m,$ i.e. $\omega(l)>\omega(l+m).$ By shifting $l$ by a multiple of $n$ if necessary, one can assume that $\omega(l)\in\{1,\ldots,n\},$ so that $l$ labels the rightmost box of one of the rows of $D_f.$ Suppose that $l$ labels the box $(x,y)$ with $y>0,$ i.e. it is not in the first row. Then $l+m=l(x,y-1),$ which is the label on the box just below the box $(x,y).$ Since $D_f$ is a Young diagram, $(x,y-1)\in D_f.$ Suppose that $(z,y-1)\in D_f$ is the rightmost box in the $y$th row. Then
$$
\omega(l+m)=\omega(l(z,y-1)+(z-x)n)=\tau_f(y)+(z-x)n.
$$
If $z>x,$ we get $\omega(l+m)>n>\omega(l).$ Contradiction. If $z=x,$ then
the $y$th and $(y+1)$th rows are of the same length and, by the condition on the row labeling $\omega(l+m)=\tau_f(y)>\tau_f(y+1)=\omega(l).$ Contradiction.

Suppose now that $l$ labels the rightmost box in the first row: $l=l(x,0)=M+mn-m-n-xn$. Then $l+m=M+(m-1-x)n$ and 
$$
\omega(l+m)=\omega(M)+(m-1-x)n>(m-1-x)n\ge n,
$$ 
because $M$ labels the rightmost box in the $n$th (top) row and, therefore, $\omega(M)=\tau_f(n)>0,$ and $m-1-x \ge 1,$ because the first row of the diagram $D_f$ has length $x+1,$ which
has to be less than $m.$ We conclude that $\omega(l+m)>n\ge \omega(l).$ Contradiction. Therefore, $\omega\in\aSmn$ and $\An:\aSmn\to \Pmn$ is a bijection. 
\end{proof}

We call $\An$ the Anderson map, since if we restrict the domain to minimal length right coset representatives (which correspond to partitions called $(m,n)$-cores), and then project to increasing parking functions by sorting, the map agrees with one constructed by Anderson \cite{An02}.

\begin{remark}
For $\wm$ as in Lemma \ref{isometry} 
we have $\An_{\wm} = \pw{00\cdots 0}$.
\end{remark}

\begin{example}
Consider the case $n=5, m=3$.
Let $\id = [1,2,3,4,5]$, $s_1 = [2,1,3,4,5]$, $s_1 s_2 = [2,3,1,4,5]$
which are all $3$-restricted.  The images of their inverses under the
Anderson map are then $f = \An_{\id} = \pw{01201}$,
$\An_{s_1}  = \pw{10201} = f \circ s_1$, and 
$\An_{s_2 s_1}  = \pw{12001} = f \circ s_1s_2$.
Indeed the $3$-restricted permutations in the shuffle $14 \shuffle 25 \shuffle
3$ correspond in a similar manner to the entire finite $\Snn{5}$ orbit of
$f$.
The precise statement is in the following proposition.
\end{example}
\begin{proposition}
\label{prop-anderson-symmetry}
Let $\w \in \aSmn$, $f = \An_\w \in \Pmn$, and write $u = \wi \in \maSn$.
 Let $H = \{h \in \Sn \mid f \circ h = f\}$. 
Let $\mathcal{V} = \{ v \in \Sn \mid uv \in \maSn\}$.
If $v \in \mathcal{V}$ then 
$$\An_{v^{-1}\w} = \An_{\w} \circ v.$$
In particular
$\mathcal{V}$ are a complete set of coset representatives for 
$\Sn/H$.

\end{proposition}
\begin{proof}
We first consider the case $v = s_i$.
Note $ u s_i = [u(1), \cdots, u(i+1), u(i), \cdots, u(n)]$,
so in particular as $u, us_i \in \maSn$, $| u(i) = u(i+1)| \neq m$;
and in fact this difference cannot be a multiple of $m$.
We observed at the end of Example \ref{Example: An}
that $\An_\w(k) = 1+x$ where $(x,y)$ are the coordinates of the
box labeled $u(k)$. Multiplication by $s_i \in \Sn$
does not change $M_\w$ nor the function that labels the boxes of the rectangle.
Hence
$\An_{s_i\w}(k) = 1+x$ where $(x,y)$ are the coordinates of the
box labeled $us_i(k)$.
In other words
$$
\An_{s_i \w}(k) =
\begin{cases}
		\An_\w(k) & k \neq i, i+1 \\
		\An_\w(i+1) & k = i  \\
		\An_\w(i) & k = i+1
\end{cases}
\qquad = \An_\w \circ s_i = f \circ s_i.
$$
Further, note that $\An_\w \circ s_i \neq \An_\w$, otherwise the
boxes labeled $u(i), u(i+1)$ would be in the same column and hence
differ by a multiple of $m$.
Conversely, given $1\le i \le n$ with $f \circ s_i \neq f$, we see
$u, u s_i \in \maSn$.

Hence we may assume $f$ is chosen so that $H$ is a standard parabolic
subgroup, in which case the set $\mathcal{V}$ will correspond to
minimal length coset representatives.
Indeed, repeating the above argument, we see that for any
$v \in \Sn$ with $uv \in \aSmn$ and $\ell(uv) = \ell(u) + \ell(v)$
that $\An_{v^{-1}w} = \An_\w \circ v$.
\end{proof}

\subsection{Map $\PS:\aSmn\to \Pmn$}

\begin{definition}
Let $\omega\in \aSmn$.
Then the map $\PS_{\omega}:\{1,\ldots,n\}\to \BZ$ is given by:
$$
\PS_{\omega}(\alpha):=\sharp\left\{\beta  \mid  \beta >\alpha,
0 < \omega^{-1}(\alpha) -  \omega^{-1}(\beta)<m \right\}
$$
$$
=\sharp\left\{i  \mid  \omega(i) >\alpha,
\omega^{-1}(\alpha)-m < i<\omega^{-1}(\alpha) \right\}.
$$
In other words, $\PS_{\w}(\alpha)$ is equal to the number of inversions $(i,j)\in\Inv(\w)$ of height less than $m$ and such that $\w(j)\equiv\alpha$ mod $n.$
\end{definition}

\begin{definition}
Let $\SP: \maSn \to \Pmn$ be defined by $\w \mapsto \PS_{\w^{-1}}$.
\\
Observe $\SP_u(i) = \sharp \{ j > i \mid 0 < u(i) - u(j) < m \}$.
\end{definition}

\begin{example}
\label{example-PS}
Using the same permutation as in Example \ref{Example: An}, one gets
$\PS_{\omega}(1)=3 = \sharp \{(2,4), (3,4), (-2,4) \},$
$\PS_{\omega}(2)=0,$ $\PS_{\omega}(3)=1 = \sharp \{(2,3)\},$ and
$\PS_{\omega}(4)=1 = \sharp \{(2,5)\},$ so  
$\PS_{\omega} = \pw{3011}$.

Likewise, $\wi = [4, -2, 3, 5]$ and $\SP_\wi(1) = 3 =\sharp \{(1,2), (1,3), (1,6) \}
\subseteq \Inv(\wi)$,
 $ 1 = \sharp \{(3,6)\},$ and
$ 1 = \sharp \{(4,6)\}.$
 
\end{example}

\begin{example}
Consider $(n,m) = (5,3)$ and $u = [0,3,6,2,4]$. 
Then $\Inv(u) = \{ (2,4), (3,4),   (3,5), (3,6) \}$ and so
  $\SP_u = \pw{01200}$.  Note $u(3) - u(4) = 4 > m$ so this inversion
does not contribute to $\SP_u(3).$ 
\end{example}

Let us prove that $\PS_{\omega}$ is indeed an $m/n$-parking function. We will need the following definition and lemmas.

\begin{definition}
A subset $K\subset \BZ$ is called 
{\em $\pm n$-invariant}
if for all $x\in K$ one has $x+n\in K$ and $x-n\in K$. 
\end{definition}

\begin{lemma}
\label{lem:circle}
Let $K$ be an $\pm n$-invariant set, and $\sharp\left(K\cap [1,n]\right)=k$.
Then there exists $i\in \BZ$ such that 
$$\sharp\left(K\cap [i-m+1,i]\right)\le \frac{km}{n}.$$
\end{lemma}

\begin{proof}
Consider an interval $I$ in $\BZ$ of length $mn$. On one hand, it is covered by $m$ intervals of length $n$, containing $k$ points of $K$ each, hence $\sharp(K\cap I)=km$. On the other hand, it is covered by $n$ intervals of length $m$, hence one of these intervals should contain at most $\frac{km}{n}$ points of $K$.
\end{proof}

\begin{lemma}
\label{lem:step}
Let $\omega\in \aSmn$, let $K$ be an $\pm n$-invariant set, and $\sharp\left(K\cap [1,n]\right)=k$.
There exists $l\in \BZ\setminus K$ such that the following conditions hold:
\begin{itemize}
\item[a)] If $j<l$ and $\omega(j)>\omega(l)$ then $j\in K$

\item[b)] $\sharp\{j\in K: l-m<j<l, \omega(j)>\omega(l)\}\le \frac{km}{n}.$
\end{itemize}
\end{lemma}

\begin{proof}
By Lemma \ref{lem:circle} there exists $i\in \BZ$ such that 
$\sharp\left(K\cap [i-m+1,i]\right)\le \frac{km}{n}$. 
Since $\omega(x+n)=\omega(x)+n$, the set of values of $\omega$ on the half line $(-\infty,i]$
is bounded from above. Let us choose $l\le i$ such that:
\begin{equation}
\label{eq:omegamin}
\omega(l)=\max\{\omega(x):x\in (-\infty,i]\setminus K\},
\end{equation}
and prove that this $l$ satisfies (a) and (b).
If $j<l$ and $j\notin K$ then by \eqref{eq:omegamin} one has $\omega(j)<\omega(l)$, hence (a) holds.
To prove (b), define $$J_m(l,K):=\{j\in K: l-m<j<l, \omega(j)>\omega(l)\}.$$
Given $j\in J_m(l,K)$, there exists a unique $\alpha(j)\in \ZZ$ such that
$i-m< j+\alpha(j)m\le i$, and for different $j$ the numbers $j+\alpha(j)m$ are all different.
Since $\omega$ is $m$-stable, we have
$$\omega(l)<\omega(j)<\omega(j+m)<\ldots<\omega(j+\alpha(j)m).$$
By \eqref{eq:omegamin} we conclude that $j+\alpha(j)m\in K$.

Therefore we constructed an injective map from $J_m(l,K)$ to $K\cap [i-m+1,i]$, and
$$\sharp J_m(l,K)\le \sharp\left(K\cap [i-m+1,i]\right)\le \frac{km}{n}.$$
\end{proof}

\begin{theorem}
\label{theorem-image-PS}
For any $m$-stable affine permutation $\omega,$ the function $\PS_{\omega}$ is an $m/n$-parking function. Thus one gets a map $\PS:\aSmn\to\Pmn.$
\end{theorem}

\begin{proof}
Let us construct a chain of $\pm n$-invariant subsets 
$$
\emptyset=K_0\subset K_1\subset \ldots K_n=\Z,
$$
with $\sharp\{K_i\cap [1,n]\}=i$ for all $i,$
by the following inductive procedure. Given $K_i$ for some $i,$
we use Lemma \ref{lem:step} to find
an integer
$l_{i+1}$ satisfying Lemma \ref{lem:step}\ (a,b). Since $K_i$ was $\pm n$-invariant and $l_{i+1}\notin K_i$,
the sets $K_i$ and $l_{i+1}+n\BZ$ do not intersect, hence we can set $K_{i+1}:=K_i\sqcup (l_{i+1}+n\BZ).$ By shifting the number $l_{i+1}$ by a multiple of $n$ if necessary,
we can assume that $\omega(l_{i+1})\in\{1,\ldots,n\}$ for all $i\in\{0,\ldots,n-1\}.$ Note that $\{\w(l_1),\w(l_2),\ldots,\w(l_n)\}=\{1,\ldots,n\}.$ 

Let us estimate $\PS_{\omega}(\omega(l_{i+1})).$ If $l_{i+1}-m< j<l_{i+1}$ and $\omega(j)>\omega(l_{i+1})$, then by Lemma \ref{lem:step}(a) 
we have $j\in K_i$ and by Lemma \ref{lem:step}(b) the number of such $j$ is at most $\frac{im}{n}$. Therefore, 
$$
\sharp \{\alpha: \PS_{\omega}(\alpha)\le\frac{im}{n}\}\ge i+1,
$$
because for any $k\in\{0,1,\ldots,i\}$ one has $\PS_\w(\omega(l_{k+1}))\le \frac{km}{n}\le \frac{im}{n}.$ Therefore, $\PS(\omega)$ is an $m/n$-parking function.
\end{proof}

\begin{conjecture}\label{Conjecture: bijectivity}
The map $\PS:\omega\mapsto \PS_{\omega}$ is a bijection between $\aSmn$ and $\Pmn.$
\end{conjecture}

In the special cases $m = kn \pm 1,$ we prove that $\PS$ is a bijection in the next Section.

It is convenient to extend the domains of the functions $\PS_{\omega}$ and $\SP_\w$
 to all integers by using exactly the same formula.
Note that in this case $\PS_{\w}(\alpha+n)=\PS_{\w}(\alpha).$
We have the following results, which should be considered as steps towards Conjecture \ref{Conjecture: bijectivity}:

\begin{proposition}
\label{thm-inc}
Let $\w \in \maSn$ and
let $1 \le i  \le n$, \,  $i < j$.
\begin{enumerate}
\item
\label{item-one}
$\w(i) < \w(i+1) \iff \SP_\w(i) \le \SP_\w(i+1)$
\item
\label{item-two}
$(i,j)\in \Inv(\w) \implies \SP_\w(i) > \SP_\w(j)$
\end{enumerate}
\end{proposition}

\begin{proof}
\omitt{
It is convenient to note $\SP_\w(i) = \sharp\{j \mid i < j, \, 0 < \w(i) - \w(j) < m  \}$.
}

We first show that if $(i,j) \in \Inv(\w)$, then $\w$ has a unique
inversion $(i,J)$ with
\begin{equation}
\label{eq-Jm}
 j \le J, \, \w(j) \equiv \w(J) \mod m, \text{ and } 0 < \w(i) - \w(J) < m.
\end{equation}
Since $\w \in \maSn$, in the list $\w(1), \w(2), \cdots $, we have that $\w(j)$ occurs
to the left of $\w(j) + rm$ for all $r \ge 1$. 
Hence, we can pick $r \ge 0$ such that $rm < \w(i) - \w(j) < (r+1)m$, i.e.
$0 < \w(i) - (\w(j) + rm)  < m$,
and set $J = \w^{-1}(\w(j) + rm)$.

Now suppose $\w(i) < \w(i+1)$.
Let $(i,j) \in \Inv(\w)$ with $0 < \w(i) - \w(j) < m$.
Then since $\w(i+1) > \w(i) > \w(j)$ we also have $(i+1,j) \in \Inv(\w)$.
Observe $i+1 < j$ as $(i,i+1) \nin \Inv(\w)$.
We pick $(i+1, J) \in \Inv(\w)$ as in \eqref{eq-Jm} above.
The map $(i,j) \mapsto (i+1, J)$ is clearly an injection, yielding
$\SP_\w(i) \le \SP_\w(i+1)$.

For ease of exposition, we recall Remark \ref{remark-inversions-equivalent}
which lets us equate an inversion $(i,j)$ with $(i+tn, j+tn)$.

Next if $i<j$ with $\w(i) > \w(j)$, suppose we have $(j,k) \in \Inv(\w)$
with $0 < \w(j) - \w(k) < m$.  Then $(i,k) \in \Inv(\w)$ too. We can pick
$K \ge k$ according to \eqref{eq-Jm} yielding $(i,K) \in \Inv(\w)$ with
$0 < \w(i) - \w(K) < m$ and $\w(k) \equiv \w(K) \mod m$.  Again the map
$(j,k) \mapsto (i,K)$ is an injection, yielding $\SP_\w(j) \le \SP_\w(i)$. 
Further there is an extra inversion
of the form $(i,J) \in Inv(\w)$, showing $\SP_\w(i) > \SP_\w(j)$.
The case $j=i+1$ gives the converse of \eqref{item-one}.
\end{proof}

Note that if $(n, n+1) \in \Inv(\w)$ then the above proposition implies
$\SP_\w(n) \le \SP_\w(1)$, as we have $\SP_\w(1) = \SP_\w(n+1)$ by our
convention.  As a consequence of this, we have the following corollary.
\begin{corollary}
Let $1 \le i, j \le n$.
$\SP_\w(i) = \SP_w(j) \implies |\w(i) - \w(j) | < n$
\end{corollary}
\begin{proof}
Without loss of generality $ i< j$. 
By Proposition \ref{thm-inc} item \eqref{item-two}, $\w(i) < \w(j)$.
If also $\w(i) + n < \w(j)$ then as $j < i+n$, $(j, i+n) \in \Inv(\w)$
so by the proposition
$\SP_\w(j) > \SP_\w(i+n) = \SP_\w(i)$ which is a contradiction.
\end{proof}

\begin{proposition}
Let $\w \in \maSn$.
\begin{enumerate}
\item
If\  $0 < \w(i) - \w(i+1) < m,$\ \  then\ \  
$
\begin{cases}
\SP_{\w s_i}(i) = \SP_\w (i+1), \\
\SP_{\w s_i}(i+1) = \SP_\w (i) -1, \\
\SP_{\w s_i}(j) = \SP_\w (j)\ \mbox{for}\  j \not\equiv i, i+1\mod n 
\end{cases}
$
\item If\  $m <  \w(i) - \w(i+1),$\ \  then\ \  
$
\begin{cases}
\SP_{\w s_i}(i) = \SP_\w (i+1), \\
\SP_{\w s_i}(i+1) = \SP_\w (i), \\
\SP_{\w s_i}(j) = \SP_\w (j)\ \mbox{for}\  j \not\equiv i, i+1\mod n 
\end{cases}
$
\end{enumerate}
\end{proposition}

\begin{proof}
Write $u = w s_i$.  Since $0 < \w(i) - \w(i+1)$, we have $\Inv(u) = 
s_i(\Inv(\w)) \setminus \{(i,i+1)\}$.
In other words, 
$(i,j) \in \Inv(u)$ iff $(i+1,j) \in \Inv(\w)$.
Since $u(i)-u(j) = w(i+1)-w(j)$ this yields $\SP_{u}(i) = \SP_\w (i+1)$
 Similarly,
for $k \neq i, i+1$, $(k,j) \in \Inv(u)$ iff $(k,j) \in \Inv(\w)$,
yielding $\SP_{u}(k) = \SP_\w (k)$.

Finally
$(i+1,j) \in \Inv(u)$ iff  $j \neq i+1$ and $(i,j) \in \Inv(\w)$.
Again
$u(i+1)-u(j) = w(i)-w(j)$.
Hence in the case $\w(i) - \w(i+1) < m$, so that $(i,i+1) \in \Inv(\w)$ contributes
to $\SP_\w(i)$, we see $\SP_\w(i) = \SP_u(i+1) +1$.
When $m < \w(i) - \w(i+1)$ this inversion does not contribute,
 so $\SP_\w(i) = \SP_u(i+1)$.
\end{proof}

As a corollary to this proposition, we see that $\SP$  is injective
on $\{\w \in \maSn \mid (i,j) \in \Inv(\w) \implies \w(i)-\w(j) < m \}$. Another interpretation of Theorem \ref{thm-inc} is that $\SP$ not only respects descents but also respects (weakly) increasing subsequences. 

\begin{example}
Let $(n,m) = (3,4)$.  Consider these three affine permutations 
$y = [1,5,0], \, \w  = y s_2 = [1,0,5], $ and $ \w s_1 = [0,1,5] 
\in \maSnn{3}{4}$.
The corresponding parking functions are
$\SP_y = \pw{120}, \SP_w = \SP_{ys_2} = \pw{102}$ (note $5-0 > 4$ so their
second and third values have swapped), and
$\SP_{\w s_1} = \pw{002}$ (note $1-0 < 4$).
\end{example}

\begin{remark}
The maps $\PS$ and $\SP$ preserve a kind of cyclic symmetry, as follows. (Compare this
to Proposition \ref{prop-anderson-symmetry} for the map $\An$.)
Let the shift operator $\pi : \Z \to \Z$
be defined by 
$$
 \pi(i) = i+1.
 $$
Clearly $\pi(i + tn) = \pi(i) + tn$, but $\pi \nin \aSn$ as
$\sum_{i=1}^n \pi(i) = \frac{n(n+3)}{2}$.
(In other contexts, $\pi$ lives in the {\em extended} affine symmetric group
$P \rtimes \Sn \supsetneqq Q \rtimes \Sn \simeq \aSn$. It corresponds to
the generator of $P/Q$ where $P$ and  $Q$ are the weight and root lattices of type $A$,
respectively.)
The conjugation map $\aSn \to \aSn$,
$\w \mapsto \pi \w \pi^{-1}
$
interacts nicely with the maps $\SP$ and $\PS$.

In window notation, conjugation by $\pi$  corresponds to sliding the ``window"
one unit to the left, but then renormalizing so the sum of the entries is
still $\frac {n(n+1)}{2}$, i.e. 
$\pi \w \pi^{-1} = [\w(0) +1, \w(1)+1, \cdots, \w(n-1) +1];$
equivalently $\pi \w \pi^{-1}(i) = \w(i+1) +1$.
It is clear that $\Invall(\pi \w \pi^{-1}) = \{(i+1,j+1) \mid (i,j) \in \Invall(\w) \}$,
and so conjugation by $\pi$ preserves heights of inversions.
In particular, it preserves both  sets $\maSn$ and $\aSmn$.
It is also clear from the definition of $\SP$ that
\begin{align*}
&\SP_{\pi u \pi^{-1}}(i+1) = \SP_u(i)&  & \text{for } u \in \maSn  \text{ \, and hence} \\
&\PS_{\pi \w \pi^{-1}} (i+1) = \PS_\w(i)& & \w \in \maSn.
\end{align*}
Consider Example
\ref{example-PS}, for which $u=[4,-2,3,5]$ and $SP_u = \pw{3011}$.
We get $\pi u \pi^{-1} = [2,5,-1,4]$ (for which $\Inv(\pi u \pi^{-1})
=\{(1,3), (2,3), (2,4), (2,7), (4,7)\}$)
 and $\SP_{\pi u \pi^{-1}}  = \pw{1301}$.
\end{remark}


\subsection{Two statistics}

Our work was partially motivated by some open questions posed by Armstrong in \cite{Ar11}.
He managed to describe $\area$ and $\dinv$ statistics on parking functions appearing in ``Shuffle Conjecture'' of \cite{HHLRU}
in terms of the Shi arrangement.

We present two natural generalizations of these statistics to the rational case. Both were introduced in a different form in \cite{H12},
but they can be best written in terms of maps $\An$ and $\PS$.

\begin{definition}
Let $\omega\in \aSmn$
 be an affine permutation labeling an alcove $\wi(\Afund)\in D_{n}^{m}$. We define:
\begin{equation}
\area(\omega):=\frac{(m-1)(n-1)}{2}-\sum \An_{\omega}(i),\ \dinv(\omega):=\frac{(m-1)(n-1)}{2}-\sum \PS_{\omega}(i).
\end{equation}
\end{definition}

\begin{proposition}
\label{area}
The statistics $\area(\omega)$ can be computed as follows. Recall that $\Dw:=\{i\in\BZ: \omega(i)>0\}$, then
$$
\area(\omega)=\sharp \left([\min \Dw,+\infty)\setminus \Dw\right).
$$
\end{proposition}

\begin{proof}
Indeed, there are $\frac{(m-1)(n-1)}{2}$ boxes in the rectangle $R_{m,n}$ below the diagonal. The ones labeled by the elements of the set $\left([\min \Dw,+\infty)\setminus \Dw\right)$ are in 1-to-1 correspondence with the boxes outside of the diagram of $\An(\omega)$, so their number equals $\frac{(m-1)(n-1)}{2}-\sum \An_i(\omega)=\area(\omega)$.
\end{proof}

One can also check that the statistic $\area$ agrees with the statistics $\ish^{-1}$ of \cite{Ar11}.

\begin{example}
For the fundamental alcove,
$\PS_\id(i)=0$,
so $\dinv(\id)=\frac{(m-1)(n-1)}{2}.$  
On the other hand, $\Delta_{\id}=\{1,2,3,\ldots\},$ so by Proposition \ref{area}
$\area(\id)=0$. 
\end{example}

\begin{example}
Consider the permutation $\wm=[m-c,2m-c\ldots,nm-c]\in \maSn$. Here the constant $c$ is uniquely determined by the condition $\sum\limits_{i=1}^n \wm(i)=\frac{n(n+1)}{2}.$ In fact, $c=\frac{(n+1)(m-1)}{2}.$ Let us compute $\SP_{\wm}=\PS_{\wm^{-1}}.$

\omitt{  
Let $1\le k\le n,$ then
\omitt{
$$
\PS_{\wm^{-1}}(k)=\sharp\{i\ge k, j\ge 0\ :\ \wm(k)-m<\wm(i-nj)<\wm(k)\}=
$$
$$
\sharp\{i\ge k, j\ge 0\ :\ (k-1)m-c<mi-c-nj<km-c\}
$$
For each $j$, there is at most one $i$ satisfying this inequality, and there is one iff
$$nj<(n-k)m\ \Leftrightarrow j<\frac{(n-k)m}{n}.$$
Therefore $\PS_{\wm^{-1}}(k)=\left\lfloor\frac{(n-k)m}{n} \right\rfloor,$ and
$
\dinv(\wm^{-1})=\frac{(m-1)(n-1)}{2}-\sum \PS_{\omega}(i)=0.
$
} 
$$
\PS_{\wm^{-1}}(k)=\sharp\{l>k\mid \wm(k)-m<\wm(l)<\wm(k)\}.
$$
For every $l\in\Z$ there is a unique way to express it as $l=i+nj,$ where $i\in\{1,\ldots,n\}.$ Therefore,
$$
\PS_{\wm^{-1}}(k)=\sharp\{(i,j)\mid i \in\{1,\ldots, n\},\ i+nj>k,\ \wm(k)-m<\wm(i+nj)<\wm(k)\}
$$
$$
=\sharp\{(i,j)\mid i \in\{1,\ldots, n\},\ i+nj>k,\ km-c-m<im-c+nj<km-c\}
$$
$$
=\sharp\{(i,j)\mid i \in\{1,\ldots, n\},\ i+nj>k,\ m(k-1)<im+nj<km\}.
$$
The condition $m(k-1)<im+nj<km$ implies that $j\neq 0.$ On the other side, $i+nj>k$ is always satisfied if $j>0$ and never satisfied if $j<0.$ Therefore, $i+nj>k$ can be replaced by $j>0.$ Also, for each $j$, there is at most one $i$ satisfying the inequality $m(k-1)<im+nj<km.$ Furthermore, such $i\in\{1,\ldots,n\}$ exists if and only if $m+nj<km$ and $nm+nj>m(k-1),$ or, equivalently, $m(k-1-n)<nj<m(k-1).$ Note that $m(k-1-n)<0,$ so $nj>m(k-1-n)$ is satisfied automatically. We conclude
$$
\PS_{\wm^{-1}}(k)=\sharp\{j\in\Z_{>0}\mid \ j<\frac{m(k-1)}{n}\}=\lfloor \frac{m(k-1)}{n} \rfloor.
$$
} 

Since the entries in the window notation for $\wm$ are increasing,
i.e. $\wm \in \aSn/\Sn$, if $(k,t) \in \Inv(\wm)$ this forces $t = i +jn$ for some
$1 \le i < k$ and $j \ge 1$. Since it is an inversion, we have $km-c > im-c+jn$.
To contribute to $\SP_{\wm}(k)$ we must have
\begin{eqnarray*}
 0 < (km-c) - (im-c+jn) < m
&\iff& km> im+jn > (k-1)m
\\
&\iff& (k-1-i)m < nj < (k-i)m 
\\
&\iff& \frac{(k-1-i)m}{n} < j < \frac{(k-i)m}{n}
\end{eqnarray*}
Hence 
$\SP_{\wm}(k) =
\sharp \{ j, i \mid  j \ge 1, 1 \le i < k, \frac{(k-1-i)m}{n} < j < \frac{(k-i)m}{n} \}$ 
Since we run over all $1 \le i < k$ this is just
$= \sharp \{ j  \mid  j \ge 1,  j < \frac{(k-1)m}{n} \} = \lfloor \frac{m(k-1)}{n} \rfloor.$
By Proposition \ref{thm-inc}, $\SP_{\wm}$ is weakly increasing.
\omitt{
It follows that the parking function $\PS_{\wm^{-1}}$ is (weakly) increasing,
}
The corresponding diagram is the maximal diagram that fits under the diagonal in an $m\times n$ rectangle. The area of such a diagram is $\frac{(n-1)(m-1)}{2},$ therefore $\dinv(\wm^{-1})=\frac{(m-1)(n-1)}{2}-\sum \PS_{\wm^{-1}}(i)=0.$ One can also check that $\An(\wm^{-1})=0$, so $\area(\wm^{-1})=\frac{(m-1)(n-1)}{2}$.
\end{example}

\begin{definition}
We define the {\em combinatorial Hilbert series} as the bigraded generating function:
$$
H_{m/n}(q,t)=\sum_{\omega\in \aSmn}q^{\area(\omega)}t^{\dinv(\omega)}.
$$
\end{definition}

It is clear that $H_{m/n}(1,1)=m^{n-1}$, since there are $m^{n-1}$ permutations in $\aSmn$.

\begin{conjecture} (cf. \cite{H12})
The combinatorial Hilbert series is symmetric in $q$ and $t$:
 $$H_{m/n}(q,t)=H_{m/n}(t,q).$$
\end{conjecture} 
 This conjecture is a special case of ``Rational Shuffle Conjecture'' \cite[Conjecture 6.8]{gn}.
A more general conjecture also implies  this identity
$$
H_{m/n}(q,q^{-1})=q^{-\frac{(m-1)(n-1)}{2}}(1+q+\ldots+q^{m-1})^{n-1}.
$$
Both conjectures are open for general $m$ and $n$.
 
\begin{example}
For $n=5$ and $m=2$, we have (see Example \ref{example: 2 5} below for details):
$$
H_{2/5}(q,t)=5+4(q+t)+(q^2+qt+t^2),
$$
and the above properties hold:
\begin{enumerate}
\item $H_{2/5}(1,1)=16=2^4,$ 
\item $H_{2/5}(q,t)=H_{2/5}(t,q),$
\item $H_{2/5}(q,q^{-1})=q^{-2}+4q^{-1}+6+4q+q^2=q^{-2}(1+q)^4.$
\end{enumerate}

\end{example}

\section{The cases $m=kn\pm 1$ and the Extended Shi Arrangements.}\label{Section: Pak-Stanley}

\subsection{Extended Shi Arrangements and Pak-Stanley Labeling.}

Recall the set of $k$-parking functions $\PF_k : = \PF_{(kn+1)/n}.$
Recall the hyperplanes $H_{ij}^k=\{
\x\in V
\mid x_i-x_j=k\}$
and the affine braid arrangement $\aBrn=\{H_{ij}^k\mid 1\le i,j\le n, k\in \Z\}$.
The {\it extended Shi arrangement}, or $k$-Shi arrangement \cite{Shi86,St96}, is 
defined as a subarrangement of the affine braid arrangement: 
\begin{definition}
The hyperplane arrangement 
$$
\Shikn:=\left\{H_{ij}^\l:1\le i<j\le n,\ -k<\l\le k \right\}
$$ 
is called the
{\em $k$-Shi arrangement}.
The connected components of the complement to $\Shikn$ are called {\it $k$-Shi regions}. The set of $k$-Shi regions is denoted $\Regkn.$
\end{definition}

One can use the notations introduced in Section \ref{Subsection: Affine Permutations} to rewrite the definition of the $k$-Shi arrangement as follows:
$$
\Shikn=\left\{H_{ij}^\l:1\le i<j\le n,\ -k<\l\le k \right\}
$$ 
$$
=\left\{H_{ij}^\l:1\le i<j\le n,\ -k<\l< 0 \right\}\sqcup\left\{H_{ij}^\l:1\le i<j\le n,\ 0\le\l\le k \right\}
$$
$$
=\left\{H_{i,j-n\ell}^0:1\le i<j\le n,\ -k<\l< 0 \right\}\sqcup\left\{H_{j,i+\ell n}^0:1\le i<j\le n,\ 0\le\l\le k \right\}
$$
$$
=\left\{H_{ij}^0:1\le i\le n, i<j<i+kn, j\nequiv i \mod n\right\}.
$$
In other words, the $k$-Shi arrangement consists of all hyperplanes of height less than $kn$ in the affine braid arrangement.   
The hyperplane $H_{ij}^\l$ divides $V$ into two half-spaces.
Let $H_{ij}^{\l,\prec}$ denote the half-space that contains $\Afund$ and $ H_{ij}^{\l,\succ}$ denote the complementary half-space.
Note that $H_{ij}^\l$ separates $\w(\Afund)$ from $\Afund$ iff  $\w(\Afund) \subseteq
H_{ij}^{\l,\succ}$ iff $(i, j-\l n)$ or $(j, i+\l n) \in \Inv(\w^{-1})$
(when taking the convention $i,j \in \{1, \ldots, n\}$).

\begin{definition}
The Pak-Stanley labeling is the map $\lambda:\Regkn\to\PF_k$,
$R \mapsto \lambda_R$  defined by the formula
\omitt{  huh? 
$$
\lambda_R(a)=\sharp\{H_{ij}^m\in H_R:m>0, j=a\}+\sharp\{H_{ij}^m:m\le 0, i=a\}.
$$
} 
$$
\lambda_R(a)=\sharp\{
H_{ij}^\l\in \Shikn \mid R \subseteq H_{ij}^{\l,\succ}, \l>0, i=a\}+
\sharp\{
H_{ij}^\l\in \Shikn \mid R \subseteq H_{ij}^{\l,\succ}, \l \le 0, j=a\}.
$$
In other words, one labels the fundamental alcove $\Afund$ by the parking function
$f=\pw{0\ldots 0},$ and then as one crosses the hyperplane $H_{ij}^\l$ in the positive direction (i.e. getting further away from $\Afund$),
one adds $1$ to $f(j)$ if $\l\le 0$ and adds $1$ to $f(i)$ if $\l>0.$ 
\end{definition}

\begin{remark}
One can rewrite this definition   
as follows:
$$
\lambda_R(a)
=\sharp\{
H_{ij}^0\in \Shikn \mid R \subseteq H_{ij}^{0,\succ}, a=i<j\}
=\sharp\{
H_{a \, a+t}^0 \mid R \subseteq H_{a\, a+t}^{0,\succ}, 0 < t < kn, t \not\equiv 0\mod n\}.
$$
\end{remark}
We illustrate the Pak-Stanley labeling for $n=3,$ $k=1$ ($m=4$) in Figure \ref{Figure: Pak-Stanley}.

\begin{figure}
\begin{center}
\begin{tikzpicture}[scale=0.4]
\filldraw [fill=black] (4,0) circle (0.2);

\draw (0,0)--(11,0);
\draw (0,2.6)--(11,2.6);

\draw (2,-3.46)--(7.5,6.06);
\draw (4.5,-4.33)--(10,5.2);

\draw (1,5.2)--(6.5,-4.33);
\draw (3.5,6.06)--(9,-3.46);

\draw (8.5,6.35) node {\footnotesize $H_{\mathbf{1},2}^0$};
\draw (11,5.5) node {\footnotesize $H_{\mathbf{2},4}^0$};

\draw (12,-0.5) node {\footnotesize $H_{\mathbf{2},3}^0$};
\draw (12,2) node {\footnotesize $H_{\mathbf{3},5}^0$};

\draw (8,-4.7) node {\footnotesize $H_{\mathbf{1},3}^0$};
\draw (10.5,-3.7) node {\footnotesize $H_{\mathbf{3},4}^0$};

\draw (5.5,0.75) node {\small $\pw{000}$};

\draw (7,1.7) node {\small $\pw{001}$};

\draw (4,1.7) node {\small $\pw{100}$};

\draw (5.5,-0.75) node {\small $\pw{010}$};

\draw (9,1) node {\small $\pw{011}$};
\draw (2,1) node {\small $\pw{200}$};
\draw (3.5,3.5) node {\small $\pw{101}$};
\draw (7.5,3.5) node {\small $\pw{002}$};
\draw (7,-2) node {\small $\pw{020}$};
\draw (4,-2) node {\small $\pw{110}$};

\draw (10,3.2) node {\small $\pw{012}$};
\draw (5.5,4.5) node {\small $\pw{102}$};
\draw (5.5,-4.2) node {\small $\pw{120}$};
\draw (9,-1) node {\small $\pw{021}$};
\draw (2,-1) node {\small $\pw{210}$};
\draw (1,3.2) node {\small $\pw{201}$};
\end{tikzpicture}
\caption{Pak-Stanley labeling for $1$-Shi arrangement for $n=3.$}
\label{Figure: Pak-Stanley}
\end{center}
\end{figure}
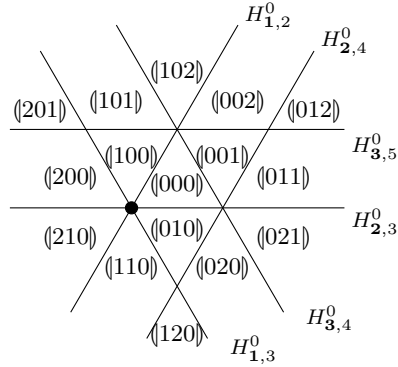


\begin{theorem}[\cite{St96}]
The map $\lambda:\Regkn\to\PF_k$ is a bijection. 
\end{theorem}

\subsection{Relation Between Sommers Regions and Extended Shi Arrangements for $m=kn\pm 1$}

Consider the case $m=kn+1.$ One can show that each region of an extended Shi arrangement contains a unique minimal alcove (i.e. an alcove with the least number of hyperplanes $H_{ij}^k$ separating it from the fundamental alcove $\Afund$). 

\begin{theorem}[\cite{FV1}]
\label{theorem:fv1}
An alcove $\omega(\Afund)$ is the minimal alcove of a $k$-Shi region if and only if $\omega^{-1}(\Afund)\subset D_n^{kn+1}.$
\end{theorem}
 
\begin{example}
We illustrate this theorem in Figure \ref{Figure: omega(A0) vs omega-1(A0)}, where the minimal alcoves of the $1$-Shi region are matched with the alcoves in the Sommers region $D_3^4.$ On the left we have the minimal alcoves $\omega(\Afund)$ labeled by the $m$-stable permutations $\omega\in \aSmn$ for $m=4, n=3$. On the right we have alcoves $\omega^{-1}(\Afund)$ that fit inside $D_3^4,$ labeled by the $m$-restricted permutations $\omega^{-1}\in\maSn.$ Note that $[-226]=[420]^{-1},$\ $[150]=[1-16]^{-1},$ and $[4-13]=[-253]^{-1}.$

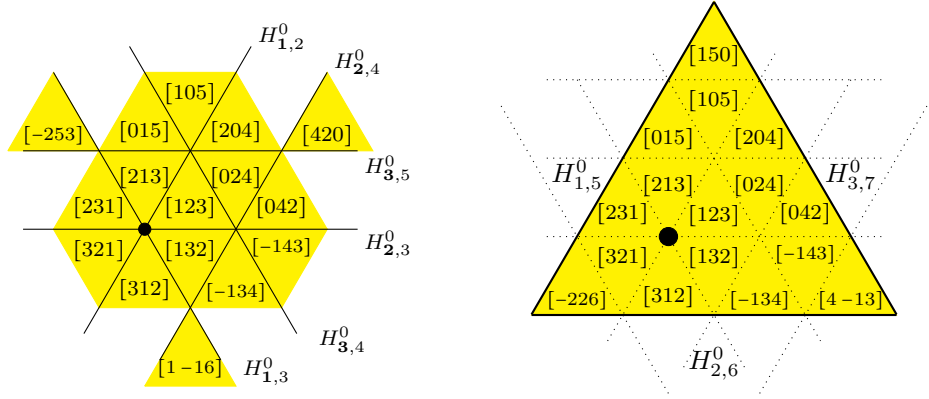
\begin{figure}
\begin{center}
\begin{tikzpicture}[scale=0.4]
\filldraw [yellow, fill=yellow] (4,0)--(7,0)--(5.5,2.6)--(4,0);
\filldraw [yellow, fill=yellow] (4,0)--(7,0)--(5.5,-2.6)--(4,0);
\filldraw [yellow, fill=yellow] (4,0)--(5.5,2.6)--(2.5,2.6)--(4,0);
\filldraw [yellow, fill=yellow] (7,0)--(5.5,2.6)--(8.5,2.6)--(7,0);
\filldraw [yellow, fill=yellow] (7,0)--(10,0)--(8.5,2.6)--(7,0);
\filldraw [yellow, fill=yellow] (1,0)--(4,0)--(2.5,2.6)--(1,0);
\filldraw [yellow, fill=yellow] (4,5.2)--(5.5,2.6)--(2.5,2.6)--(4,5.2);
\filldraw [yellow, fill=yellow] (7,5.2)--(5.5,2.6)--(8.5,2.6)--(7,5.2);
\filldraw [yellow, fill=yellow] (4,0)--(5.5,-2.6)--(2.5,-2.6)--(4,0);
\filldraw [yellow, fill=yellow] (7,0)--(5.5,-2.6)--(8.5,-2.6)--(7,0);
\filldraw [yellow, fill=yellow] (4,5.2)--(7,5.2)--(5.5,2.6)--(4,5.2);
\filldraw [yellow, fill=yellow] (8.5,2.6)--(11.5,2.6)--(10,5.2)--(8.5,2.6);
\filldraw [yellow, fill=yellow] (-0.5,2.6)--(2.5,2.6)--(1,5.2)--(-0.5,2.6);
\filldraw [yellow, fill=yellow] (1,0)--(4,0)--(2.5,-2.6)--(1,0);
\filldraw [yellow, fill=yellow] (7,0)--(10,0)--(8.5,-2.6)--(7,0);
\filldraw [yellow, fill=yellow] (4,-5.2)--(7,-5.2)--(5.5,-2.6)--(4,-5.2);
\filldraw [fill=black] (4,0) circle (0.2);

\draw (0,0)--(11,0);
\draw (0,2.6)--(11,2.6);

\draw (2,-3.46)--(7.5,6.06);
\draw (4.5,-4.33)--(10,5.2);

\draw (1,5.2)--(6.5,-4.33);
\draw (3.5,6.06)--(9,-3.46);

\draw (8.5,6.35) node {\footnotesize  $H_{\mathbf{1},2}^0$};
\draw (11,5.5) node {\footnotesize  $H_{\mathbf{2},4}^0$};

\draw (12,-0.5) node {\footnotesize $H_{\mathbf{2},3}^0$};
\draw (12,2) node {\footnotesize  $H_{\mathbf{3},5}^0$};

\draw (8,-4.7) node {\footnotesize  $H_{\mathbf{1},3}^0$};
\draw (10.5,-3.7) node {\footnotesize  $H_{\mathbf{3},4}^0$};

\draw (5.5,0.7) node {\footnotesize $[123]$};

\draw (7,1.7) node {\footnotesize $[024]$};

\draw (4,1.7) node {\footnotesize $[213]$};

\draw (5.5,-0.7) node {\footnotesize $[132]$};

\draw (8.5,0.7) node {\footnotesize $[042]$};
\draw (2.5,0.7) node {\footnotesize $[231]$};
\draw (4,3.25) node {\footnotesize $[015]$};
\draw (7,3.25) node {\footnotesize $[204]$};
\draw (7,-2.1) node {\scriptsize $[-134]$};
\draw (4,-2) node {\footnotesize $[312]$};

\draw (10,3.1) node {\footnotesize $[420]$};
\draw (5.5,4.5) node {\footnotesize $[105]$};
\draw (5.5,-4.6) node {\scriptsize $[1-\!16]$};
\draw (8.5,-0.6) node {\scriptsize $[-143]$};
\draw (2.5,-0.7) node {\footnotesize $[321]$};
\draw (1,3.1) node {\scriptsize $[-253]$};

\end{tikzpicture}
\ \ \ 
\qquad
\begin{tikzpicture}[scale=0.4]

\filldraw [yellow, fill=yellow] (4,0)--(7,0)--(5.5,2.6)--(4,0);
\filldraw [yellow, fill=yellow] (4,0)--(7,0)--(5.5,-2.6)--(4,0);
\filldraw [yellow, fill=yellow] (4,0)--(5.5,2.6)--(2.5,2.6)--(4,0);
\filldraw [yellow, fill=yellow] (7,0)--(5.5,2.6)--(8.5,2.6)--(7,0);
\filldraw [yellow, fill=yellow] (7,0)--(10,0)--(8.5,2.6)--(7,0);
\filldraw [yellow, fill=yellow] (1,0)--(4,0)--(2.5,2.6)--(1,0);
\filldraw [yellow, fill=yellow] (4,5.2)--(5.5,2.6)--(2.5,2.6)--(4,5.2);
\filldraw [yellow, fill=yellow] (7,5.2)--(5.5,2.6)--(8.5,2.6)--(7,5.2);
\filldraw [yellow, fill=yellow] (4,0)--(5.5,-2.6)--(2.5,-2.6)--(4,0);
\filldraw [yellow, fill=yellow] (7,0)--(5.5,-2.6)--(8.5,-2.6)--(7,0);
\filldraw [yellow, fill=yellow] (4,5.2)--(7,5.2)--(5.5,2.6)--(4,5.2);
\filldraw [yellow, fill=yellow] (1,0)--(4,0)--(2.5,-2.6)--(1,0);
\filldraw [yellow, fill=yellow] (7,0)--(10,0)--(8.5,-2.6)--(7,0);
\filldraw [yellow, fill=yellow] (4,5.2)--(5.5,2.6)--(2.5,2.6)--(4,5.2);
\filldraw [yellow, fill=yellow] (4,5.2)--(7,5.2)--(5.5,7.8)--(4,5.2);
\filldraw [yellow, fill=yellow] (10,0)--(11.5,-2.6)--(8.5,-2.6)--(10,0);
\filldraw [yellow, fill=yellow] (1,0)--(2.5,-2.6)--(-0.5,-2.6)--(1,0);

\filldraw [fill=black] (4,0) circle (0.3);

\draw[thick]  (-0.5,-2.6)--(11.5,-2.6);
\draw[dotted] (0,0)--(11,0);
\draw[dotted] (0,2.6)--(11,2.6);
\draw[dotted] (0,5.2)--(11,5.2);

\draw[thick] (-.5,-2.58)--(5,6.92)--(5.5, 7.783);
\draw[dotted] (2,-3.46)--(7.5,6.06);
\draw[dotted] (4.5,-4.33)--(10,5.2);
\draw[dotted] (7,-5.2)--(12.5,4.33);

\draw[dotted] (-1.5,4.33)--(4,-5.2);
\draw[dotted] (1, 5.2)--(6.5,-4.33);
\draw[dotted]  (3.5,6.06)--(9,-3.46);
\draw[thick] (5.5,7.783 )--(6,6.92)--(11.5,-2.58);

\draw (5.5,0.7) node {\footnotesize $[123]$};

\draw (7,1.7) node {\footnotesize $[024]$};

\draw (4,1.7) node {\footnotesize $[213]$};

\draw (5.5,-0.7) node {\footnotesize $[132]$};

\draw (8.5,0.75) node {\footnotesize $[042]$};
\draw (2.5,0.75) node {\footnotesize $[231]$};
\draw (4,3.25) node {\footnotesize $[015]$};
\draw (7,3.25) node {\footnotesize $[204]$};
\draw (7,-2.1) node {\scriptsize $[-134]$};
\draw (4,-2) node {\footnotesize $[312]$};

\draw (5.5,6) node {\footnotesize $[150]$};
\draw (5.5,4.5) node {\footnotesize $[105]$};
\draw (10,-2.1) node {\scriptsize $[4-\!13]$}; 
\draw (8.5,-0.6) node {\scriptsize $[-143]$};
\draw (2.5,-0.7) node {\footnotesize $[321]$};
\draw (1,-2.1) node {\scriptsize $[-226]$};

\draw (10,2) node {$H_{{3},7}^0$};
\draw (1,2) node {$H_{{1},5}^0$};
\draw (5.5,-4.2) node {$H_{{2},6}^0$};

\end{tikzpicture}
\caption{Minimal alcoves for $\Sh_3^{1}$ and Sommers region $D_{3}^{4}$.}
\label{Figure: omega(A0) vs omega-1(A0)}
\end{center}
\end{figure}   
\end{example}

Theorem \ref{theorem:fv1} and Lemma \ref{Lemma: Sommers region} imply a bijection $\alc:\aSn^{kn+1}\to \Regkn.$ 

\begin{theorem}\label{Theorem: PS bijectivity kn+1}
One has $\lambda\circ \alc=\PS$ in this case. 
In particular, $\PS$ is a bijection for $m=kn+1.$
\end{theorem}

\begin{proof}
As it was mentioned in Section \ref{Subsection: Sommers region}, an affine permutation $\omega$ has an inversion $(i,i+h)$ if and only if $\omega^{-1}(\Afund)$ is separated from $\Afund$ by the hyperplane $H_{i,i+h}^0$ or, equivalently, if and only if $\omega^{-1}(\Afund)\subset H_{i,i+h}^{0,\succ}.$ Given a region $R$, for any affine permutation $\omega$ such that $\omega^{-1}(\Afund)\subset R,$ the number of inversions of the form $(a,a+h)$ of height $h<kn$ is equal to $\lambda_R(a).$ 

If $\omega^{-1}\in \aSn^{kn+1}$ then the alcove $\omega^{-1}(\Afund)$ is the minimal alcove in the region $R$ and $\alc(\omega^{-1})=R.$ By definition, $\PS_{\omega^{-1}}(a)$ is equal to the number of inversions $(\alpha,\beta)$ of $\omega^{-1},$ such that $\beta-\alpha<kn+1$ and $\omega^{-1}(\beta)=a,$ which is the same as the number of inversions $(a,a+h)$ of $\omega,$ such that $\omega(a)-\omega(a+h)<kn+1.$ Note that $\omega(a)-\omega(a+h)$ cannot be equal to $kn,$ so $\PS_{\omega^{-1}}(a)$ is, in fact, equal to the number of inversions $(a,a+h),$ such that $\omega(a)-\omega(a+h)<kn.$ To match it with $\lambda_R(a),$ one has to prove the following equation for any $a\in\{1,\ldots,n\}$ and any $\w$ :
\begin{equation}
\label{eq: height vs width}
\sharp\left\{(a,a+h)\in\Inv(\omega)\mid h<kn\right\}=\sharp\left\{(a,a+h)\in\Inv(\omega)\mid \omega(a)-\omega(a+h)<kn\right\}.
\end{equation}
Given $r\in\{1,\dots,n-1\},$ define
$$
\gamma_{a,\w}(r):=\sharp\left\{(a,a+h)\in\Inv(\omega)\mid h<kn, h\equiv r \modd n\right\}
$$
and
$$
\gamma'_{a,\w}(r):=\sharp\left\{(a,a+h)\in\Inv(\omega)\mid \omega(a)-\omega(a+h)<kn, h\equiv r \modd n\right\}.
$$
 Let $h_{\max}^r$ be the maximal number such that $(a,h_{\max}^r)\in\Inv(\omega)$ and $h_{\max}^r\equiv r\ \modd n.$ It is not hard to see that   
 $$
 \gamma_{a,\w}(r)=\gamma'_{a,\w}(r)=\min\left(k,\left\lfloor\frac{h_{\max}^r}{n}\right\rfloor\right).
 $$ 
 Indeed, the total number of inversions $(a,a+h)$ such that $h\equiv r\ \modd n$ equals $\lfloor\frac{h_{\max}^r}{n}\rfloor$. If it is less than or equal to $k$ then 
 all of them satisfy both $h<kn$ and $\omega(a)-\omega(a+h)<kn.$ In turn, if it is greater than $k,$ then the inversions $(a,a+h)$ for $h=r,r+n,\ldots,r+(k-1)n$ satisfy the condition $h<kn,$ while the inversions $(a,a+h)$ for $h=h_{\max}^r,h_{\max}^r-n,\ldots,h_{\max}^r-(k-1)n$ satisfy the condition $\omega(a)-\omega(a+h)<kn.$    

Finally, the sum of identities  $\gamma_{a,\w}(r)=\gamma'_{a,\w}(r)$ for all $r$ is equivalent to \eqref{eq: height vs width}. 
\end{proof}

\begin{example}
When one applies the map $\PS$ to the affine permutations in the left part of Figure \ref{Figure: omega(A0) vs omega-1(A0)} one gets the Pak-Stanley labeling shown in Figure \ref{Figure: Pak-Stanley}.
\end{example}

The case $m=kn-1$ is treated similarly. The main difference is that instead of the set of all $k$-Shi regions $\Regkn$ one should consider the set of bounded $k$-Shi regions $\widehat{\Regkn}.$ One can show that every bounded $k$-Shi region contains exactly one maximal alcove. 

\begin{theorem}[\cite{FV2}]
\label{thm-max}
An alcove $\omega(\Afund)$ is the maximal alcove of a bounded $k$-Shi region if and only if $\omega^{-1}(\Afund)\subset D_n^{kn-1}.$
\end{theorem}

As above, we use Lemma \ref{Lemma: Sommers region} and Theorem \ref{thm-max}
to obtain the bijection
$\balc:\aSn^{kn-1}\to \widehat{\Regkn}.$ We prove the following theorem:

\begin{theorem}\label{Theorem: PS bijectivity kn-1}
The image of the subset $\widehat{\Regkn}\subset \Regkn$ under the Pak-Stanley labeling is exactly $\PF_{(kn-1)/n}\subset \PF_{(kn+1)/n}.$
Furthermore, one gets $\lambda\circ \balc=\PS$ in this case.
In particular, $\PS$ is a bijection for $m=kn-1.$ 
\end{theorem}

\begin{proof}
It is sufficient to prove the formula $\lambda\circ \balc=\PS.$ Indeed, this would imply that the restriction of the Pak-Stanley labeling to the subset $\widehat{\Regkn}\subset \Regkn$ is an injective (and, therefore, bijective) map from $\widehat{\Regkn}$ to $\PF_{(kn-1)/n}.$

If $\omega^{-1}\in \aSn^{kn-1}$ then the alcove $\omega^{-1}(\Afund)$ is the maximal alcove of a bounded region $R$ and $\balc(\omega^{-1})=R.$ Similarly to Theorem \ref{Theorem: PS bijectivity kn+1}, we get that $\PS_{\omega^{-1}}(a)$ is equal to the number of inversions $(a,a+h)$ of $\omega$ such that $\omega(a)-\omega(a+h)<kn-1.$ Since $\omega\in {}^{kn-1}\aSn,$ one has $\omega(a)-\omega(a+h)\neq kn-1$ for any $h>0.$ Therefore, $\PS_{\omega^{-1}}(a)$ is equal to the number of inversions $(a,a+h)$ in $\omega$ such that $\omega(a)-\omega(a+h)<kn.$ In the proof of Theorem \ref{Theorem: PS bijectivity kn+1} we have shown that this number is equal to $\lambda_R(a).$
\end{proof}







\section{Minimal Length Representatives and the Zeta Map.}\label{Section: zeta}

\begin{definition}
Let $\invsub$ be the set of subsets $\Delta\subset\mathbb Z_{\ge 0},$ such that $\Delta+m\subset\Delta,$ $\Delta+n\subset\Delta,$ and $\min(\Delta)=0.$ A number $a$ is called an
{\em  $n$-generator}
 of $\Delta$, if $a\in \Delta$ and $a-n\notin \Delta$.
Every $\Delta\in \invsub$ has exactly $n$ distinct $n$-generators.
\end{definition}

In \cite{GM11,GM12} such subsets were called {\it $0$-normalized semimodules over the semigroup generated by $m$ and $n$}. We will simply call them $m,n$-invariant subsets.

There is a natural map $R:\aSmn\to \invsub$ given by $\omega \mapsto \Delta_{\w}-\min(\Delta_{\w})$ (here, as before, $\Delta_{\omega}=\{i\in\BZ: \omega(i)>0\}$). Let $\minl_n^m$ be the set of $m$-stable minimal length right coset representatives of $\Sn\backslash\aSn.$ In other words, 
\begin{equation}
\label{min representative}
\minl_n^m:=\{\omega\in\aSmn\mid\omega^{-1}(1)<\ldots<\omega^{-1}(n)\}.
\end{equation}
One can 
check that the restriction $R|_{\minl_n^m}:\minl_n^m\to \invsub$ is a bijection. Indeed, the integers 
$\omega^{-1}(1),\ldots,\omega^{-1}(n)$ are the $n$-generators of $\Delta_{\w},$ and
since $\w\in \minl_n^m$ we have $\omega^{-1}(1)<\ldots<\omega^{-1}(n),$ so one can uniquely recover $\w$ from $\Delta_\w.$
Let $\hat{R}:=R|_{\minl_n^m}:\minl_n^m\to \invsub$ denote the restriction.

Recall that $\Ymn$ is the set of Young diagrams that fit under diagonal in an $n\times m$ rectangle and $P:\Pmn\to \Ymn$ is the natural map. In \cite{GM11,GM12} the first two named authors constructed two maps $D:\invsub\to\Ymn$ and $G:\invsub\to\Ymn,$ proved that $D$ is a bijection, and related the two maps to the theory of $q,t$-Catalan numbers in the following way. In the case $m=n+1$ one gets
$$
c_n(q,t)=\sum\limits_{\Delta\in \invsubn} q^{\delta-|D(\Delta)|}t^{\delta-|G(\Delta)|},
$$
where $\delta=\frac{n(n-1)}{2}$ and $c_n(q,t)$ is the Garsia-Haiman $q,t$-Catalan polynomial. It is known that these polynomials are symmetric $c_n(q,t)=c_n(t,q),$ although the proof is highly non-combinatorial and uses the machinery of Hilbert schemes, developed by Haiman. Finding a combinatorial proof of the symmetry of the $q,t$-Catalan polynomials remains an open problem.

The above consideration motivates the rational slope generalization of the $q,t$-Catalan numbers:
$$
c_{m,n}(q,t)=\sum\limits_{\Delta\in \invsub} q^{\delta-|D(\Delta)|}t^{\delta-|G(\Delta)|},
$$
where $\delta=\frac{(m-1)(n-1)}{2}$ is the total number of boxes below the diagonal in an $n\times m$ rectangle. The symmetry of these polynomials remains an open problem beyond the classical case $m=n+1$ and the cases $\min(m,n)\le 4$ (see \cite{GM12} for $\min(m,n)\le 3$ and \cite{LLL} for $\min(m,n)=4$). It was also shown in \cite{GM11} that the composition $G\circ D^{-1}:\Ymn\to \Ymn$ generalizes Haglund's zeta map exchanging the pairs of statistics $(\area,\dinv)$ and $(\bounce, \area)$ on Dyck paths. It was conjectured that the map $G,$ and therefore, the generalized Haglund's zeta, are also bijections. This would imply a weaker symmetry property $c_{m,n}(q,1)=c_{m,n}(1,q).$ In \cite{GM12} the bijectivity of $G$ was proved for $m=kn\pm 1.$ For more details on this work we refer the reader to \cite{GM11,GM12}. 

Let $\star$ denote the involution on $\aSn$: $\w^{\star}(x)=1-\w(1-x)$. 

\begin{lemma}
The map $\star$ preserves the set $\aSmn$ and the set $\minl_n^m$. The map $\overline{\star}:(i,j)\mapsto (1-j,1-i)$ provides a height preserving bijection from $\{(i,j)\mid i<j,\ \w(i)>\w(j)\}$ to $\{(i,j)\mid i<j,\ \w^\star(i)>\w^\star(j)\}.$
\end{lemma}

\begin{proof}
Let us check that $\w^\star$ is an affine permutation:
$
\w^{\star}(x+n)=1-\w(1-x-n)=1-\w(1-x)+n,
$
$$\sum_{i=1}^{n}\w^{\star}(i)=n-\sum_{i=1-n}^{0}\w(i)=n-\sum_{i=1}^{n}(\w(i)-n)=n+n^2-\frac{n(n+1)}{2}=\frac{n(n+1)}{2}.
$$
If $\w(1)<\ldots<\w(n)$ then $\w(1-n)<\cdots< \w(0)$, so $\w^{\star}(1)<\ldots<\w^{\star}(n)$.
Let $(i,j)$ be such that $i<j$ and $\w(i)>\w(j).$ Then $1-j<1-i,$ and 
$$
\w^\star(1-j)=1-\w(1-(1-j))=1-\w(j)>1-\w(i)=\w^\star(1-i).
$$ 
Note also that $\overline{\star}$ squares to identity. Therefore, since $\star$ is an involution, $\overline{\star}$ is a bijection between $\{(i,j)\mid i<j,\ \w(i)>\w(j)\}$ and $\{(i,j)\mid i<j,\ \w^\star(i)>\w^\star(j)\}$.
Furthermore, since $\overline{\star}$ preserves height, $\star$ preserves the set $\aSmn.$
\end{proof}

The following Theorem shows that our constructions are direct generalizations of those of \cite{GM11, GM12}:

\begin{theorem}\label{Theorem: restriction to minimal length reps}
One has the following identities:
\begin{enumerate}
\item $P\circ \An\circ\hat{R}^{-1}=D,$
\item $P\circ \PS\circ \star \circ \hat{R}^{-1}=G.$
\end{enumerate} 
\end{theorem} 

\begin{proof}
The first statement follows from the definition of $\An$ and Lemma \ref{an an}. For the second statement, we need to recall the definition of the map $G$. 

Given an $m,n$-invariant subset $\Delta\in \invsub$, let $u_1<\ldots<u_n$ be its $n$-generators.
The map $G$ was defined in \cite{GM11,GM12} by the formula
$$
G_{\Delta}(\alpha)=\sharp \left([u_\alpha,u_\alpha+m]\setminus \Delta\right). 
$$

Given a minimal coset representative $\omega \in \minlnm$,
we can consider an $m,n$-invariant subset $R(\omega)=\Delta_{\omega}-\min(\Delta_{\w})\in \invsub$. Its $n$-generators are $u_\alpha=\omega^{-1}(\alpha)-\min(\Delta_{\w})$, and by \eqref{min representative} we have $u_1<\ldots<u_n$.
For every $x\in [u_\alpha,u_\alpha+m]\setminus R(\omega)$, define $x':=x+\min(\Delta_{\w})$, then all such $x'$ (and hence $x$) are defined by the inequalities
$$
\omega^{-1}(\alpha)<x'<\omega^{-1}(\alpha)+m,\ \omega(x')\le 0.
$$
Note that by \eqref{min representative} the inequality $\omega(x')<0$ can be replaced by $\omega(x')<\alpha$. Indeed, we have $\w^{-1}(1)<\w^{-1}(2)<\ldots<\w^{-1}(n),$ and, therefore, $\w(x')\notin \{1,\ldots,\alpha-1\}$ for $\omega^{-1}(\alpha)<x'<\omega^{-1}(\alpha)+m.$ Therefore the set $[u_\alpha,u_\alpha+m]\setminus R(\omega)$ is in bijection with the set 
$$
\{(i,j)\mid i<j<i+m,\ \w(i)>\w(j),\ \w(i)=\alpha\}.
$$
In turn, the map $\overline{\star}$ bijectively maps this set to the set
$$
\{(1-j,1-i)\mid (1-j)<1-i<(1-j)+m,\ \w^\star(1-j)>\w^\star(1-i),\ \w^\star(1-i)=1-\w(i)=1-\alpha\},
$$
or, after a shift by $n$ and a change of variables,
$$
\{(i,j)\mid i<j<i+m,\ \w^\star(i)>\w^\star(j),\ \w^\star(j)=n+1-\alpha\}.
$$
Therefore, according to the definition of the map $\PS,$ we get
$$
G_{R(\omega)}(\alpha)=\PS_{\omega^{\star}}(n+1-\alpha),
$$
and thus
$$
G(R(\w))=P(\PS(\w^\star)).
$$
\end{proof}

The involution $\star$ could have been avoided in Theorem \ref{Theorem: restriction to minimal length reps} by adjusting the definition of the map $\PS.$ However, in that case one would have to use $\star$ to match the map $\PS$ for $m=kn+1$ with the Pak-Stanley labeling (see Section \ref{Section: Pak-Stanley}).  

The composition $\PS\circ\An^{-1}:\Pmn\to\Pmn$
\omitt{ changed \An to \An^{-1}, right?}
should be thought of as a rational slope parking function generalization of the Haglund zeta map $\zeta$.
Note that its bijectivity remains conjectural beyond cases $m=kn\pm 1,$ which follows immediately from Theorems \ref{Theorem: PS bijectivity kn+1} and \ref{Theorem: PS bijectivity kn-1}.   


\section{Relation to DAHA representations}\label{Section: DAHA}

\subsection{Finite-dimensional representations of DAHA}

It turns out that the map $A$ is tightly related to the representation theory of double affine Hecke algebras (DAHA).
This theory is quite elaborate and far beyond the scope of this paper, so we refer the reader to Cherednik's book \cite{CheBook} for all details. Here we just list the necessary facts about finite-dimensional representations of DAHA.

Let $\Hh_n$ denote the DAHA of type $A_{n-1}$.
It contains
the finite Hecke algebra
generated by the elements $T_i$, $1 \le i < n$ 
as well as two  commutative subalgebras
\omitt{and $X_1,\ldots,X_n$ and $Y_1,\ldots,Y_n$,
} 
$X_i/X_j$, $1 \le i \neq j \le n$, and $Y_1^{\pm 1},\ldots,Y_n^{\pm 1}$
subject to
commutation relations between $X's$ and $Y's$  that depend on two parameters $q$ and $t$.
(Alternatively, one can take generators 
$T_i$, $0 \le i < n$, $\pi$ and
$X_i/X_j$, $1 \le i \neq j \le n$.)
 $\Hh_n$  admits a (degree $0$ Laurent) {\em polynomial representation}
$V=\C[X_i/X_j]_{1 \le i \neq j \le n}$, where $X_i/X_j$ act as multiplication operators,
and $Y_i$ act as certain difference operators.
We can also obtain $V$ by inducing a 1-dimensional representation of the subalgebra
generated by the $T_i$ and $Y_i^{\pm 1}$ up to $\Hh_n$.
The product $Y_1 Y_2 \cdots Y_n $ (or equivalently $\pi^n$) acts as a constant
on this representation.
This constant agrees with the scalar by which the product acts on the initial
1-dimensional representation. 
We usually take this constant to be $1$, or indeed impose the relation
$Y_1 Y_2 \cdots Y_n = 1$ in $\Hh_n$.
However, to match the combinatorics developed in this paper, it is convenient
to choose that scalar to be $q^{\frac{n+1}{2}}$.

There exists a basis of $V$ consisting of 
nonsymmetric Macdonald polynomials $E_{\sigma}(X_i)$ labeled by
minimal length right coset representatives $\sigma \in \Sn\backslash\aSn$
such that $Y_i$ are diagonal in this basis: $Y_i(E_{\sigma})=a_i(\sigma)E_{\sigma}.$

The weights $a_i(\sigma)$ are directly related to the combinatorial content of this paper and can be described as follows.
Corresponding to  the fundamental alcove
$\sigma = \id$ we have $E_{\id}=1$ and its weight equals to:
$$
a(\id)=(a_1(\id),\ldots,a_n(\id))=
q^{\frac{n+1}{2n}}(t^{\frac{1-n}{2}}, t^{\frac{3-n}{2}},\ldots,t^{\frac{n-1}{2}}).
$$
\omitt{
$$
a(\sigma_0)=(a_1(\sigma_0),\ldots,a_n(\sigma_0))=
(t^{\frac{1-n}{2}}, t^{\frac{3-n}{2}},\ldots,t^{\frac{n-1}{2}}).
$$
}
As we cross the walls (from $\sigma \Afund$ to $\sigma s_i \Afund$),
the weights are transformed as follows:
\begin{equation}
\label{reflections for weights}
s_i(a_1,\ldots,a_n)=\begin{cases}
(a_1,\ldots, a_{i+1},a_i,\ldots,a_n),& \text{if}\ i\neq 0\cr
(a_n/q,a_2,\ldots,a_{n-1},qa_{n}),& \text{if}\ i=0.
\end{cases}
\end{equation}
One can check that \eqref{reflections for weights} defines an action of the
affine symmetric group on the set of sequences of Laurent monomials in $q$ and $t$.

If the parameters $q$ and $t$ are connected by the relation $q^m=t^n$ for coprime $m$ and $n$, the polynomial representation $V$ becomes reducible, and admits a finite-dimensional quotient $L_{m/n}$ of dimension $m^{n-1}$. The basis of $L_{m/n}$
is again given by the nonsymmetric Macdonald polynomials $E_{\sigma}$, but now the
permutation $\sigma$ should have $\sigma \Afund$
belong to the (dominant) region bounded by the hyperplane $x_1-x_n=m$.  
In other words, we can cross a wall if and only if the ratio of the corresponding weights is not equal to $t^{\pm 1}$.
(See \cite[Theorem 6.5]{chered-fourier}  for a discussion on these finite-dimensional
quotients.
See \cite[(1.17)]{chered-diagonal} for the formula for the intertwiners that 
take $E_\sigma$ to $E_{\sigma s_i}$.
See \cite{chered-perfect} for the nonsymmetric Macdonald  evaluation formula
that describes the $E_\sigma$ in the radical of the polynomial representation.)


\subsection{From DAHA weights to Sommers region}

For the finite-dimensional representation $L_{m/n}$ we have $q^m=t^n$, so $t=q^{m/n}$. This means that every monomial $q^{x}t^{y}$ can be written as $q^{\frac{nx+my}{n}}$, so we can rewrite the DAHA weights as
$$
a(\sigma)=(a_1,\ldots,a_n)=(q^{b_1(\sigma)/n},\ldots,q^{b_n(\sigma)/n}).
$$
It turns out that ``evaluated weights'' $b_i$ are tightly related to the labeling of the region $D_{n}^{m}$ by affine permutations.
Let $c=\frac{(m-1)(n+1)}{2}$.
Consider the affine permutation  $\wm =[m-c,2m-c,\ldots,nm-c]$. By Lemma \ref{isometry}, $\wm $ identifies
the dilated fundamental alcove with the simplex $D_{n}^{m}$.
Recall $\w \in \maSn \iff \wm^{-1}\omega (\Afund) \subseteq m D_n^1 = m \Afund$.

\begin{theorem}
Under this identification, for every $\w \in \maSn$
 one has:
$$
b(\wm^{-1}\omega)=\omega,
$$
by which we mean  for $\sigma = \wm^{-1}\omega$ that
$(b_1(\sigma),\ldots,b_n(\sigma)) = (\w(1), \ldots, \w(n))$.
\end{theorem}

\begin{proof}
In the weight picture we start from the fundamental alcove at $\sigma = \id$,
where we have weights 
\begin{gather*}
a(\id) =
q^{\frac{n+1}{2n}}(t^{\frac{1-n}{2}}, t^{\frac{3-n}{2}},\ldots,t^{\frac{n-1}{2}})
=(q^{(m-c)/n}, q^{(2m-c)/n},\ldots,q^{(nm-c)/n}),
\\
  b(\id) =
b(\wm^{-1} \wm) =
(m-c,2m-c,\ldots,nm-c)
=(\wm(1), \ldots, \wm(n)).
\end{gather*}
By Lemma \ref{isometry}, $\Afund \subseteq mD_n^1$  corresponds to the alcove
$\wm(\Afund) \subseteq D_{n}^{m}$, 
 which we  label by $\wm=[m-c,\ldots,nm-c] \in \maSn$.
Therefore the desired identity holds for
$\sigma = \id$ 
and can be extended to any
$\sigma$ with $\sigma \Afund \subseteq m \Afund = m D_n^1$
(equivalently  $\wm \sigma \in \maSn$)
by rules \eqref{reflections for weights}. 
\end{proof}


\begin{figure}
\begin{center}

\begin{tikzpicture}[scale=0.7]
\filldraw [yellow, fill=yellow] (4,0)--(7,0)--(5.5,2.6)--(4,0);
\filldraw [yellow, fill=yellow] (4,0)--(7,0)--(5.5,-2.6)--(4,0);
\filldraw [yellow, fill=yellow] (4,0)--(5.5,2.6)--(2.5,2.6)--(4,0);
\filldraw [yellow, fill=yellow] (7,0)--(5.5,2.6)--(8.5,2.6)--(7,0);
\filldraw [yellow, fill=yellow] (7,0)--(10,0)--(8.5,2.6)--(7,0);
\filldraw [yellow, fill=yellow] (1,0)--(4,0)--(2.5,2.6)--(1,0);
\filldraw [yellow, fill=yellow] (4,5.2)--(5.5,2.6)--(2.5,2.6)--(4,5.2);
\filldraw [yellow, fill=yellow] (7,5.2)--(5.5,2.6)--(8.5,2.6)--(7,5.2);
\filldraw [yellow, fill=yellow] (4,0)--(5.5,-2.6)--(2.5,-2.6)--(4,0);
\filldraw [yellow, fill=yellow] (7,0)--(5.5,-2.6)--(8.5,-2.6)--(7,0);
\filldraw [yellow, fill=yellow] (4,5.2)--(7,5.2)--(5.5,2.6)--(4,5.2);
\filldraw [yellow, fill=yellow] (1,0)--(4,0)--(2.5,-2.6)--(1,0);
\filldraw [yellow, fill=yellow] (7,0)--(10,0)--(8.5,-2.6)--(7,0);
\filldraw [yellow, fill=yellow] (4,5.2)--(5.5,2.6)--(2.5,2.6)--(4,5.2);
\filldraw [yellow, fill=yellow] (4,5.2)--(7,5.2)--(5.5,7.8)--(4,5.2);
\filldraw [yellow, fill=yellow] (10,0)--(11.5,-2.6)--(8.5,-2.6)--(10,0);
\filldraw [yellow, fill=yellow] (1,0)--(2.5,-2.6)--(-0.5,-2.6)--(1,0);

\filldraw [fill=black] (-0.5,-2.6) circle (0.3);

\draw[thick]  (-0.5,-2.6)--(11.5,-2.6);
\draw[dotted] (0,0)--(11,0);
\draw[dotted] (0,2.6)--(11,2.6);
\draw[dotted] (0,5.2)--(11,5.2);

\draw[thick] (-.5,-2.58)--(5,6.92)--(5.5, 7.783);
\draw[dotted] (2,-3.46)--(7.5,6.06);
\draw[dotted] (4.5,-4.33)--(10,5.2);
\draw[dotted] (7,-5.2)--(12.5,4.33);

\draw[dotted] (-1.5,4.33)--(4,-5.2);
\draw[dotted] (1, 5.2)--(6.5,-4.33);
\draw[dotted]  (3.5,6.06)--(9,-3.46);
\draw[thick] (5.5,7.783 )--(6,6.92)--(11.5,-2.58);

\draw (5.5,0.7) node {\scriptsize $q^{\frac 23} (\frac qt,1,\frac{t}{q})$};

\draw (7,1.7) node {\scriptsize $q^{\frac 23}(\frac{t}{q^2},1,\frac{q^2}t)$};

\draw (4,1.7) node {\scriptsize $q^{\frac 23}(1,\frac qt,\frac{t}{q})$};

\draw (5.5,-0.7) node {\scriptsize $q^{\frac 23}(\frac qt,\frac{t}{q},1)$};

\draw (8.5,0.75) node {\scriptsize $q^{\frac 23}(\frac{t}{q^2},\frac {q^2}t,1)$};
\draw (2.5,0.75) node {\scriptsize $q^{\frac 23}(1,\frac{t}{q},\frac qt)$};
\draw (4,3.25) node {\scriptsize $q^{\frac 23}(\frac{t}{q^2},\frac qt,q)$};
\draw (7,3.25) node {\scriptsize $q^{\frac 23}(1,\frac{t}{q^2},\frac {q^2}t)$};
\draw (7,-2.1) node {\scriptsize $q^{\frac 23}(\frac{1}{q},\frac{t}{q},\frac {q^2}t)$};
\draw (4,-2) node {\scriptsize $q^{\frac 23}(\frac{t}{q},\frac qt,1)$};

\draw (5.5,5.7) node {\scriptsize $q^{\frac 23}(\frac qt,q,\frac{t}{q^2})$};
\draw (5.5,4.5) node {\scriptsize $q^{\frac 23}(\frac qt,\frac{t}{q^2},q)$};
\draw (10,-2.1) node {\scriptsize $q^{\frac 23}(\frac {q^2}t,\frac{1}{q},\frac{t}{q})$}; 
\draw (8.5,-0.6) node {\scriptsize $q^{\frac 23}(\frac{1}{q},\frac {q^2}t,\frac{t}{q})$};
\draw (2.5,-0.7) node {\scriptsize $q^{\frac 23}(\frac{t}{q},1,\frac qt)$};
\draw (1,-2.1) node {\scriptsize $q^{\frac 23}(\frac 1t,1,t)$};
\end{tikzpicture}


\omitt{
\caption{Evaluation of DAHA weights at $t=q^{4/3}$ (left);
Sommers region labeled by $\omega$ (right)}
\caption{DAHA weights for $L_{4/3}$}
\caption{Evaluation of DAHA weights at $t=q^{4/3}$ (left); Sommers region labeled by $\omega$ (right)}
} 
\caption{DAHA weights for $L_{4/3}$ above. When one evaluates at $t=q^{4/3}$, the weights
become $(q^\frac{u(1)}3, q^\frac{u(2)}3, q^\frac{u(3)}3)$ for the matching
alcove $\wm u \Afund$ which is labeled by $u \in \maSnn{3}{4}$  in
Figure \ref{fig:sommers region 3 4}.
Note for the fundamental alcove,
$a(\id) =
q^{\frac 23}( \frac 1t, 1, t)
= q^{\frac 23}(q^{-\frac 43}, q^0, q^{\frac 43})
= (q^{-\frac 23}, q^{\frac 23}, q^{\frac 63})$ and $u=[-2,2,6] = \wm \id$.
Compare this to
Figure \ref{fig:sommers region 3 4}, where the alcove
$\Afund$ in the left figure matches
the alcove labeled $[-2,2,6]$ in right figure.
}
\end{center}
\end{figure}


\subsection{From DAHA weights to parking functions}

Instead of direct evaluation of DAHA weights as powers of $q^{1/n}$, one can instead draw monomials $q^{x}t^{y}$ on the $(x,y)$-plane. This point of view was used in much wider generality in \cite{SuVa}, where the weights were interpreted in terms of periodic skew standard Young tableaux. Here we focus on finite-dimensional representations and relate this picture to parking function diagrams. 

Let $a=(a_1,\ldots,a_n)$ be a DAHA weight. We define a function $T_{a}:\ZZ^2\to \ZZ$ labeling the square lattice by the following rule.
For every $i$, let us present $a_i=(q^{\frac{n+1}{2n}} t^{\frac{-1-n}{2}}) q^{x_i}t^{y_i}$
and define $T_a(x_i,y_i)=i$.
Under this renormalization, $\{y_1, \ldots, y_n\} =   \{1, \ldots, n\}$.
Hence we obtain $n$ 
squares labeled $1,\ldots, n$ in the rows $1,\ldots n$ in some order.  We can extend this
labeling to the whole plane by the following two-periodic construction. First,
one can identify $q^m$ with $t^n$ and write $T_{a}(x+m,y-n)=T_{a}(x,y)$.

Secondly, recall that the $b_i$ that correspond to $a$
 can be naturally extended to an 
affine permutation using the quasi-periodic condition $b_{i+n}=b_i+n$. This means that one can define $a_i$ for all integer $i$ by the rule $a_{i+n}=qa_i$, and $T_{a}(x+1,y)=T_{a}(x,y)+n$.
Hence the fillings in the boxes of $T_a$ increase across rows automatically; that
is, $T_a$ is row-standard.  The more interesting question is when is $T_a$ column
standard, which in this context means fillings increase up columns.
\begin{lemma}
The weight $a$ appears in the finite-dimensional representation $L_{m/n}$
 if and only if $T_{a}$ is a standard Young tableau (SYT), that is,
$T_{a}(x+1,y)>T_{a}(x,y)$
and
$T_{a}(x,y+1)>T_{a}(x,y)$.
\end{lemma}

\begin{proof}
Indeed, in terms of $b_i$ this means that $b_{i}+m$ appears after $b_i$,
which  is precisely equivalent to $m$-stability.
\end{proof}

\begin{corollary}
There is a natural bijection between the alcoves in the Sommers region and surjective maps $T:\ZZ^2\to \ZZ$ satisfying the following conditions:
\begin{equation}
\label{periodic syt}
T(x+1,y)=T(x,y)+n,\ T(x+m,y-n)=T(x,y),\ T(x,y+1)>T(x,y).
\end{equation}
\end{corollary}
 
 
 
\begin{lemma}
There exists a unique up to shift $n\times m$ rectangle  such that all squares labeled by positive numbers are located above the NW-SE diagonal. The corresponding parking function diagram coincides with the Anderson-type labeling up to a central symmetry.
\end{lemma}

\begin{proof}
If $(b_1,\ldots,b_n)$ corresponds to  the weight $a$,
then its corresponding $n$-invariant subset has n-generators  $b_i$,
and contains all fillings in squares to the right of labeled ones,
including the periodic shift by $(m, -n)$.
There exists a unique line with slope $m/n$ which is tangent to the resulting
infinite set of squares, and the tangency points define the $n\times m$
rectangle. Now the statement follows from the definition of the map $\An$.
\end{proof}

\begin{example}
Consider the weight $(a_1,a_2,a_3)=q^{\frac 23}(\frac{t}{q^2},\frac qt,q)
= (q^{0}, q^{\frac 13}, q^{\frac 53}) $ for $t=q^{4/3}$. 
We have $(b_1,b_2,b_3)=(0,1,5)$ and $\omega=[015]$.
The corresponding $(3,4)$-invariant subset is
$\Delta_{\omega}=\omega(\Z_{>0})=\{0,1,3,4,5,\ldots\}$, and the parking
function diagram  is shown in Figure \ref{weight to pf figure} on the right.

On the left side of Figure \ref{weight to pf figure} is a piece of $T_a$,
showing rows with $1\le y \le 4$ and columns with $-2\le x \le 4$.
Rewriting $a = q^{\frac 23} t^{-2}(q^{-2} t^3, q^1 t^1, q^1 t^2)$,
we see we put the filling 1 in square $(-2,3)$, 2 in square  $(1,1)$,
and 3 in square $(1,2)$. The periodicity conditions fill in the rest of
the squares of $T_a$.  The unique NW-SE line has been drawn, and the corresponding
rectangle it determines is rotated by 180 to obtain a Young diagram below the diagonal.

To get to the weight $a$ from the trivial weight, we need to apply
$\wm^{-1} \w = [-2,2,6]^{-1} \circ [0,1,5] = [-3,4,5]$.
Observe $[-3,4,5] = [-2(3)+3, 1(3) + 1, 1(3) +2]$. From this we could
also read off that the fillings 1,2,3 belong in squares
$(-2,3)$, $(1,1)$, $(1,2)$ respectively.
(We remind the reader of 
Figure \ref{fig:sommers region 3 4}, where the alcove labeled $[-3,4,5]$
in the left figure, matches the alcove labeled $[0,1,5]$ in right figure.)

Note, if we had instead normalized  in the more standard way
so that $Y_1 Y_2 Y_3 =1$ and the
fundamental alcove had weight $a' =(\frac 1t, 1, t)$, then we would have had
a shift by $2 = 3 \frac 23$ yielding
$(b_1',b_2',b_3')=(-2,-1,3)= (0-2,1-2,5-2)$ but still $\omega=[015]$ and
we would draw $T_{a'}$ as above.

\begin{figure}
\begin{center}
\begin{tikzpicture}

\filldraw [yellow,fill=yellow] (2,0)--(2,1)--(1,1)--(1,3)--(-1,3)--(-2,3)--(-2,4)--(5,4)--(5,0)--(2,0);

\draw[dashed] (-2,0)--(5,0);
\draw[dashed] (-2,1)--(5,1);
\draw[dashed] (-2,2)--(5,2);
\draw[dashed] (-2,3)--(5,3);
\draw[dashed] (-2,0)--(-2,4);
\draw[dashed] (-1,0)--(-1,4);
\draw[dashed] (0,0)--(0,4);
\draw[dashed] (1,0)--(1,4);
\draw[dashed] (2,0)--(2,4);
\draw[dashed] (3,0)--(3,4);
\draw[dashed] (4,0)--(4,4);

\draw[dashed](-2,3)--(2,0);

\draw[->,>=stealth] (0,0)--(5,0) node [right] {$q$};
\draw[->,>=stealth] (0,0)--(0,5) node [right] {$t$};

\draw(-1.5,3.5) node {$1$};
\draw(-0.5,3.5) node {$4$};
\draw(0.5,3.5) node {$7$};
\draw(1.5,3.5) node {$10$};
\draw(2.5,3.5) node {$13$};
\draw(3.5,3.5) node {$16$};
\draw(4.5,3.5) node {$19$};

\draw(-1.5,2.5) node {$-6$};
\draw(-0.5,2.5) node {$-3$};
\draw(0.5,2.5) node {$0$};
\draw(1.5,2.5) node {$3$};
\draw(2.5,2.5) node {$6$};
\draw(3.5,2.5) node {$9$};
\draw(4.5,2.5) node {$12$};

\draw(-1.5,1.5) node {$-7$};
\draw(-0.5,1.5) node {$-4$};
\draw(0.5,1.5) node {$-1$};
\draw(1.5,1.5) node {$2$};
\draw(2.5,1.5) node {$5$};
\draw(3.5,1.5) node {$8$};
\draw(4.5,1.5) node {$11$};

\draw(-1.5,0.5) node {$-11$};
\draw(-0.5,0.5) node {$-8$};
\draw(0.5,0.5) node {$-5$};
\draw(1.5,0.5) node {$-2$};
\draw(2.5,0.5) node {$1$};
\draw(3.5,0.5) node {$4$};
\draw(4.5,0.5) node {$7$};

\end{tikzpicture}
\ \ \ \
\begin{tikzpicture}

\filldraw [yellow,fill=yellow] (0,0)--(0,2)--(1,2)--(1,0);
\draw (0,0)--(0,3)--(4,3)--(4,0)--(0,0);
\draw [dashed] (0,1)--(4,1);
\draw [dashed] (0,2)--(4,2);
\draw [dashed] (1,0)--(1,3);
\draw [dashed] (2,0)--(2,3);
\draw [dashed] (3,0)--(3,3);

\end{tikzpicture}

\end{center}
\caption{Periodic SYT on $(x,y)$-plane (left); parking function diagram (right)}
\label{weight to pf figure}
\end{figure}
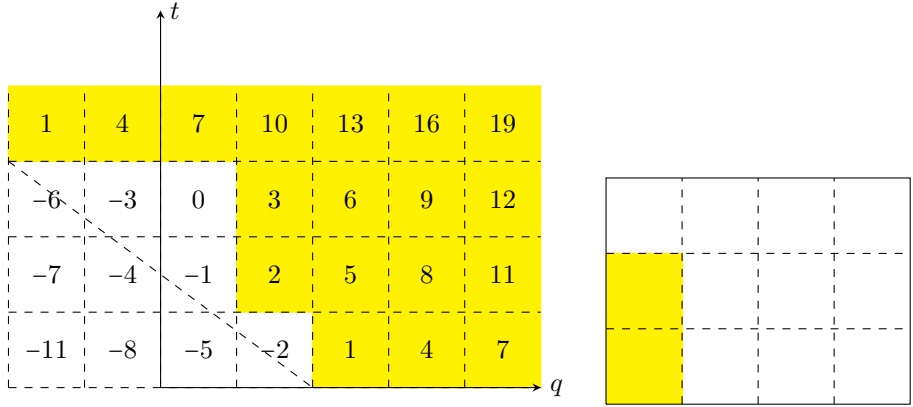
\end{example}

\section{Injectivity of $\PS$ for the finite symmetric group}\label{Section: finite}

In this Section we prove an analogue of Conjecture \ref{Conjecture: bijectivity} for the finite group $\Sn .$ 

\begin{definition}
Let $\Smn$ denote the intersection $\Sn \cap \aSmn.$ In other words, $\w \in \Smn$ if for all $1\le x\le n-m$ the inequality $\omega(x+m)>\omega(x)$ holds and 
$\{1, \ldots, n\} = \{\w(1),\ldots, \w(n)\}$.
We call such permutations {\em finite $m$-stable}.
\end{definition}


\begin{proposition}
The number of finite  $m$-stable permutations equals 
$$\sharp \Smn=\frac{n!}{\prod_{i=1}^{m}n_i!},\
\mbox{\rm where}\ n_i=\begin{cases}
\left\lfloor \frac{n-i}{m}\right\rfloor+1& \mbox{\rm if}\ i\le n,\\
0& \mbox{\rm if}\ i>n.\\
\end{cases}$$
\end{proposition}

\begin{proof}
The set $X=\{1,\ldots,n\}$ can be split into $m$ disjoint subsets 
$$X_i:=\{x\in X: x\equiv i \modd m\}$$       
of cardinality $n_i$. 
A permutation $\omega$ is finite  $m$-stable if and only if it increases on each $X_i$,
hence it is uniquely determined by an ordered partition
$$\{1,\ldots,n\}=\omega(X_1)\sqcup\ldots\sqcup \omega(X_m).$$
\end{proof}

\begin{example}
For $n=5, m=3$, $X_1 = \{1,4\}, X_2 = \{2,5\}, X_3 = \{3\}$. 
Observe $\w \in \Smn$ iff $\wi$ occurs in the shuffle
$14 \shuffle 25 \shuffle 3$,
which are precisely the $m$-restricted permutations in $\Snn{5}$.
\end{example}
\begin{definition}
Given a permutation $\omega\in \Smn$, let us define $\PS_{\omega}(\alpha)$ as
the number of inversions $(x,y)$ of $\omega$ such that $x<y<x+m,\
\omega(x)>\omega(y)=\alpha$ (the height of such inversion is less than $m$).
Define $$
\PS_{\omega}:=(\PS_{\omega}(1),\ldots,\PS_{\omega}(n)).
$$
\end{definition}
In other words, this is just the restriction $\PS \mid_{\Sn}$.
Hence by Theorem \ref{theorem-image-PS} the integer sequences in the image of $\PS$ are
$m/n$-parking functions.

Observe that if $\w \in\Sn$ then $\PS_\w(n) = 0$.

\begin{theorem}\label{Theorem: bijectivity finite}
The map $\PS$ from the set $\Smn$ to
$\Pmn$
is injective.
\end{theorem}

We provide two proofs of this Theorem, as they are somewhat different and might both be useful for the future  attempts to proof Conjecture \ref{Conjecture: bijectivity} in the affine case.

\begin{proof}[First proof]
Given a 
parking function  $\pw{\PS_{\omega}(1)\cdots\PS_{\omega}(n)}$ in the image,
we need to reconstruct $\omega$ or, equivalently, $\omega^{-1} \in \Sn$. We will first reconstruct the number $x_1=\w^{-1}(1),$ then $x_{2}=\w^{-1}(2),$ and so on, all the way up to $x_n=\w^{-1}(n).$
Note that $\wi = [x_1, x_2, \ldots, x_n]$ and $\PS_\w(i) = \SP_{\wi}(i) =
\sharp \{j \mid i < j \le n, 0 < x_i - x_j < m \}$. We have used that since
$\w \in \Sn$, for all $(i,j) \in \Inv(\w)$, $1 \le i < j \le n$.
Also since $\w \in \Sn$, for all $j \ge 1$ we have $\w(j) \ge 1$.
For the first step,
note that $x_1< m+1$, since otherwise $x_1-m$ and $x_1$ will form an inversion 
of $\w$ of height $m$, as $x_1 -m \ge 1$ so it occurs to the right of $x_1$ in $\wi$.
For every $1 \le y<x_1$, there is an inversion $(y,x_1)$ of height less than $m$
and there are no other inversions of the form $(-,x_1)$, hence $x_1=\PS_{\omega}(1)+1.$
On the next step we recover $x_2.$ Note that for every $y<x_2,$ there is an inversion $(y,x_2),$ unless $y=x_1.$ It follows that $x_2$ is either equal to $x_1+m$ or $x_2< m+1.$ It is not hard to see, that all these possible values of $x_2$ correspond to different values of $\PS_{\w}(2).$ Therefore, knowing $\PS_{\w}(2),$ one can recover $x_2.$ Let us show that one can proceed in that manner inductively all the way to $x_n.$

Suppose that one has already reconstructed $x_i=\omega^{-1}(i)$ for all $i<k$. Define the set 
$$
Y_{k-1}=\{x_1\ldots,x_{k-1}\}\sqcup\{l\in\mathbb Z:l<1\}=
\{y: 1\le y \le n, \omega(y)<k\}\sqcup\{l\in\mathbb Z:l<1\}.
$$ 
Let us use the notation $I(y):=(y-m,y]\cap \BZ$ for any $y\in\BZ_{\le n}.$ Consider the function $\varphi_k(y):=\sharp\left(I(y)\setminus Y_{k-1}\right)-1$ defined
on the domain
$y\in\BZ_{\le n}\setminus Y_{k-1}$. Given  $\PS_{\omega}(k)=\varphi_k(x_k),$ we need to reconstruct $x_k.$ Let us prove that the function $\varphi_k(y)$ is  non-decreasing.
Indeed, let $y>y'$ and $y,y'\in\BZ_{\le n}\setminus Y_{k-1}.$ Let $z\in I(y)=(y-m,y]\cap \mathbb Z$ and $z'\in I(y')=(y'-m,y']\cap \mathbb Z$ be such that $z-z'\equiv 0$ mod $m.$ It follows that if $z\in Y_{k-1},$ then also $z'\in Y_{k-1}.$
Otherwise, 
if $z'\nin Y_{k-1}$ then $1 \le z'$ and $\w(z') \ge k$. Further $y > y'$ implies
$z > z' \ge 1$ and $z\in Y_{k-1}$ gives $\w(z) < k \le \w(z')$.
Therefore, $(z',z)$ is an inversion of $\w$ of height divisible by $m,$ which implies that
$\w$ is not $m$-stable.  Contradiction.  

We conclude, that
$$
\varphi_k(y)=\sharp\left(I(y)\setminus Y_{k-1}\right)-1\ge \sharp\left(I(y')\setminus Y_{k-1}\right)-1=\varphi_k(y').
$$
Finally, we remark that $x_k\notin Y_{k-1}$ but $x_k-m\in Y_{k-1},$ since
otherwise $x_k-m\ge 1$ and $\omega(x_k-m)>k,$ and that produces an inversion of
height $m$.
Therefore, one check there is a strict inequality
$\varphi_k(y)<\varphi_k(x_k)$
 for any $y<x_k$ with $y \in \BZ_{\le n}\setminus Y_{k-1}.$
Thus, $x_k=\min\{y\in\BZ\setminus Y_{k-1}|\varphi_k(y)=\PS_{\omega}(k)\}$.
In particular, this set is non-empty.
 We illustrate this proof on an example in Figure \ref{Figure: bijection finite 1}.
\end{proof}

\begin{proof}[Second proof]
Define the function $g(\alpha,i)$ by the following formula:
$$
g(\alpha,i)=\sharp\{j\in (i-m,i]\cap\{1,\dots,n\}: \omega(j)> \alpha \}.
$$
By definition of $\PS_{\omega},$ one immediately gets
$$
g(\omega(i),i)=\PS_{\omega}(\omega(i)).
$$ 

\begin{lemma}\label{monot}
The function $g(\alpha,i)$ is non-decreasing in $i$ for any fixed $\alpha.$ 
\end{lemma}

\begin{proof}
Indeed, suppose that $g(\alpha,i)<g(\alpha,i-1).$ The interval $(i-m,i]$ is obtained from the interval $(i-1-m,i-1]$ by dropping $i-m$ and adding $i.$ Therefore, one should have $\omega(i-m)>\alpha$ and $\omega(i)\le \alpha$ to get $g(\alpha,i)<g(\alpha,i-1).$
But that
implies $\omega(i-m)>\omega(i),$
producing  an inversion of height $m.$
Contradiction. 
\end{proof}

We will need the following corollary:

\begin{corollary}\label{rec_fin}
For any $i\in\{1,\dots,n\},$  $\omega(i)$ is the minimal integer $\alpha$, such that $\alpha\neq \omega(j)$ for any $j<i,$ and $\PS_{\omega}(\alpha)=\sharp\{j\in (i-m,i)\cap\{1,\dots,n\}: \omega(j)>\alpha\}.$
\end{corollary}

\begin{proof}
Fix $i$.
Let $\alpha$ satisfy the above conditions.
Notice such an $\alpha$ must exist since $\w(i)$ satisfies these conditions as
$$
\PS_{\omega}(\omega(i))=g(\omega(i),i)=\sharp\left\{j\in (i-m,i): \omega(j)>\omega(i)\right\}.
$$ 
By minimality, $\alpha \le \w(i)$.
If $\alpha \neq \w(i)$ then we must have $\alpha < \w(i)$,
yielding $i \in \{j\in (i-m,i]\cap\{1,\dots,n\}: \omega(j)> \alpha \}$.
However $i \nin \left\{j\in (i-m,i) \cap\{1,\dots,n\}:
\omega(j) > \alpha \right\}$ whose cardinality is $ \PS_\w(\alpha)$ by assumption.
Hence $g(\alpha,i)  = \PS_{\omega}(\alpha)+1.$
If it were the case that $\alpha = \w(k)$ for some $k>i$, then 
 since $g(\alpha,-)$ is non-decreasing,
we get $\PS_w(\alpha) = \PS_w(w(k)) = g(w(k),k) =
g(\alpha,k) \ge g(\alpha,i) > \PS_\w(\alpha)$.
Contradiction.
On the other hand, $\alpha$ was chosen so  $\alpha\neq \omega(j)$ for any
$j<i$. 
Hence it must be that $\alpha = \w(i).$
 \end{proof}

Now we can complete the proof of Theorem \ref{Theorem: bijectivity finite} and  reconstruct $\omega$ starting from $\omega(1),$ then $\omega(2),$ and so on, using Corollary \ref{rec_fin}. Indeed, if we already reconstructed $\omega(1),\omega(2),\dots,\omega(i-1),$ then we can compute $\sharp\{j\in (i-m,i)\cap\{1,\dots,n\}: \omega(j)>\alpha\}$ for all $\alpha\in \{1,\ldots,n\}.$ Then $\omega(i)$ is the smallest number $\alpha\in \{1,\ldots,n\}$ such that $\alpha\neq\omega(j)$ for $j<i,$ and $\PS_{\omega}(\alpha)=\sharp\{j\in (i-m,i)\cap\{1,\dots,n\}: \omega(j)>\alpha\}.$ We illustrate this proof on example on Figure \ref{Figure: bijection finite 2}. 

\end{proof}

\begin{figure}
\begin{tabular}{|c|c|c|}
\hline 
$k$ & 
	$x_i, \,  i<k$  & $\varphi_k(y)$\cr
\hline 
1 & - - - - - - - & 0 1 2 2 2 2 2 \cr
\hline 
2 & 1 - - - - - - & - 0 1 2 2 2 2  \cr
\hline 
3 & 1 2 - - - - - & - - 0 1 2 2 2 \cr
\hline 
4 & 1 2 - 3 - - - & - - 0 - 1 1 2 \cr
\hline 
5 & 1 2 4 3 - - - & - - - - 0 1 2 \cr
\hline 
6 & 1 2 4 3 - - 5 & - - - - 0 1 - \cr
\hline 
7 & 1 2 4 3 - 6 5 & - - - - 0 - - \cr
\hline 
  & 1 2 4 3 7 6 5 & - - - - - - - \cr
\hline
\end{tabular}
\caption{Suppose that $n=7,$ $m=3,$ and $\PS_{\omega}=\pw{0010210}$. Let us reconstruct
$\wi$ using the first proof of Theorem \ref{Theorem: bijectivity finite}. We record on every step the numbers that we have already reconstructed and  the values of the function $\varphi_k$ for all other numbers.}
\label{Figure: bijection finite 1}
\end{figure}



\begin{figure}
\begin{tabular}{|c|c|c|}
\hline 
$k$ & $\omega(i),\ i<k$  & $\PS_{\omega}(\alpha)-\sharp\{j\in\{k-m+1,\ldots, k-1\}: \omega(j)>\alpha\}$\cr
\hline 
1 & - - - - - - - & 0 0 1 0 2 1 0 \cr
\hline 
2 & 1 - - - - - - & - 0 1 0 2 1 0 \cr
\hline 
3 & 1 2 - - - - - & - - 1 0 2 1 0 \cr
\hline 
4 & 1 2 4 - - - - & - - 0 - 2 1 0 \cr
\hline 
5 & 1 2 4 3 - - - & - - - - 2 1 0 \cr
\hline 
6 & 1 2 4 3 7 - - & - - - - 1 0 - \cr
\hline 
7 & 1 2 4 3 7 6 - & - - - - 0 - - \cr
\hline 
  & 1 2 4 3 7 6 5 & - - - - - - - \cr
\hline
\end{tabular}
\caption{As in Figure \ref{Figure: bijection finite 1}, $n=7,$ $m=3,$ and $\PS_{\omega}=
\pw{0010210}$. This time we reconstruct $\w$ using the second proof of Theorem \ref{Theorem: bijectivity finite}. We record on every step the numbers that we have already reconstructed and the difference $\PS_{\omega}(\alpha)-\sharp\{j\in\{k-m+1,\ldots, k-1\}: \omega(j)>\alpha\}$ for all $\alpha\in \{1,\ldots,n\}\setminus\{\omega(1),\omega(2),\dots,\omega(k-1)\},$ so that on each step we choose
the position of the leftmost 0 in the second column.}
\label{Figure: bijection finite 2}
\end{figure}



We do not know how to describe the image  $\PS(\Smn)$ for general $m$.
As an example, let us consider the case $m=2$ for which we do have a complete 
description.
Let us recall that $S_{n}^{2}$ is the set of finite permutations $\omega$ of $n$ elements with no inversions of height $2$, that is, $\omega(i+2)>\omega(i)$ for all $x$. We define the map $\inv^{(2)}$ from the set $S_{n}^{2}$ to the set of sequences of 0's and 1's as
$$\inv^{(2)}_\w(\alpha):=\chi\left(\omega(\omega^{-1}(\alpha)-1)>\alpha\right)
= \begin{cases}
1 & \text{ if } \omega(\omega^{-1}(\alpha)-1)>\alpha \\
0 & \text{ else.}
\end{cases}$$

\begin{lemma}
\label{m2 finite}
The image of $\inv^{2}$ consists of all $n$-element sequences $f$ of $0$'s and $1$'s, such that for every $\alpha\in\{1,\ldots,n\}$ at least half of the subsequence
$(f_\alpha,\ldots f_n)$ are $0$'s.
The image of $\inv^{2}$ agrees with that of $\PS |_{\Smnn{n}{2}}$.
\end{lemma}

\begin{proof}
Let $\omega$ be a permutation in $S_n^{2}$ and let $f=\inv^{(2)}(\omega)$. For every $\alpha$ such that $f_\alpha=1$ one can find $\beta=\omega(\omega^{-1}(\alpha)-1)>\alpha$ such that $f_\beta=0$ (otherwise $\omega$ would have an inversion of height $2$).
In other words, if we consider $\w^{-1} = [x_1,\ldots, x_n]$, $f(\alpha) = 1$ iff
$x_\alpha = i$ and $i-1$ occurs to the right, i.e. $i-1 = x_\beta$ with  $\beta > \alpha$.
Which occurs iff $\PS_\w(\alpha)=1$. 
And in this case $i-2$ cannot be to the right of $i-1$ as that would place it to
the right of $i$, i.e. $f(\beta) = \PS_\w(\beta)=0$.
 Note that the correspondence $\alpha\mapsto\beta=\omega(\omega^{-1}(\alpha)-1)$ from $1$'s to $0$'s in the sequence $f$ is injective and increasing. Therefore, for every $\alpha\in\{1,\ldots,n\}$ at least half of the subsequence $(f_\alpha,\ldots f_n)$ are $0$'s.

Since we know that $\inv^{(2)}$ is injective, the lemma now follows from the comparison of the cardinalities of the two sets.
\end{proof}

The sequences appearing in Lemma \ref{m2 finite} have a clear combinatorial meaning. Let us read the sequence $s$ backwards and replace $0$'s with a vector $(1,1)$ and $1$'s with a vector $(1,-1)$. We get a lattice path in $\mathbb{Z}^2$ which never goes below the horizontal axis. Such a path may be called a Dyck path with open right end, and Lemma \ref{m2 finite} establishes a bijection between the set of such paths of length $n$ and the set of
finite $2$-stable permutations.




\subsection{Algorithm to construct $\SP^{-1}$ in the affine case}

Here we present a conjectural algorithm that inverts $\SP$. 
While we have not yet shown the algorithm terminates, which in this
case means it eventually becomes $n$-periodic, 
we have checked it on several examples.

Given $f \in \Pmn$, extend $f$ to $\N$ by $f(i+tn) = f(i)$.  
Construct an injective function $U: \N \to \N$ as follows.
Informally, we will think of $U$ as the bottom row in the following table.
$$\begin{array}{r|ccc}
i & 1&2& \cdots \\
\hline
f(i) & f(1) & f(2) & \cdots \\
\hline
U(i)& U(1) &U(2) & \cdots 
\end{array}
$$
Since $U$ is manifestly injective, it will make sense to talk about $U^{-1}$.

We will insert the numbers $\alpha \in \N$ into the table as follows.
\begin{enumerate}
\item
Place $\alpha = 1$ under the leftmost $0$.  
In other words, let  $i = \min\{j \in \N | f(j) = 0\}$ and then set $U(i) = 1$.
As there always exists some $1 \le j \le n$ such that $f(j) = 0$, this is always possible.
\item
Assume $\{1, 2, \ldots, \alpha - 1 \}$ have already been placed. 
Place $\alpha$ in the leftmost empty position $i$ (i.e. $U(i) = \alpha$, 
with $i \nin \{U^{-1}(\beta) | 1 \le \beta < \alpha\}$  for
$i$ minimal) such that these two conditions hold.
\begin{itemize}
\item[($\mathtt{I}$)]
\label{item-m-restrict}
$\alpha$ is to the right of $\alpha - tm$ for $1 \le t < \alpha/m$, $t \in \N$.
More precisely, $i > U^{-1}(\alpha - tm)$.
\item[($\mathtt{II}$)]
\label{item-invertsSP}
If $U(i)  =\alpha$, then $f(i) = \sharp \{ \beta |
\beta \in (\alpha -m, \alpha), U^{-1}(\beta) > i \}$.
\end{itemize}
In other words, we build $U$ so that $f(i) =\sharp \{j | j > i, 0 < U(i)-U(j) < m \}$
counts the number of $m$-restricted inversions. 
Note that placing $\alpha$ is always possible, since a valid (non-minimal) position
 for $\alpha$ is under a $0$ of $f$ such that it and all spots to the right of it
are as yet unoccupied.
\end{enumerate}
\begin{conjecture}
For the $U$ constructed above, $\exists N$ such that for all $i \ge N$, $t \in \N$
\begin{enumerate}
\item
$U(i+tn) = U(i) + tn$,  so in particular
\item
\label{item-nogaps}
$U(N+j)$ for $1 \le j \le n$ have all been assigned values
\end{enumerate}
\end{conjecture}

Given $U$ constructed from $f \in \Pmn$ as in the algorithm and satisfying
the conditions of the conjecture, we construct $\w \in \maSn$ as follows:
Pick $t$ so $1+tn \ge N$.
By the periodicity of $U$ and that $U$ has no ``gaps" 
after $N$,
$\{U(i+tn) \bmod n | 1\le i\le n \} = \{ 1, 2, \ldots, n\}$.
Hence $b := \sum_{i=1}^n U(i+tn) \equiv \frac{n(n+1)}{2} \bmod n$.
Let $k$ be such that $b - \frac{n(n+1)}{2} = kn$.
Now set
$$ \w(i) = U(i+tn) -k.$$
This forces $\sum_{i=1}^n \w(i) = \frac{n(n+1)}{2}$, and so we see $\w \in \aSn$.
By construction,
($\mathtt{I}$) and ($\mathtt{II}$)
imply $w \in \maSn$ and $\SP_\w = f$.

We illustrate the algorithm to construct $U$ and $\w$ on the following example.


\begin{example}
Let $n=5, m=3$. Let $f = \pw{11002} \in \PF_{3/5}.$

\begin{center}
\begin{figure}
$$
\begin{array}{r|rrrrrrrrrrrrrrrrrrrrr}
\hline 
i & 1&2&3&4&5& \, 6&7&8&9&10& \, 11&12&13&14&15& \, 16&17&18&19&20& \cdots
\\
\hline 
f(i) & 1&1&0&0&2& \, 1&1&0&0&2&   \, 1&1&0&0&2& \, 1&1&0&0&2& \cdots
\\
\hline 
U(i) & 
\\
&.&.&1
\\
&
2&.&1
\\
&
2&3&1
\\
&
2&3&1&4
\\
&
2&3&1&4&.& .&.&5
\\
&
2&3&1&4&.& 6&.&5
\\
&
2&3&1&4&7& 6&.&5&
\\
&
2&3&1&4&7& 6&.&5&8
\\
&
2&3&1&4&7& 6&9&5&8&. &.& .&10
\\
&
2&3&1&4&7& 6&9&5&8&. &11& .&10
\\
&
2&3&1&4&7& 6&9&5&8&12 &11& .&10
\\
&
2&3&1&4&7& 6&9&5&8&12 &11& .&10&13
\\
&
2&3&1&4&7& 6&9&5&8&12 &11& 14&10&13
\\
&
2&3&1&4&7& 6&9&5&8&12 &11& 14&10&13 & .& . & 15
\\
&
2&3&1&4&7& 6&9&5&8&12 &11& 14&10&13 &. & 16 & 15
\\
&
2&3&1&4&7& 6&9&5&8&12 &11& 14&10&13 & 17 & 16 & 15 & \cdots
\\
\hline
\end{array}
$$
\caption{Algorithm to construct $U$ from 
$f = \pw{11002} \in \PF_{3/5}.$}
\label{figure-algorithm}
\end{figure}
\end{center}

Refer to Figure \ref{figure-algorithm} for a demonstration of how $U$ is constructed.
Note that $U(7) \neq 8$ since that would place $8$ before $5$, violating
being $3$-restricted.

In the above we can in fact take $N=5$.
Observe $\{ U(6), U(7), U(8), U(9), U(10) \} = \{ 6,9,5,8, 12 \}$ yielding
$b = 40$ and $k = 5$.  Hence we set $\w = [1,4,0,3,7]$.
Now one can easily verify $\w \in \maSnn{5}{3}$ and $\SP_\w = \pw{11002}$.
\end{example}

In practice, we have found $U$ to be surjective as well; in other words
there are no ``gaps" even before $N$. 
Further, when $f = \SP_u$ for some finite permutation $u \in \Sn$, we can take $N=1$.





\section{Relation to Springer theory}\label{Section: Springer}

\subsection{Springer fibers for the symmetric group}

Let $V$ be a finite-dimensional vector space and let $N$ be a nilpotent transformation of $V$.
Let $\Fl(V)$ denote the space of complete flags in $V$. A classical object in the representation theory is 
the Springer fiber (\cite{Sp78,Sh80}) defined as
$$X_N:=\{{\bf F}=\{V=V_1\supset V_2\supset\ldots\supset V_{n}\}\in \Fl(V): N(V_i)\subset V_i\}.$$
It is known that $X_N$ admits an affine paving with combinatorics completely determined by the conjugacy class of $N$ (see e.g. \cite{Sh80} and references therein). 


We will be interested in a particular case of this construction. Let us fix a basis $(e_1,\ldots,e_n)$ in the space $V$, consider the operator of shift by $m$: 
$$
N(e_i):=\begin{cases}
e_{i+m},& i+m\le n\\
0,& \text{otherwise}\\
\end{cases}
$$
The following theorem describes the structure of the affine cells in the variety $X_{N}$.

\begin{theorem}
\label{Theorem: cells of finite Springer fiber}
The variety $X_N$ admits an affine paving, where the cells $\Sigma_{\omega}$ are
parametrized by the finite $m$-restricted permutations $\omega \in \Smn$.
The dimension of $\Sigma_{\omega}$ is given by the number of inversions in $\omega^{-1}$ of height less than $m$.
\end{theorem}

\begin{proof}
The cells are essentially given by the intersections of Schubert cells in $\Fl(V)$ with the subvariety $X_N$. For the sake of completeness, let us recall their construction.
Given a permutation $\omega\in S_n$, we can define a stratum $\Sigma_{\omega}$ in $\Fl(V)$ consisting of the following flags:
$$
{\bf F}=\{V_1\supset V_2\supset\ldots\supset V_n\},\ V_i=\spann\{ 
v^{\omega(i)},\ldots,v^{\omega(n)} \},
$$
where
$$
v^\alpha=e_\alpha+\sum_{\beta>\alpha}\lambda^\alpha_\beta e_\beta.
$$
Note that the position of $v^\alpha$ in the basis equals $\omega^{-1}(\alpha)$.
After a triangular change of variables, we can assume that $\lambda^\alpha_\beta=0$ for $\beta>\alpha$ with $\omega^{-1}(\beta)>\omega^{-1}(\alpha).$ Therefore one can write
\begin{equation}
\label{tilde basis}
v^\alpha=e_\alpha+\sum_{\beta>\alpha,\omega^{-1}(\beta)<\omega^{-1}(\alpha)}\lambda^\alpha_\beta e_\beta=
e_\alpha+\sum_{(\alpha,\beta)\in\Inv(\w^{-1})}\lambda^\alpha_\beta e_\beta.
\end{equation}
The parameters $\lambda^\alpha_\beta$ are uniquely defined by the flag ${\bf F}$. They serve as coordinates on the affine space $\Sigma_{\omega}$, whose dimension is equal to the
length of $\omega$, i.e. $=\sharp \Inv(\w) = \sharp \Inv(\wi)$ .
 
Let us  study the intersection $\Sigma_{\omega}^N:=\Sigma_{\omega}\cap X_N.$  Since $Nv^\alpha$ starts with $e_{\alpha+m}$, the vector $v^{\alpha+m}$ should go after $v^\alpha$ in the basis, so one needs $\omega ^{-1}(\alpha+m)>\omega^{-1}(\alpha)$. Therefore $\Sigma^{N}_{\omega}$ is non-empty if and only if $\omega^{-1}$ is $m$-stable.
A flag ${\bf F}$ is $N$-invariant, if $N(v^\alpha)$ belongs to $\spann \{
v^\beta: \omega^{-1}(\beta)>\omega^{-1}(\alpha) \}$ for all $\alpha$.  
If 
$$
\beta>\alpha+m\ \text{and}\ \omega^{-1}(\beta)>\omega^{-1}(\alpha),
$$ 
then the coefficient in $v^{\alpha+m}-N(v^\alpha)$ at $e_\beta$ can be eliminated by subtracting an appropriate multiple of $v^\beta$. Once all these coefficients are eliminated, 
 the remaining coefficients in $v^{\alpha+m}-N(v^\alpha),$ labeled by $\beta>\alpha+m$ such that $\omega^{-1}(\beta)<\omega^{-1}(\alpha)$ will
 vanish automatically.
 Therefore $\Sigma^N_{\omega}$ is cut out in $\Sigma_{\omega}$ by the equations:
\begin{equation}
\label{restrictions on lambda}
\lambda^{\alpha+m}_\beta=\lambda^\alpha_{\beta-m}+\phi(\lambda)\ \text{if}\ \beta>\alpha+m,\omega^{-1}(\beta)<\omega^{-1}(\alpha),
\end{equation}
where $\phi(\lambda)$ are certain explicit polynomials in $\lambda^\mu_\nu$ with $\nu-\mu<\beta-\alpha-m,$ with no linear terms. 

It is clear that such equations are labeled by the inversions $(\alpha,\beta)$ in $\omega^{-1}$ of height bigger than $m$. Note also that the linear parts of these equations are linearly independent. Therefore the number of free parameters on $\Sigma^N_{\omega}$ equals to the number of inversions of $\wi$  of height less than $m$. 
\end{proof}

\begin{example}
Consider a 2-stable permutation $\omega=[2, 1, 4, 3]=\omega^{-1}$. The basis \eqref{tilde basis} has the form:
$$
v^{\omega(1)}=e_2,\ v^{\omega(2)}=e_1+\lambda^1_2 e_2,\ v^{\omega(3)}=e_4,\ v^{\omega(4)}=e_3+\lambda^3_4 e_4.
$$

There are two free parameters, so $\dim\Sigma^{N}_{\omega}=2$. Note that although $\lambda^3_4\neq \lambda^1_2$,
$$N(v^{\omega(2)})=e_3+\lambda^1_2 e_4\in \spann \{ v^{\omega(3)},v^{\omega(4)} \}.$$
\end{example}

\begin{example}
Consider a $3$-stable permutation $\omega^{-1}=[1, 5, 3, 2, 6, 4, 7],$ so $\omega=[1, 4, 3, 6, 2, 5, 7]$ is $3$-restricted. 
The basis \eqref{tilde basis} has a form:
$$
v^{\omega(1)}=e_1,\ v^{\omega(2)}=e_4,\ v^{\omega(3)}=e_{3}+\lambda^3_4 e_4,\ v^{\omega(4)}=e_6,\  v^{\omega(5)}=e_2+\lambda^2_3 e_3+\lambda^2_4 e_4+\lambda^2_6 e_6, 
$$
$$
v^{\omega(6)}=e_5+\lambda^5_6 e_6,\ v^{\omega(7)}=e_7.
$$
Since $N(v^{\w(5)})-v^{\w(6)}=(\lambda^2_3-\lambda^5_6)e_6+\lambda^2_4 e_7\in V_5,$ and 
$$
V_5=\spann
\{ v^{\w(5)},v^{\w(6)},v^{\w(7)} \} =\spann \{ e_2+\lambda^2_3 e_3+\lambda^2_4 e_4+\lambda^2_6 e_6,e_5+\lambda^5_6 e_6,e_7 \},
$$
the coefficient of $e_6$ must vanish and so we get the relation $\lambda^2_3=\lambda^5_6.$ Therefore $\dim \Sigma_{\omega}^N=4.$ 
\end{example}

\subsection{Springer fibers for the affine symmetric group}

We recall the basic definitions of the type $A$ affine Springer fibres, and refer the reader e. g. to \cite{GKM,KL,LS91} for more details. 

Let us choose an indeterminate $\varepsilon$ and consider the field $K=\C((\varepsilon))$ of Laurent power series and the ring $\CO=\C[[\varepsilon]]$ of power series in $\varepsilon$. Let $V=\C^n((\varepsilon))$ be a $K$-vector space of dimension $n$.

\begin{definition}
The 
{\em affine Grassmannian}
 $\CG_n$ for the group $GL_n$ is the moduli space of  $\CO$-submodules
 $M\subset V$ such that the following three conditions are satisfied:
\begin{itemize}
\item [(a)] $M$ is $\CO$-invariant.
\item [(b)] There exists $N$ such that $\varepsilon^{-N}\C^n[[\varepsilon]]\supset M\supset \varepsilon^{N}\C^n[[\varepsilon]].$
\item [(c)] Let $N$ be an integer satisfying the above condition. Then the following normalization condition is satisfied:
$$
\dim_{\C} \varepsilon^{-N}\C^n[[\varepsilon]]/M=\dim_{\C} M/ \varepsilon^{N}\C^n[[\varepsilon]].
$$
\end{itemize}

The 
{\em affine complete flag variety}
 $\CF_n$ for the group $GL_n$ is the moduli space of collections $\{M_0\supset\ldots\supset M_n\}$, where each $M_i$ satisfies (a) and (b), 
$\dim_{\C} M_i/M_{i+1}=1,$\ $M_n=\varepsilon M_0,$ and $M_0\in\CG_n,$ i.e. $M_0$ also satisfies the normalization condition (c).
\end{definition}

\begin{definition}
Let $T$ be an endomorphism of $V$. It is called 
{\em nil-elliptic}
if $\lim_{k\to\infty}T^{k}=0$ and the characteristic polynomial of $T$ is irreducible over $K$. 

Given a nil-elliptic operator $T$, one can extend its action to $\CG_n$ and to $\CF_n$ and define the 
{\em affine Springer fibers}
 as the corresponding fixed point sets. 
\end{definition}

\begin{remark}
The condition $\lim_{k\to\infty}T^{k}=0$ means that for any $N\in\N$ there exists $k\in\N$ such that $$T^k(\C^n[[\varepsilon]])\subset \varepsilon^N\C^n[[\varepsilon]].$$
\end{remark}

In \cite{LS91} Lusztig and Smelt studied the structure of the affine Springer fibers for a particular choice of $T$. Given a $\C$-basis $\{e_1,\ldots, e_n\}$ in $\C^n$,
one can consider it as a $K$-basis of
 $V=\C^n((\varepsilon))$. Consider the operator $N$ defined
by the equations $N(e_i)=e_{i+1},\ N(e_n)=\varepsilon e_1.$ The following theorem is the main result of \cite{LS91}.

\begin{theorem}(\cite{LS91})
Consider the nil-elliptic operator $T:=N^m$, where $m$ is coprime to $n$. Then the corresponding affine Springer fiber $\CF_{m/n}\subset\CF_n$ 
admits an affine paving by $m^{n-1}$ affine cells. 
\end{theorem}

It turns out that the affine paving of this affine Springer fiber is tightly related to the combinatorics of the simplex $D_{n}^{m}$.
This was implicitly stated in \cite{GORS12,H12,LS91}, but we would like to make this correspondence precise and explicit.

\begin{theorem}
\label{Theorem: cells of affine Springer}
There is a natural bijection between the affine cells in $\CF_{m/n}$  and the affine permutations in $\maSn$. The dimension of the cell $\Sigma_{\omega}$ labeled by the
 affine permutation $\omega$ is equal to $\sum_{i=1}^{n}\SP_{\omega}(i)$.
\end{theorem}

\begin{proof}
Let us introduce an auxiliary variable $z=\varepsilon^{1/n}$. We can identify the vector space $V=\C^n((\varepsilon))$ with the space
$\spann_\C\{1,z,\ldots,z^{n-1} \} ((z^n))\simeq \C((z))$ of Laurent power series in $z$ by sending the basis $\{e_1,\ldots,e_n\}$ to $\{1,z,\ldots,z^{n-1}\}.$
Note that under this identification, $\C^n[[\varepsilon]]$ is mapped to $\C[[z]]$. 
By construction, $N$ coincides with the multiplication operator by $z$ and hence $T=N^m$ coincides with the multiplication operator by $z^m$. Therefore $\CF_{m/n}$ consists of flags $\{M_0\supset\ldots\supset M_n\}$ of $\C[[z^n,z^m]]$-modules, such that $\dim_{\C} M_i/M_{i+1}=1$, $M_n=z^nM_0$ and $M_0\in\CG_n.$ Let us extend the notation $M_i$ to arbitrary $i\in\Z$ by setting $M_{i+kn}:=z^{kn}M_i.$ As a result, we get an infinite flag $\{\ldots\supset M_0\supset\ldots\supset M_n\supset\ldots\}$ of $\C[[z^n,z^m]]$-modules satisfying the same conditions as above and $M_{i+n}=z^nM_i.$ 

For $f(z)\in V=\C((z)),$ let $\Ord(f)$ denote the order of $f(z)$ in $z,$ i.e. the smallest degree of $z,$ such that the corresponding coefficient in $f(z)$ does not vanish. Given a subset $M\subset V$, define $$\Ord(M)=\{\Ord(f)\ :\ f\in M, f\neq 0\}.$$ We will need the following lemma, whose proof is standard and left to the reader:

\begin{lemma}
Let $L\subset M\subset \C((z))$ be two $\C[[z^m,z^n]]$-submodules in $z^{-N}\C[[z]]$ for some large $N\in\N.$ Then $\sharp\left(\Ord(M)\setminus\Ord(L)\right)=\dim_\C M/L.$
\end{lemma}


Given a flag $\{\ldots\supset M_0\supset\ldots\supset M_n\supset\ldots\}$ as above,
set $ \Ord(M_i)\setminus \Ord(M_{i+1}) = \{\w(i)\}.$
Note that one automatically gets $\Ord(M_i)=\{\w(i),\w(i+1),\ldots\},$ because $\bigcap\limits_i\Ord(M_i)=\emptyset.$ Recall the notation 
$$
\Inv(\w) : =
\{ (i,j) \in \N \times \N \mid 1\le i\le n,\ i < j,  \w(i) > \w(j) \}
$$
and 
$$
\Invall(\w) : =
\{ (i,j) \in \N \times \N \mid i < j,  \w(i) > \w(j) \}
$$
for the inversion sets of $\w.$ For each $i$ there exists a unique $f_i(z)\in M_i$  such that
\begin{equation}
\label{eqn:fi}
f_i=z^{\w(i)}+\sum_{(j,i)\in \Invall(\w)}\lambda^{\w(i)}_{\w(j)} z^{\w(j)}.
\end{equation}
Indeed, take any function $f\in M_i$ such that $\Ord(f)=\{\w(i)\}$ and use functions from $M_{i+1}$ to eliminate coefficients at $z^{\w(j)}$ for $j>i$ and $\w(j)>\w(i).$ The resulting function is unique up to a scalar, because otherwise $\dim_\C M_i/M_{i+1}$ would be at least $2.$ It follows that $f_{i+n}=z^nf_i.$ 

We claim that $\w$ is an affine permutation and, moreover,  $\omega\in\maSn$. Indeed, since $f_{i+n}=z^nf_i$ we get $\w(i+n)=\w(i)+n,$ and since $z^mf_i\in M_i$ we get that $\w(i)+m\in\{\w(i),\w(i+1),\ldots\},$ and, therefore, for any $j<i,$\ $\w(j)-\w(i)\neq m.$ Finally, we need to check the normalization condition $\sum\limits_{i=1}^n \w(i)=\frac{n(n+1)}{2},$ which follows form the normalization condition on $M_0\in\CG_n.$ Indeed, it is not hard to see that for all $L\in\CG_n$ the sum of elements of $\Ord(L)\setminus\Ord(t^nL)$ should be the same. In particular, for $L=\C[[z]]$ we have $\Ord(\C[[z]])\setminus \Ord(z^n\C[[z]])=\{0,1,\ldots,n-1\},$ and their
 sum is $\frac{n(n-1)}{2}.$ Therefore, since $\Ord(M_0)\setminus \Ord(z^nM_0)=\{\w(0),\ldots,\w(n-1)\},$ we get $\sum\limits_{i-0}^{n-1} \w(i)=\frac{n(n-1)}{2},$ which equivalent to the required condition. 

The above gives us 
a map $\nu:\CF_{m/n}\to \maSn.$ Let us prove that the fibers $\Sigma_{\w}:=\nu^{-1}(\w)$ of this map are affine cells and compute their dimensions. 
This is very similar to the computation in the finite case (see Theorem \ref{Theorem: cells of finite Springer fiber}). Let us set $f^{\w(i)}:=f_i.$ The expansions \eqref{eqn:fi} 
can be rewritten as 
$$
f^\alpha=z^\alpha+\sum_{(\alpha,\beta)\in \Invall(\w^{-1})}\lambda^{\alpha}_{\beta} z^\beta.
$$
Since $f_{i+n}=z^nf_i,$ one gets $\lambda^{\alpha+n}_{\beta+n}=\lambda^\alpha_\beta.$ Let us also extend the notation by setting $\lambda^{\alpha}_{\beta}=0$ whenever $(\alpha,\beta)\notin\Invall(\w^{-1})$, so that one can write 
$$
f^\alpha=z^\alpha+\sum_{\beta>\alpha} \lambda^\alpha_{\beta}z^{\beta}.
$$
Let us say that the coefficient $\lambda^{\alpha}_{\beta}$ is of height $\beta-\alpha.$ As before, let $\alpha=\w(i).$ The condition $z^m f^\alpha\in M_i$ implies the following relations on the coefficients. The function 
$$
z^mf^\alpha-f^{\alpha+m}=\sum_{\beta>\alpha} \lambda^\alpha_\beta z^{\beta+m}-\sum_{\beta>\alpha+m} \lambda^{\alpha+m}_{\beta}z^\beta
=\sum_{\beta>\alpha+m}(\lambda^\alpha_{\beta-m}-\lambda^{\alpha+m}_\beta)z^\beta
$$
should belong to $M_i$. Take $\beta>\alpha+m$ and let $j=\w^{-1}(\beta).$ If $j>i,$ then the term of degree $\beta$ can be eliminated by subtracting $f^\beta=f_j\in M_i$ with an appropriate coefficient.
If $j<i,$ then
$\wi(\beta-m) < \wi(\beta) < \wi(\alpha) < \wi(\alpha+m)$. Hence
$(\alpha,\beta-m)\in\Invall(\w^{-1})$ and $(\alpha+m,\beta)\in\Invall(\w^{-1}),$ so the coefficients $\lambda^\alpha_{\beta-m}$ and $\lambda^{\alpha+m}_\beta$ are both not forced to be zero, i.e. they are parameters on $\Sigma_\w$.
 The term $(\lambda^\alpha_{\beta-m}-\lambda^{\alpha+m}_\beta)z^\beta$ has to vanish automatically after we eliminated all lower order terms. 
As we eliminate terms of degree
$\gamma$ such that 
 $\alpha+m<\gamma<\beta,$ the coefficient at $z^\beta$ changes, but the added terms can only depend on coefficients of smaller height. More precisely, all additional terms are non-linear, and the total height of each monomial is always $\beta-\alpha-m.$ In the end, we get that $\lambda^\alpha_{\beta-m}-\lambda^{\alpha+m}_\beta$ should be equal to zero modulo the coefficients of smaller height.  

This means that for each $(\alpha,\beta)\in\Invall(\w^{-1})$ of height $\beta-\alpha>m$ there is an equation that allows one to express $\lambda^\alpha_{\beta-m}$ in terms of $\lambda^{\alpha+m}_\beta$ and higher order terms in coefficients with lower height. A priori, the linear parts of these equations can be dependent if 
for all $0\le q\le n$ one has $(\alpha+qm,\beta+qm)\in\Invall(\w^{-1}).$ However, since $m$ and $n$ are relatively prime, this would mean that $\omega^{-1}(\gamma)>\omega^{-1}(\gamma+\beta-\alpha)$ for all $\gamma\in \Z,$ which is impossible.
Therefore, one can resolve the relations on the coefficients with respect to $\lambda^\alpha_\beta$ such that $(\alpha,\beta)\in\Invall(\w^{-1})$ and $(\alpha,\beta+m)\in\Invall(\w^{-1}).$ So, the 
coordinates on $\Sigma_{\w}$ 
correspond to the inversions $(\alpha,\beta)\in\Inv(\w^{-1}),$ such that $(\alpha,\beta+m)\notin\Inv(\w^{-1})$.
Since $\lambda^\alpha_\beta=\lambda^{\alpha+n}_{\beta+n},$ one should count inversions in $\Inv(\w^{-1})$ only. 
It is not hard to see that such inversions are in bijection with inversions of height less
than $m.$ Indeed, the required map is $(\alpha,\beta)\mapsto (\alpha,\beta-km),$ where $k$ is the maximal integer such that $\beta-km>\alpha.$ 

Alternatively, one can also notice that the relations are in bijection with inversions of height greater
than $m.$ Indeed, the relation $\lambda^\alpha_\beta\equiv\lambda^{\alpha+m}_{\beta+m}$ (modulo lower height terms) corresponds to the inversion $(\alpha,\beta+m)\in\Invall(\w^{-1})$ of height greater than $m.$ Therefore, the dimension of $\Sigma_{\w}$
is the total number of inversions minus the number of inversion of height greater
than $m.$ Since there are no inversions of height $m,$ the dimension is equal to the number of inversions of height less than $m.$ 
Since $\sum_{i=1}^{n}\PS_{\omega^{-1}}(i)$ is exactly the total number of inversions of height less than $m,$ we conclude that 
$$
\dim(\Sigma_{\omega})=\sum_{i=1}^{n}\SP_{\omega}(i)=\sum_{i=1}^{n}\PS_{\omega^{-1}}(i)=\frac{(m-1)(n-1)}{2}-\dinv(\w^{-1}).
$$
For a more abstract proof see e.g. \cite{GKM} and \cite[Theorem 2.7, eq. 4.5]{H12}. 
\end{proof}

\begin{remark}
Similar reasoning shows that the Grassmannian version of the affine Springer fiber $\CG_{m/n}\subset\CG_n$ parametrizes appropriately normalized
 $\C[[z^n,z^m]]$-submodules
in $\C((z)).$ This affine Springer fiber was studied e.g. in \cite{GM11,GM12} under the name of Jacobi factor of the plane curve singularity $\{x^m=y^n\}$. The cells in it are parametrized by the subsets in $\ZZ$ which are invariant under addition of $m$ and $n$, and can be matched to the lattice points in $D_{n}^{m}$.
Note the lattice points in turn correspond to the minimal length left
coset representatives $\maSn \cap \aSn/\Sn$.
\end{remark}

\begin{corollary}
If the map $\PS$ is a bijection  then the Poincar\'e polynomial of $\CF_{m/n}$ is given by the following formula:
\begin{equation}
\label{Poinc}
\sum_{k=0}^{\infty}t^{k}\dim \Hom^{k}\left(\CF_{m/n}\right)=\sum_{f\in \PF_{m/n}}t^{2\sum_{i}f(i)}.
\end{equation}
\end{corollary}

\begin{proof}
Since the variety $\CF_{m/n}$ can be paved by the even-dimensional cells, it has no odd cohomology and $(2k)$-th Betti number is equal to 
the number of cells of complex dimension $k$. Therefore by Theorem \ref{Theorem: cells of affine Springer}:
$$
\sum_{k}t^{k}\dim \Hom^{k}\left(\CF_{m/n}\right)=\sum_{\w\in \maSn}t^{2\dim \Sigma_{\w}}=\sum_{\w\in \aSmn}t^{2\sum_i \PS_{\w}(i)}=
\sum_{f\in \PF_{m/n}}t^{2\sum_i f(i)}.
$$
\end{proof}

Equation $\eqref{Poinc}$ was conjectured in \cite[Sec. 10]{LS91} for all coprime $m$ and $n$.

\section{Some examples for $m \neq kn\pm 1.$}\label{Section: Examples}

In this section we discuss some examples for which $m \neq kn\pm 1$.

\begin{example}
\label{exl-pf-Anderson}
There are $81 = 3^4$
$3/5$-parking functions.
The $7= \frac{1}{5+3} \binom {5+3}{5} $ 
increasing parking functions
are $\pw{00000},\pw{00001}, \pw{00002}, \pw{00011}, \pw{00012}, \pw{00111},
\pw{00112}$.
Grouping them into the $\Snn{5}$-orbits $\{ f \circ \w \mid 
 \w \in \Snn{5} \}$ yields $81 = 1 + 5+5+10+20+10+30$.
There are 
7 vectors in $\Z^5 \cap V \cap D_5^3$. Their transposes are:
$$(0,0,0,0,0), (1,0,0,0,-1),  (0,1,0,0,-1), 
(1,0,0,-1,0), 
 (0,0,1,-1,0), (0,1,-1,0,0), (1,-1,0,1,-1)$$
The $30$ parking functions in the $\Snn{5}$-orbit
of $\pw{00112}$ correspond under the map $\An$
to the $30$ permutations in $\Snn{5} \cap \maSnn 53$
which are those in the support of the shuffle $14 \shuffle 25 \shuffle 3$
(that is to say, the intersection of the Sommers region with the orbit of
the identity permutation). 
\\
On the other hand, the parking function $\pw{00000} = \An(\wm)$
corresponds under $\An$ to the affine permutation
$\wm = [-3,0,3,6,9] \in \maSnn 53$.  Anything else in its right $\Snn{5}$-orbit
lies outside the Sommers region.
\end{example}

\omitt{  move example to earlier
\begin{example}
Consider again $(n,m) = (5,3)$ and $u = [0,3,6,2,4]$. 
Then $\Inv(u) = \{ (2,4), (3,4),$  $ (3,5), (3,6) \}$ and so
  $\SP_u = \pw{01200}$.  Note $u(3) - u(4) = 4 > m$ so this inversion
does not contribute to $\SP_u(3).$ 
\end{example}
} 

Despite the  fact many of the above theorems and constructions use $\aSmn$,
it is more 
uniform
to study the set $\{ u \Afund \mid u \in \maSn\}$ and $\SP$ than
$\{ \w \Afund \mid \w \in \aSmn\}$ and $\PS$.
One reason is that while the Sommers region can always be defined
for $\gcd(m,n)=1$, a hyperplane arrangement that is the correct
analogue of the Shi arrangement cannot. Consider the following example.

\begin{example}
\label{exl-noshi}.
In the case $(n,m) = (5,3)$ it is {\em impossible} to find a set of
hyperplanes that separate the alcoves 
$\{ \w \Afund \mid \w \in \aSmn\}$
and has $\frac{1}{5+3} \binom {5+3}{5} = 7$ dominant regions,
i.e. that there are exactly $7$ dominant regions with a unique $\w\Afund$
in each.
In other words, the notion of $\Regkn$ does not extend well when
$m \neq kn \pm 1$.

Indeed, 7 dominant regions corresponding to 3-restricted affine permutations
$\w \in \maSnn 53 \cap \aSnn{5}/\Snn{5}$
are shown in bold in Figure \ref{s 5 3 dominant}.
Each permutation drawn corresponds to the dominant alcove $\wi \Afund$.
Hence the hyperplanes crossed (by the pictured $\w$) correspond exactly to $\Inv(\w)$.
The hyperplanes $H_{4,6}^0,H_{5,6}^0$ and $H_{5,7}^0$
separate $[02346]$ from other  3-restricted permutations.
To separate $[-21457]$ from $[-11258]$, one must 
add either $H_{3,6}^0$ or $H_{5,8}^0$ to the arrangement,
but this would leave either
of the non-3-restricted permutations
 $[-22456]$ or $[01248]$ in a region with no 3-restricted permutations.
Therefore any extension of the classical braid arrangement for $S_5$ would either have a region with two 3-restricted permutations
or a region with none of them.

Note there are more hyperplanes ($H_{4,7}^0, H_{5,11}^0$) that $\wm = [-3,0,3,6,9]$
has crossed that we did not draw on the picture. 
\end{example}

\begin{figure}
\begin{tikzpicture}
\draw (10,0) node {$\mathbf{[12345]}$};
\draw (8,-1)..controls (9,1) and (11,1)..(12,-1);
\draw (10,1) node {$\mathbf{[02346]}$};
\draw (7,-1)--(13,5);
\draw (13,-1)--(7,5);
\draw [dashed] (4,-1)--(10,5);
\draw [dashed] (16,-1)--(10,5);
\draw [dashed] (8,7)--(10,5);   
\draw [dashed] (12,7)--(10,5);  
\draw (10,7) node {$\mathbf{[-30369]}$};  
\draw (8,1.5) node {$\mathbf{[-12356]}$};
\draw (6,3) node {$\scriptstyle [-22456]$}; 
\draw (12,1.5) node {$\mathbf{[01347]}$};
\draw (14,3) node {$\scriptstyle [01248]$}; 
\draw (10,3) node {$\scriptstyle [-11357]$}; 
\draw (11.5,4.5) node {$\mathbf{[-11258]}$};
\draw (8.5,4.5) node {$\mathbf{[-21457]}$};
\omitt{
\draw (12,-1.2) node {$H_{1,5}^1$};
\draw (13,-1.2) node {$H_{2,5}^1$};
\draw (16,-1.2) node {$H_{3,5}^1$};
\draw (7,-1.2) node {$H_{1,4}^1$};
\draw (4,-1.2) node {$H_{1,3}^1$};
} 
\draw (12,-1.3) node {$H_{5,6}^0$};
\draw (13,-1.3) node {$H_{5,7}^0$};
\draw (16,-1.3) node {$H_{5,8}^0$};
\draw (7,-1.3) node {$H_{4,6}^0$};
\draw (4,-1.3) node {$H_{3,6}^0$};
\end{tikzpicture}
\caption{7 permutations  $\w \in \maSnn{5}{3}$ labeling $\w^{-1} \Afund$
 in the dominant cone}
\label{s 5 3 dominant}
\end{figure}

\begin{example}
\label{example: 2 5}
We list all affine permutations in $\aSmnn{5}{2}$ together with their images under the maps $\An$ and $\PS$ in Figure \ref{table:t25}.
 Here $\w$ is a $2$-stable affine permutation (that is, $\w(i+2)>\w(i)$), and $\w^{-1}$ is 2-restricted .
Note that for $m=2$ one has $\An_{\w}(k)=\w^{-1}(k)-M_{\w}\ \modd\ 2$, where, as above, $M_{\w}=\min\{k : \omega(k)>0\}$.
 
The combinatorial Hilbert series has a form:
$$H_{2/5}(q,t)=5+4(q+t)+(q^2+qt+t^2).$$
In particular, it is symmetric in $q$ and $t$ and thus answers a question posed in \cite[Section 5.4]{Ar11}.

The special vertex of $D_{2/5}$ corresponding to the fundamental alcove by Lemma \ref{isometry}
is described by the affine permutation $\w_2=[-1,1,3,5,7]$. 
\end{example}

\begin{figure}[ht]
\begin{tabular}{|c|c|c|c|c|c|}
\hline
$\omega$ & $\omega^{-1}$ & $\An$ & $\area$ & $\PS$ & $\dinv$ \\ \hline
[-1,2,5,3,6]	& [0,2,4,6,3]	  & $\pw{0,0,0,0,1}$	& 1& 	$\pw{0,0,1,1,0	}$& 0\\ \hline
[0,2,3,4,6]	& [0,2,3,4,6]	  & $\pw{0,0,1,0,0}$	& 1&	$\pw{0,0,0,0,1	}$&1\\ \hline
[0,2,4,3,6]	& [0,2,4,3,6]	  & $\pw{0,0,0,1,0}$	& 1&	$\pw{0,0,1,0,1	}$&0\\ \hline
[0,3,1,4,7]	& [3,0,2,4,6]	  & $\pw{1,0,0,0,0}$	& 1&	$\pw{1,0,0,0,1	}$&0\\ \hline
[0,3,2,4,6]	& [0,3,2,4,6]	   & $\pw{0,1,0,0,0}$	& 1&	$\pw{0,1,0,0,1}$	&0\\ \hline
[2,0,3,6,4]	& [-1,1,3,5,7] & $\pw{0,0,0,0,0}$ &2&	$\pw{0,0,0,1,1}$&	0\\ \hline
[1,2,3,4,5]	& [1,2,3,4,5]	& $\pw{0,1,0,1,0}$	& 0&	$\pw{0,0,0,0,0	}$&2\\ \hline
[1,2,3,5,4]	& [1,2,3,5,4]	& $\pw{0,1,0,0,1}$	& 0& $\pw{0,0,0,1,0}$&	1\\ \hline
[1,2,4,3,5]	& [1,2,4,3,5]	& $\pw{0,1,1,0,0}$	& 0&	$\pw{0,0,1,0,0}$&	1\\ \hline
[1,3,2,4,5]	& [1,3,2,4,5]	& $\pw{0,0,1,1,0}$	& 0&	$\pw{0,1,0,0,0}$&	1\\ \hline
[1,3,2,5,4]	& [1,3,2,5,4]	& $\pw{0,0,1,0,1}$	& 0&	$\pw{0,1,0,1,0}$&	0\\ \hline
[1,4,2,5,3]	& [1,3,5,2,4]	& $\pw{0,0,0,1,1}$	& 0&	$\pw{0,1,1,0,0}$&	0\\ \hline
[2,1,3,4,5]	& [2,1,3,4,5]	& $\pw{1,0,0,1,0}$	& 0&	$\pw{1,0,0,0,0}$&	1\\ \hline
[2,1,3,5,4]	& [2,1,3,5,4]	& $\pw{1,0,0,0,1}$	& 0&	$\pw{1,0,0,1,0}$&	0\\ \hline
[2,1,4,3,5]	& [2,1,4,3,5]	& $\pw{1,0,1,0,0}$	& 0&	$\pw{1,0,1,0,0}$&	0\\ \hline
[3,1,4,2,5]	& [2,4,1,3,5]	& $\pw{1,1,0,0,0}$	& 0&	$\pw{1,1,0,0,0}$&	0\\ \hline
\end{tabular}
\caption{Affine permutations in $\aSmnn{5}{2}$, their inverses in $\maSnn{5}{2}$; maps $\An$ and $\PS$ to $\PF_{2/5}$;\newline 
$\area$ and $\dinv$ statistics}
\label{table:t25}
\end{figure}

\end{document}